\documentclass{article}

\usepackage[english]{babel}

\usepackage{amssymb} 
\usepackage{graphicx} 
\usepackage{amsthm} 
\usepackage{color}
\usepackage{amsmath}
\usepackage{mathrsfs} 
\usepackage{mathtools} 
\usepackage{dsfont}
\usepackage[parfill]{parskip}
\usepackage[a4paper, total={6in, 8in}]{geometry}
\usepackage{bbm}
\usepackage{hyperref}
\usepackage{cleveref} 
\usepackage{enumerate}
\usepackage[normalem]{ulem} 
\usepackage{todonotes}

\usepackage{caption}
\usepackage[labelformat=simple]{subcaption}

\usepackage{tikz}
\usetikzlibrary{decorations.pathmorphing,decorations.pathreplacing,calc,arrows.meta}

\tikzset{
  zline/.style   ={-{Stealth[length=2mm]}, gray!55},
  vtx/.style     ={circle, fill=black, inner sep=1.5pt},
  ivtx/.style    ={circle, fill=black!70, inner sep=1.1pt},
  outerarch/.style={very thick, blue!65!black},
  innerarch/.style={very thick, teal!75!black},
  forced/.style  ={very thick, red!75!black},
  excluded/.style={very thick, red!70!black, dashed},
  escape/.style  ={very thick, orange!85!black},
  reach/.style   ={thick, violet!70!black},
  cut/.style     ={gray!65, dashed},
  pathconn/.style={thick, green!45!black, decorate,
                   decoration={snake, amplitude=.5mm, segment length=2.2mm}},
  brc/.style     ={decorate, decoration={brace, amplitude=4pt, mirror}},
}

\usepackage{setspace}

\usepackage[
backend=biber,
style=numeric,
maxbibnames=99
]{biblatex}
\newcommand{\eps}{\varepsilon}

\newcommand{\R}{\mathbb R}

\newcommand{\E}{\mathbb E}
\renewcommand{\P}{\mathbb P}

\newcommand{\cL}{\mathcal L}
\newcommand{\cR}{\mathcal R}

\newcommand{\cG}{\mathcal{G}}

\renewcommand{\b}{\beta}

\renewcommand{\phi}{\varphi}

\newcommand{\euler}{{\mathrm e}}

\newcommand{\Z}{\mathbb Z}
\newcommand{\N}{\mathbb{N}}
\renewcommand{\P}{\mathbb P}

\newcommand{\X}{\mathcal X}

\newcommand{\de}{\mathrm{d}} 

\newcommand{\ov}{\overline}

\newcommand{\0}{\mathbf{0}}
\newcommand{\x}{\mathbf{x}}
\newcommand{\y}{\mathbf{y}}

\newcommand{\abs}[1]{\left| #1 \right|}

\newcommand{\set}[1]{\left\{ #1 \right\}}

\newcommand{\1}{\mathbf{1}}
\newcommand{\st}{\star}

\let\oldsqrt\sqrt
\def\sqrt{\mathpalette\DHLhksqrt} \def\DHLhksqrt#1#2{%
\setbox0=\hbox{$#1\oldsqrt{#2}$}\dimen0=\ht0
\advance\dimen0-0.2\ht0
\setbox2=\hbox{\vrule height\ht0 depth -\dimen0}%
{\box0\lower0.4pt\box2}}

\theoremstyle{plain} 
\newtheorem{theorem}{Theorem}
\newtheorem{prop}{Proposition}[section]
\newtheorem{lemma}{Lemma}[section]
\newtheorem{corollary}{Corollary}[section]

\Crefname{theorem}{Theorem}{Theorems}
\Crefname{prop}{Proposition}{Propositions}
\Crefname{corollary}{Corollary}{Corollaries}
\crefname{theorem}{Theorem}{Theorems}
\crefname{prop}{Proposition}{Propositions}
\crefname{corollary}{Corollary}{Corollaries}
\crefname{lemma}{Lemma}{Lemmas}
\crefname{mydef}{Definition}{Definitions}
\crefname{remark}{Remark}{Remarks}

\theoremstyle{definition} 
\newtheorem{mydef}{Definition}[section]

\Crefname{mydef}{Definition}{Definitions}

\theoremstyle{remark} 
\newtheorem{remark}{Remark}[section]

\Crefname{remark}{Remark}{Remarks}

\title{Rainbow percolation}
\author{Peter Gracar\thanks{School of Mathematics, University of Leeds,
United Kingdom. ORCID:
\href{https://orcid.org/0000-0001-8340-8340}{0000-0001-8340-8340}.}
\and Benjamin Lees\thanks{School of Mathematics, University of Leeds,
United Kingdom. ORCID:
\href{https://orcid.org/0000-0003-2657-5358}{0000-0003-2657-5358}}}
\date{August 19, 2026}

\begin{document}
\maketitle

\begin{abstract}
We consider the weight-dependent random connection model on a Poisson point
process of intensity $\lambda$ on $\R\times(0,1)$ in which the vertices
$(x,t)$ and $(y,s)$ are joined precisely when $(t\vee s)|x-y|\le\beta$.
Points at distance $d$ are joined with probability $\min(1,\beta/d)^2$, the
critical decay of one-dimensional long-range percolation, and edges sharing
a vertex are dependent through the common mark. We prove that the model has
a genuine phase transition: for $\lambda\beta<1$ almost surely all
connected components are finite, while for $\lambda\beta\ge31$ an infinite
component exists, so at intensity one the critical value satisfies
$\beta_c\in[1,31]$; a numerical study included as an appendix places it
near $2$.  By kernel and profile comparisons the supercritical bound
extends to the age-dependent random connection model on the line, which
with indicator profile has a non-degenerate phase transition at every
value of its parameter, closing a case of the one-dimensional phase
diagram left open in earlier work. The lower bound is proved by
disconnecting nested pairs of long edges (``rainbows'') with cut-point
certificates, an argument developed first in a discrete skeleton of the
model with the vertices pinned to $\Z$. The skeleton is of independent
interest: it has no supercritical phase at all, jumping from total
fragmentation to trivial connectivity even though almost surely
infinitely many edges cross every fixed site. The supercritical argument
is a Peierls argument on the binary tiling of the hyperbolic
half-plane. 
\end{abstract}

\section{Introduction}\label{sec:intro}

We study the weight-dependent random connection model
\cite{Bringmann2019,GracarGrauerLuechtrathMoerters2019,
GracarHeydenreichMoenchMoerters2022,Jorritsma2023a} on the real line at the critical decay
of its weak kernel. The vertex set is a Poisson point process of intensity
$\lambda$ on $\R\times(0,1)$, and, taking the indicator profile, two
vertices $(x,t)$ and $(y,s)$ are joined by an edge precisely when
\[
   (t\vee s)\,|x-y|\le\beta,
\]
where $\beta>0$ is a fixed parameter. A vertex $(x,t)$ thus behaves like a
ball of radius $\beta/t$ centred at $x$; since $\beta/U$ with $U$ uniform
on $(0,1)$ is exactly Pareto with scale $\beta$, the model describes
heavy-tailed balls on the line, joined by the symmetric rule that each
centre must lie within the reach of the other. Averaging over the marks,
two points at distance $d>\beta$ are joined with probability $\beta^2/d^2$.
This $\min$-rule with Pareto radii is the connection rule introduced by Yukich
\cite{Yukich2006} on the integer lattice; the model studied here is its
continuum counterpart on the line, at the boundary of the parameter range he
considers.

Two features distinguish this graph from long-range percolation with independent edges, and together place it outside of where classical one-dimensional theory applies. First, the connection probability
decays like the inverse square of the distance, which is the critical decay
of one-dimensional long-range percolation: for independent edges with this
decay the existence of an infinite component depends on the multiplicative
constant and not merely on the exponent
\cite{Schulman1983,NewmanSchulman1986,AizenmanNewman1986,
DuminilCopinGarbanTassion2024}. Because the
radii here have tail exponent exactly $1$, varying $\beta$ or $\lambda$
moves only the constant and never the exponent, so the model cannot be
compared to an off-critical regime. In particular, no stochastic domination
argument decides percolation : any change of the tail exponent
produces a model whose threshold is degenerate (\Cref{prop:gamma}) . Second, edges sharing a vertex are strongly
dependent through the common mark, so the graph is \emph{not} a long-range
percolation model with independent edges, and the classical results do not
apply directly.

Our main results show that the model nevertheless undergoes a genuine phase
transition. For $\lambda\beta<1$ almost surely every connected component is
finite (\Cref{thm:continuum}), while for $\lambda\beta\ge31$ an infinite
component exists almost surely (\Cref{thm:super}); at intensity one the
critical value therefore satisfies $\beta_c\in[1,31]$. A numerical study,
reported in \Cref{sec:numerics}, places the critical value near
$\beta_c\approx2$.  The critical exponent is the only one with this
behaviour: under the rule $(t\vee s)^{\gamma}\abs{x-y}\le\beta$ the radii
have tail exponent $1/\gamma$, and \Cref{prop:gamma} shows that
$\beta_c=\infty$ for every $\gamma<1$ and $\beta_c=0$ for every
$\gamma>1$. Our third main result is obtained from the comparison in its
proof: the age-dependent random connection model on the line with
indicator profile has a non-degenerate phase transition at every value
of its parameter (\Cref{thm:adrcm}). 

The lower bound rests on the analysis of a simplified model that we believe
to be of independent interest: the discrete skeleton of the continuum
model, obtained by pinning the vertices to $\Z$. Every vertex $i\in\Z$
carries an independent radius $R_i$ with the Pareto law
\[
   \P(R_i>x)=\frac{x_m}{x},\qquad x\ge x_m,
\]
where $x_m\in(0,1)$ plays the role of $\beta$, and two distinct vertices
$i,j$ are joined by an edge precisely when each lies within the reach of
the other,
\[
   \{i,j\}\in E \quad\Longleftrightarrow\quad |i-j|\le R_i\wedge R_j
\]
(the discretisation is set out in \Cref{sec:setup} and
\Cref{sec:simplified}). A fixed pair at distance $d$ is joined with
probability $x_m^2/d^2$, so the skeleton retains both features of
the continuum model, the critical decay and the dependence through common
radii, while stripping away the spatial randomness. Our main result for
this lattice model is that, in contrast to the continuum, it is
totally fragmented for \emph{every} admissible parameter: for $x_m<1$
almost surely every connected component is finite (\Cref{thm:main}). For
$x_m\ge1$ every radius is at least $1$, so all nearest-neighbour edges are
present and the graph is connected; within this family the transition
therefore occurs at $x_m=1$, but it is a degenerate one. The model
jumps from total fragmentation to trivial connectivity, with no supercritical
percolation phase in between. The
absence of an infinite component is not due to a shortage of long edges, as
almost surely infinitely many edges cross any fixed site
(\Cref{prop:edges-io}, \Cref{rem:edges-io-Z}), so the graph contains finite
components of arbitrarily large euclidean diameter (\Cref{cor:diameter}). The source of the
divergence between the two models is that pinning the vertices to $\Z$ caps
the number of points per connection range at every scale, whereas in the
continuum the parameter $\lambda\beta$ pushes that density arbitrarily
high, which we make precise in \Cref{rem:superlit}.

\begin{remark}[Comparison with classical long-range percolation]\label{rem:lrp}
For independent bonds on $\Z$ with $p_d\sim\beta_{\mathrm{eff}}/d^2$, where
$\beta_{\mathrm{eff}}:=\lim_{d\to\infty}d^2p_d$, there is no infinite
component if $\beta_{\mathrm{eff}}\le1$ and $p_1<1$
\cite{AizenmanNewman1986}, and there is one if $\beta_{\mathrm{eff}}>1$ and
$p_1$ is close enough to $1$ \cite{NewmanSchulman1986}. The constant
$\beta_{\mathrm{eff}}$ is the coefficient of the logarithmic divergence in
the expected number of edges crossing a fixed point: since $n$ pairs lie at
distance $n$,
$\sum_{i\le0<j}p_{j-i}=\beta_{\mathrm{eff}}\sum_{n\ge1}n^{-1}$.

\begin{enumerate}[(i)]
	\item In the simplified model of \Cref{sec:simplified},
$\P(\{i,j\}\in E)=x_m^2/|i-j|^2$ for every $|i-j|\ge1$, so
$\beta_{\mathrm{eff}}=p_1=x_m^2$. Both are below $1$ for every admissible
$x_m$, so for independent edges \Cref{thm:main} would follow from
\cite{AizenmanNewman1986}. What has to be ruled out is that the dependence
through the shared radii changes this. What makes this question unclear, even intuitively, is that it acts in the direction that
would favour percolation: $\{i,j\}\in E$ is increasing in $(R_i,R_j)$, so
the edges are positively associated.

\item The degeneracy of the transition on $\Z$ can be read off the same two
quantities. Classically $\beta_{\mathrm{eff}}$ and $p_1$ are separate
parameters and percolation needs both to be large; here they are equal, and
raising $x_m$ to the point where $\beta_{\mathrm{eff}}>1$ also makes every
radius at least $1$, hence every nearest-neighbour bond present. The window
$\beta_{\mathrm{eff}}\in(1,\infty)$ has no counterpart in this family,
which is the obstruction of \Cref{rem:superlit} seen from the side of the
connection probabilities rather than the density.

\item In the continuum, for intensity $\lambda$ and connection probability
$\beta^2/d^2$,
\[
   \lambda^2\int_{-\infty}^{0}\!\int_{0}^{\infty}
   \frac{\beta^2}{(y-x)^2}\,\de y\,\de x
   =\lambda^2\beta^2\int_0^\infty\frac{\de u}{u},
\]
so $\beta_{\mathrm{eff}}=\lambda^2\beta^2$. \Cref{thm:continuum} then reads
$\beta_{\mathrm{eff}}<1$ and \Cref{thm:super} reads
$\beta_{\mathrm{eff}}\ge961$: the lower bound agrees with the classical
constant, now for dependent edges. The numerical study of
\Cref{sec:numerics} places the critical value near
$\beta_{\mathrm{eff}}\approx4$, so the agreement concerns the bound and
not the critical value.

\item The subcritical mechanism is the classical one, read in the model
restricted to $[0,n]$: on the full line every position is crossed
(\Cref{rem:edges-io-Z}), for the independent model by the same computation. For independent bonds
the positions of $[0,n]$ crossed by no edge of the
restriction number $\asymp n^{1-\beta_{\mathrm{eff}}}$; our certificates
give $\asymp n^{1-x_m}$, against $n^{1-x_m^2}$ for the independent-pairs
model with the same marginals; since $x_m>x_m^2$ on $(0,1)$ the dependence
costs an exponent, but both are positive for $x_m<1$
(\Cref{rem:cutpoints}).

\item Two points are left open. The classical statement includes
$\beta_{\mathrm{eff}}=1$, whereas \Cref{thm:continuum} requires
$\lambda\beta<1$, so the case $\lambda\beta=1$ is undecided. At the
critical exponent the percolation density jumps at the threshold
\cite{AizenmanNewman1986,AizenmanChayesChayesNewman1988,
DuminilCopinGarbanTassion2024}; we do not know whether it does so at
$\beta_c$ here.
The simulations of \Cref{sec:numerics} point to a jump, of height close
to $0.9$.
\end{enumerate}
\end{remark}

For independent edges, joined at distance $d$ with probability decaying as $d^{-s}$, much more is known. At $s=2$ on $\Z$, Imbrie and Newman
\cite{ImbrieNewman1988} exhibited an intermediate phase in which the truncated
two-point function decays with a continuously varying power, and Aizenman,
Chayes, Chayes and Newman \cite{AizenmanChayesChayesNewman1988} proved the
matching discontinuity of the magnetisation for the associated Ising and Potts
models. For $s\in(1,2)$ the chemical distance between $x$ and $y$ grows like
$(\log\abs{x-y})^{\Delta+o(1)}$ with $\Delta=\log2/\log(2/s)$
\cite{Biskup2004}, as does the diameter of a box \cite{Biskup2011}, and Biskup
and Lin \cite{BiskupLin2019} removed the $o(1)$. Over the same range Berger
\cite{Berger2002} proved that there is no infinite cluster at criticality, and
that the walk on the infinite cluster is transient for $s\in(1,2)$ and
recurrent for $s\ge2$. The distances at the critical decay $s=2$ were settled
last, and they are polynomial rather than polylogarithmic, of order
$\abs{x-y}^{\theta}$ for an exponent $\theta\in(0,1)$ that again depends on the
multiplicative constant \cite{DingSly2013}. All of this has a counterpart in
higher dimensions, with $1$ and $2$ replaced by the dimension and twice the
dimension; see \cite{Baeumler2023} for the critical decay and
\cite{BaeumlerBerger2024} for lower bounds on the critical exponents.

The model studied here belongs instead to the inhomogeneous family, in which
the connection probability is modulated by vertex weights and edges through a
common vertex are therefore dependent; see Deijfen, van der Hofstad and
Hooghiemstra \cite{DeijfenHofstadHooghiemstra2013} for scale-free percolation
and \cite{GracarGrauerLuechtrathMoerters2019,GracarLuechtrathMoerters2021} for
the weight-dependent random connection model. The $\min$-rule is due to Yukich
\cite{Yukich2006}, who placed radii $\delta U_z^{-p}$ at the sites of the
integer lattice, with the $U_z$ independent and uniform on $[0,1]$, and joined
two sites when each lies inside the other's radius. Those radii are Pareto with
shape $1/p$ and scale $\delta$, so the simplified model of
\Cref{sec:simplified} is his $G_{p,\delta}$ on $\Z$ with $p=1$ and
$\delta=x_m$. On the line he requires $p>1$, where the graph is scale-free and
ultra-small, with graph distances of order $\log\log\abs{x-y}$; shape $1$ is
the excluded endpoint $p=1$, and \Cref{thm:main} says that there the graph has
no infinite component at all once $\delta<1$. The continuum model of
\Cref{sec:setup} is the Poisson version on $\R$ anticipated in his Remark~2,
and \Cref{thm:super} shows that it does have an infinite component when
$\lambda\beta$ is large.  \Cref{prop:gamma} covers his remaining
regimes: in the continuum the threshold is zero for $p>1$ and infinite
for $p<1$. 

The subcritical proof runs as follows. A \emph{rainbow} is a nested pair of long edges
$\{a,b\}$, $\{c,d\}$ with $a<c<d<b$ such that neither $\{a,c\}$ nor
$\{d,b\}$ is an edge and the two overhangs $\ell=c-a$, $r=b-d$ differ by at
most the inner span $m=d-c$; rainbows straddle the origin at infinitely
many independent scales, with per-scale probability bounded below
(\Cref{prop:rainbows-io}). Whether the two arches of a
rainbow communicate is decided by the radii inside the outer span
(\Cref{lem:window}), and on a containment event of uniformly positive
probability (\Cref{lem:contain}) the question reduces to two one-sided
\emph{gap} problems, one in each overhang $(a,c)$ and $(d,b)$. Gaps are
produced by cut-point certificates in the spirit of Newman and Schulman
\cite{NewmanSchulman1986}: intersections of caps on distinct independent
radii that force all edges near a given position to be short. Because such a
certificate is a product over vertices, all probabilities involved are exact
products, and a second-moment argument yields a gap with probability bounded
below uniformly in the scale (\Cref{prop:cutpoints}, \Cref{lem:cutmarked}).
A confinement lemma (\Cref{lem:confine}) then upgrades disconnection of the
arches: between the two cuts of a rainbow, the component of
\emph{every} vertex is trapped. Finally, the certificates at different
scales are sufficiently close to independent for the Kochen--Stone lemma
\cite{KochenStone1964} and a Kolmogorov zero--one law to apply
(\Cref{prop:scale}, \Cref{prop:scalesio}), so almost surely every vertex is
confined at infinitely many scales, which gives \Cref{thm:main}.

The supercritical side is proved by a renormalisation of a different kind.
The mark space $(0,1)$ is split into dyadic bands, whose restrictions are
independent Poisson processes and, by an exact scale invariance, dilated
copies of one another: each band is a dense but subcritical Gilbert-type
graph whose connection range is comparable to its scale. By the
$\min$-rule, the connection threshold between a band-$k$ point and any
point of larger radius exceeds $\beta2^k$, so a chain of band-$k$ points
with gaps below $\beta2^k$ absorbs, deterministically, every point of
every higher band inside its span. The only randomness left is whether
dyadic windows are locally dense, and these events are independent with
failure probability at most $16e^{-\lambda\beta/4}$, uniformly over scales
and locations. The incidence structure of the dyadic windows across all
scales is the binary tiling of the hyperbolic half-plane, truncated at a
floor, and a Peierls contour argument for independent site percolation on
this tiling, resting on a boundary-connectivity lemma of Tim\'ar
\cite{Timar2013}, produces an infinite cluster (\Cref{sec:super}).

The paper is organised as follows. \Cref{sec:setup} introduces the
continuum model,  proves the degeneracy of the threshold at every
other tail exponent (\Cref{prop:gamma}) and the phase transition of the
age-dependent random connection model (\Cref{thm:adrcm}),  shows that
every point of the line is crossed by infinitely many edges and that
interleaving edges merge, and defines rainbows (\Cref{sec:rainbows}). \Cref{sec:simplified} introduces the simplified model
and carries out the programme above, in
\Cref{sec:formation,sec:reduction,sec:cutpoints,sec:main}; \Cref{sec:continuum}
then transfers the main theorem to the continuum model. \Cref{sec:super} proves the supercritical
bound \Cref{thm:super}; it can be read independently of
\Cref{sec:simplified}. \Cref{sec:numerics} reports a numerical study
locating the critical value between the two bounds.

\section*{Notation}

In this section we present a glossary of some notation we use throughout the paper.

\begin{itemize}
    \item $R_v$: radius associated to $v\in \Z$. $R_v$ is Pareto distributed with shape 1 and scale $x_m\in(0,1)$, as defined in \Cref{sec:simplified}; its continuum counterpart $R_x$, Pareto with scale $\beta$, is defined in \Cref{sec:rainbows}
    \item $R^\sharp_u$: the modified radii of the marked model given in \Cref{def:marked}
    \item $E$: $\{\{i,j\}\,:\, |i-j|\leq R_i\wedge R_j\}$, the edge set, as defined in \Cref{sec:simplified}; the continuum connection rule is given in \Cref{sec:setup} and recast through radii in \Cref{sec:rainbows}
    \item Rainbow: a pair of edges $\{a,b\},\{c,d\}\in E$ with $a<c<d<b$ such that $\{a,c\},\{d,b\}\notin E$, and such that its overhangs are balanced, as defined in \Cref{def:rainbow}
    \item $\ell,m,r\in \Z$: $c-a, d-c, b-d$, the left, middle, and right lengths of a rainbow, respectively, as defined in \Cref{def:rainbow}
    \item $\{\text{gap at }i\}$: $\{\not\exists j\leq i\text{ s.t. }\exists k>i\text{ with } |k-j|\leq R_j\wedge R_k\}$, the event \eqref{eq:gap}, defined after \Cref{thm:main}
    \item $\{\text{gap}^\sharp\text{ at }i\}$: $\{\not\exists 0\leq j\leq i<k\leq n+m\text{ with } |k-j|\leq R^\sharp_j\wedge R^\sharp_k\}$, the gap event corresponding to the marked model, defined in \Cref{def:marked}
    \item $B$: containment event $\{R_v<\min(v-a,b-v)\text{ for all }v\in(c,d)\}$, defined after \Cref{lem:window}; its continuum version is defined in \Cref{lem:containcont}
    \item $\mathcal D^L$: left overhang disconnection event, defined after \Cref{lem:window}, $$\{a\not\sim[c,d]\ \text{by a path whose intermediate vertices lie in }(a,c)\}$$
    \item $\mathcal D^R$: right overhang disconnection event, defined after \Cref{lem:window}, $$\{[c,d]\not\sim b\ \text{by a path whose intermediate vertices lie in }(d,b)\}$$
    \item $\mathcal N$: no-cross-edge event $\{\{u,w\}\notin E\ \text{for all }u\in(a,c),\ w\in(d,b)\}$, defined after \Cref{lem:window}
    \item $\mathcal C_i$: cut-point certificate event $\{R_j<i+1-j\text{ for all }0\leq j\leq i\}\cap\{R_{i+1}<1\}$, defined in \Cref{sec:cutpoints}; the continuum counterpart of $\mathcal C_i$ and $\mathcal C^\sharp_i$ is the certificate $\mathcal C^W_\tau$ of \eqref{eq:cutcont}
    \item $\mathcal C^\sharp_i$: $\{R^\sharp_j<i+1-j\ \text{for }1\le j\le i\}
   \cap\{R^\sharp_{i+1}<1\}
   \cap\{R^\sharp_k<k\ \text{for }i+2\le k\le n-1\}$, the cut-point certificate event in the marked model, defined in \Cref{lem:cutmarked}
\end{itemize}

\section{Setup}\label{sec:setup}

Let $\mathcal{X}$ be a Poisson point process of intensity $\lambda>0$ (outside the statements of  \Cref{thm:continuum,thm:super,prop:gamma,thm:adrcm}  we take $\lambda=1$, which by the scaling relation recalled in \Cref{sec:continuum} loses no generality) on $\mathbb{R}\times(0,1)$.
Consider the weight-dependent random connection graph on $\mathcal{X}$ with the weak kernel  $g_{\text{weak}}(t,s)=t\vee s$  and profile function $\rho(x)=\1_{[0,1]}(x)$. More precisely, a vertex $(x,t)\in\mathbb{R}\times(0,1)$ is connected via an edge to a vertex $(y,s)\in\mathbb{R}\times(0,1)$ with probability
\[
   \rho\big(\beta^{-1} g_{\text{weak}}(t,s)|x-y|\big)=\1_{[0,\beta]}((t\vee s)|x-y|).
\]
 Outside of \Cref{prop:gamma,thm:adrcm}, where other kernels replace
$g_{\text{weak}}$ in this formula, the weak kernel is the only kernel we
work with; similarly, the profile is the indicator throughout the paper unless specified otherwise. 
Note that the connection probability is monotonically decreasing in $|x-y|$, that is, vertices have a higher probability of being connected when they are close. Similarly, the connection probability is monotone decreasing in the maximum of $t$ and $s$ -- if one interprets the mark $t$ (resp. $s$) as the inverse weight of the vertex, this means the probability of an edge existing is monotonically increasing in the smaller of the two weights.

Write $\P_\beta$ for the law of the resulting graph, and let
\[
    \beta_c=\inf\{\beta:\P_\beta(\text{there exists an infinite connected component})>0\}.
\]
We are interested in determining whether $\beta_c=0$, $\beta_c=\infty$ or
$\beta_c\in(0,\infty)$; our first two main results show that the third
alternative holds, locating $\beta_c\in[1,31]$ at intensity $1$ ,
and by \Cref{prop:gamma} each of the two degenerate alternatives occurs
at every other tail exponent .

Our first main result bounds $\beta_c$ from below. Its proof runs through the
simplified model of \Cref{sec:simplified} and is completed in
\Cref{sec:continuum}.

\begin{theorem}\label{thm:continuum}
If $\lambda\beta<1$, then almost surely every connected component of the
weight-dependent random connection model  with weak kernel
$g_{\text{weak}}$ and profile function $\rho=\1_{[0,1]}$  is finite. In
particular, at
intensity $\lambda=1$ the model has no infinite component for any
$\beta<1$, i.e.\ $\beta_c\ge1$.
\end{theorem}

The complementary upper bound is proved in \Cref{sec:super}, by a
renormalisation that decomposes the mark space into independent dyadic
bands; its proof does not use the simplified model.

\begin{theorem}\label{thm:super}
If $\lambda\beta\ge31$, then almost surely the weight-dependent random
connection model  with weak kernel $g_{\text{weak}}$ and profile
function $\rho=\1_{[0,1]}$  contains an infinite connected component. In
particular,
at intensity $\lambda=1$,
\[
   1\ \le\ \beta_c\ \le\ 31 .
\]
\end{theorem}

Together the two theorems show that the continuum model has a non-degenerate
phase transition. The simplified lattice model of \Cref{sec:simplified}
does not: there, \emph{every} admissible parameter is subcritical
(\Cref{thm:main}). This contrast between the two models runs through the
paper and is discussed where it is sharpest, in the criticality remark of
\Cref{sec:main} and in \Cref{rem:superlit}.

The location of the model within the trichotomy depends on the tail
exponent of the radii. For $\gamma>0$ replace the weak kernel by
$(t\vee s)^{\gamma}$, so that the connection rule becomes
\[
   (t\vee s)^{\gamma}\,\abs{x-y}\le\beta.
\]
A vertex $(x,t)$ has radius $\beta t^{-\gamma}$, Pareto with tail
exponent $1/\gamma$, and the model above is the case $\gamma=1$. We keep
the notation $\beta_c$ for the threshold at exponent $\gamma$. The next
proposition shows that the phase transition threshold is degenerate whenever $\gamma\neq 1$.

\begin{prop}\label{prop:gamma}
Fix $\gamma>0$, $\lambda>0$ and $\beta>0$, and consider the
weight-dependent random connection model with kernel $(t\vee s)^{\gamma}$
and profile function $\rho=\1_{[0,1]}$.
\begin{enumerate}[(i)]
\item If $\gamma<1$, then almost surely every connected component is
finite; hence $\beta_c=\infty$ at every intensity.
\item If $\gamma>1$, then almost surely an infinite connected component
exists; hence $\beta_c=0$ at every intensity.
\end{enumerate}
\end{prop}

\begin{proof}
(i) A neighbour of $(x,t)$ has its position within
$\beta(t\vee s)^{-\gamma}\le\beta t^{-\gamma}=:R_x$ of $x$, so every
connected component lies inside one occupied component of the Boolean
model on $\R$ with centres $\X$ and radii $R_x$. The radii are
independent of the positions, with
\[
   \E[R_x]=\beta\int_0^1 t^{-\gamma}\,\de t=\frac{\beta}{1-\gamma}<\infty,
\]
and a one-dimensional Boolean model whose radii have finite mean has no
unbounded occupied component at any intensity
\cite[Theorem~3.1]{MeesterRoy1996}. As $\X$ is locally finite, an
infinite component would occupy an unbounded set of positions, so almost
surely there is none.

(ii) First, $(t\vee s)^{\gamma}\le t\vee s$ for $t\vee s\in(0,1)$ and
$\gamma>1$, so at equal $\beta$ the graph at exponent $\gamma$ contains,
on the same points, the graph at exponent $1$; by \Cref{thm:super} the
latter almost surely has an infinite component once $\lambda\beta\ge31$.
Second, let $\beta>0$ be arbitrary and $\eps\in(0,1)$. The points with
mark below $\eps$ form a Poisson process of intensity $\lambda$ on
$\R\times(0,\eps)$, and $(x,t)\mapsto(\eps x,t/\eps)$ sends it to a
Poisson process of intensity $\lambda$ on $\R\times(0,1)$ while
transforming the rule into
$(t'\vee s')^{\gamma}\abs{x'-y'}\le\beta\eps^{1-\gamma}$. The induced
subgraph is therefore a copy of the model at exponent $\gamma$ with
parameter $\beta\eps^{1-\gamma}$, which for $\gamma>1$ exceeds any bound
as $\eps\downarrow0$. Choosing $\eps$ with
$\lambda\beta\eps^{1-\gamma}\ge31$ produces an infinite component in the
induced subgraph, hence in the graph.
\end{proof}

At $\gamma=1$ the restriction step returns the parameter unchanged,
$\beta\eps^{1-\gamma}=\beta$, and the argument gives nothing at the
critical exponent.

The containment in part (ii) is not specific to the family
$(t\vee s)^{\gamma}$. In the taxonomy of
\cite{GracarLuechtrathMoerters2021}, our model is the weight-dependent
random connection model in $d=1$ with the boundary case $\gamma=0$ of
the preferential attachment kernel
$g_{\mathrm{pa}}(s,t)=\beta^{-1}(s\vee t)^{1-\gamma}(s\wedge t)^{\gamma}$
and with indicator profile; the factor $\beta^{-1}$ is the convention of
\cite{GracarLuechtrathMoerters2021}, in which $\beta$ enters through the
kernel. For $\gamma\in(0,1)$, a parameter of the kernel and not the
exponent of \Cref{prop:gamma}, the same kernel and profile give the
age-dependent random connection model of
\cite{GracarGrauerLuechtrathMoerters2019} on the line, which joins
$(x,t)$ and $(y,s)$ precisely when
$(t\vee s)^{1-\gamma}(t\wedge s)^{\gamma}\abs{x-y}\le\beta$. From
$(t\wedge s)^{\gamma}\le(t\vee s)^{\gamma}$ the kernel is pointwise
decreasing in $\gamma$, so our model is the smallest graph of the
family. Theorem~1.4 of \cite{GracarLuechtrathMoench2022} determines the
phase diagram of the family in $d=1$ for profiles of polynomial decay
$\delta>2$ except for $\gamma<\tfrac12$, which is posed there as an open
problem; \Cref{fig:adrcm} shows the diagram in the coordinates of the
interpolation kernel $(s\wedge t)^{\gamma}(s\vee t)^{\alpha}$ of that
paper, in which the family is the line $\alpha=1-\gamma$ and our model
the endpoint $(0,1)$. Our third main result is a phase transition for
the age-dependent random connection model, for the indicator profile and
for profiles of polynomial decay.

\begin{theorem}\label{thm:adrcm}
Fix $\gamma\in(0,1)$, $\lambda>0$ and $\beta>0$, and consider the
weight-dependent random connection model with kernel
$(t\vee s)^{1-\gamma}(t\wedge s)^{\gamma}$ and a profile function $\rho$
that either is the indicator $\1_{[0,1]}$, in which case we set
$\delta=\infty$, or satisfies
$c(1\wedge r^{-\delta})\le\rho(r)\le C(1\wedge r^{-\delta})$ for all
$r>0$, for some $\delta\in(2,\infty)$ and constants $0<c\le C$. If
$\lambda\beta$ is sufficiently large, depending on $\rho$, then almost
surely an infinite connected component exists; for the indicator profile
it suffices that $\lambda\beta\ge31$. If $\gamma<\delta/(\delta+1)$ and
$\lambda\beta$ is sufficiently small, depending on $\gamma$ and $\rho$,
then almost surely every connected component is finite; for the
indicator profile and $\gamma<\tfrac12$ it suffices that
$\lambda\beta<(1-2\gamma)/8$. In particular, at intensity one the
critical value of the age-dependent random connection model satisfies
\[
   \beta_c\in(0,\infty)\qquad\text{for every }
   \gamma<\frac{\delta}{\delta+1},
\]
and for the indicator profile $\beta_c\in(0,31]$ for every
$\gamma\in(0,1)$.
\end{theorem}

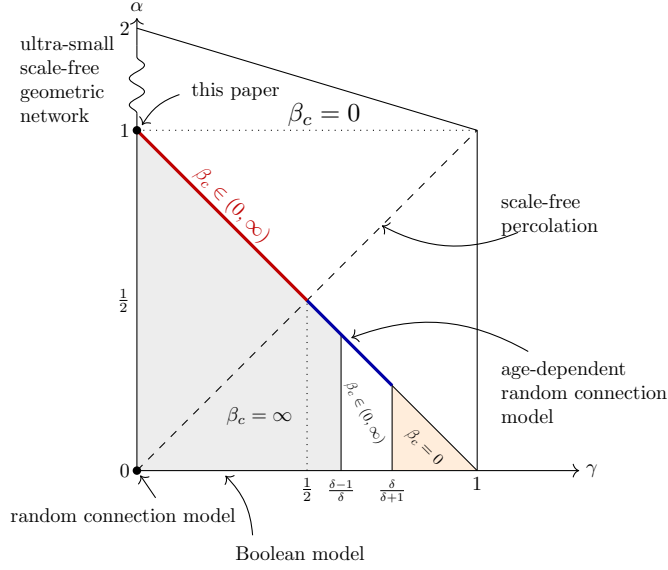
\begin{figure}[ht]\centering
\begin{tikzpicture}[scale=0.45, every node/.style={scale=0.75}]
  \fill[gray!14] (0,0) -- (0,10) -- (6,4) -- (6,0) -- cycle;
  \fill[orange!14] (7.5,0) -- (10,0) -- (7.5,2.5) -- cycle;
  \draw[->] (0,0) -- (13,0) node[right] {$\gamma$};
  \draw (5,0) node[anchor=north] {$\tfrac12$}
        (10,0) node[anchor=north] {$1$};
  \draw[dotted] (0,10) -- (10,10);
  \draw (10,0) -- (10,10) (10,10) -- (0,13);
  \draw (0,0) -- (0,10.5);
  \draw[->] (0,12.5) -- (0,13.3) node[above] {$\alpha$};
  \draw[decorate, decoration={snake, segment length=10pt, amplitude=1mm}]
        (0,10.5) -- (0,12.5);
  \draw (0,0) node[anchor=east] {$0$}
        (0,5) node[anchor=east] {$\tfrac12$}
        (0,10) node[anchor=east] {$1$}
        (0,13) node[anchor=east] {$2$};
  \draw (6,-0.1) node[anchor=north, scale=0.85] {$\tfrac{\delta-1}{\delta}$};
  \draw (6,0) -- (6,4);
  \draw (7.5,-0.1) node[anchor=north, scale=0.85] {$\tfrac{\delta}{\delta+1}$};
  \draw (7.5,0) -- (7.5,2.5);
  \draw[dotted] (5,0) -- (5,5);
  \draw[dashed] (0,0) -- (10,10);
  \draw (0,10) -- (10,0);
  \draw[very thick, blue!65!black] (5,5) -- (7.5,2.5);
  \draw[very thick, red!75!black] (0,10) -- (5,5);
  \node[rotate=-45, red!75!black] at (2.8,7.8) {$\beta_c\in(0,\infty)$};
  \node[vtx] at (0,10) {};
  \node[anchor=west] at (1.4,11.15) {this paper};
  \draw[->, bend right=20] (1.25,11.05) to (0.15,10.15);
  \draw (4.8,-2.0) node[anchor=north] {Boolean model};
  \draw[->, bend right=20] (3.4,-1.9) to (2.6,-0.15);
  \draw (-0.45,11.5) node[anchor=east, align=left] {ultra-small\\ scale-free\\ geometric\\ network};
  \draw (10.5,2.3) node[anchor=west, align=left] {age-dependent\\ random connection\\ model};
  \draw[->, bend right=25] (11.5,3.5) to (6.3,3.9);
  \draw (10.5,7.5) node[anchor=west, align=left] {scale-free\\ percolation};
  \draw[->, bend left=25] (11.2,7) to (7.2,7.2);
  \node[vtx] at (0,0) {};
  \draw (-0.4,-0.9) node[align=left, anchor=north] {random connection model};
  \draw[->] (0.55,-0.8) -- (0.1,-0.14);
  \node at (3.6,1.6) {$\beta_c=\infty$};
  \node[rotate=-62, scale=0.7] at (6.7,1.7) {$\beta_c\in(0,\infty)$};
  \node[rotate=-45, scale=0.8] at (8.45,0.7) {$\beta_c=0$};
  \node[scale=1.3] at (5.5,10.45) {$\beta_c=0$};
\end{tikzpicture}
\caption{ The phase diagram of the interpolation kernel
$(s\wedge t)^{\gamma}(s\vee t)^{\alpha}$ of
\cite{GracarLuechtrathMoench2022} in $d=1$, for profiles of polynomial
decay $\delta>2$; the regions off the line $\alpha=1-\gamma$ are those
of Theorem~1.4 there. Our model is the marked endpoint
$(\gamma,\alpha)=(0,1)$ of that line. Along the line the kernel
decreases as $\gamma$ grows, so percolation for $\lambda\beta\ge31$
spreads from the endpoint to every $\gamma$. On the blue segment
$\tfrac12\le\gamma<\delta/(\delta+1)$ the transition was known to be
non-degenerate by Theorem~1.4 and Corollary~1.5 there; the red segment
$\gamma<\tfrac12$ was open and follows from \Cref{thm:adrcm}. }
\label{fig:adrcm}
\end{figure}

\begin{proof}
From $(t\wedge s)^{\gamma}\le(t\vee s)^{\gamma}$ we get
$(t\vee s)^{1-\gamma}(t\wedge s)^{\gamma}\le t\vee s$ pointwise, so for
the indicator profile the graph contains, at every $\beta$ and on the
same points, the weight-dependent random connection model of
\Cref{thm:super}, and for $\lambda\beta\ge31$ an infinite component
exists almost surely. For a profile of polynomial decay the
supercritical phase is \Cref{cor:adrcmclass}, proved in
\Cref{sec:super}.

For the subcritical statement, Corollary~1.5 of
\cite{GracarLuechtrathMoench2022} states that in $d=1$, for the kernel
$(t\wedge s)^{\gamma}(t\vee s)^{\alpha}$ with $\alpha=1-\gamma$ and
any non-trivial profile bounded above by a constant multiple of
$1\wedge r^{-\delta}$ for some $\delta\in(2,\infty)$, the critical
value at intensity one lies in $(0,\infty)$ whenever
$\tfrac12\le\gamma<\delta/(\delta+1)$. A profile of polynomial decay
satisfies the hypothesis with its own $\delta$, and the indicator
profile with any $\delta$ above $\max(2,\gamma/(1-\gamma))$, so the
corollary covers every admissible profile in the range
$\gamma\ge\tfrac12$. For $\gamma<\tfrac12$ the kernel is pointwise
decreasing in $\gamma$, so the graph is contained, on the same points
and with the same profile, in the model with $\gamma=\tfrac12$, and
positivity of the threshold transfers. Finally, the map
$(x,t)\mapsto(\lambda x,t)$ sends intensity $\lambda$ to intensity one
and $\beta$ to $\lambda\beta$.

It remains to prove the explicit bound for the indicator profile and
$\gamma<\tfrac12$. The first-moment computation in the proof of
Lemma~2.1 of \cite{GracarLuechtrathMoerters2021} bounds the expected
number of self-avoiding paths of length $n$ from a fixed vertex by
$\big(8\lambda\beta/(1-2\gamma)\big)^n$; the computation involves the
profile only through its total mass, and the indicator profile at
parameter $\beta$ enters with edge density $2\beta$. For
$\lambda\beta<(1-2\gamma)/8$ the bound is summable in $n$, and the
argument there shows that almost surely every connected component is
finite.
\end{proof}

In particular $\beta_c=\infty$ occurs nowhere on the line
$\alpha=1-\gamma$: the supercritical bound transfers at every
$\gamma\in(0,1)$, for the indicator profile and for every profile of
polynomial decay. For $\gamma<(\delta-1)/\delta$ the line is the upper
boundary of the region $\beta_c=\infty$ of \Cref{fig:adrcm}, so the
threshold is infinite below the line and finite on it.

\subsection[First properties of the model]{ First properties of the
model }

 We return to the model of \Cref{thm:continuum,thm:super}, with the
weak kernel and the indicator profile. 
The next two propositions say that long edges are everywhere, and that two of
them belong to the same connected component as soon as their spans interleave. Together they make an infinite
component look hard to avoid; how it is nevertheless avoided for
$\lambda\beta<1$ occupies \Cref{sec:simplified,sec:continuum}.

The first proposition says that the number of edges crossing a point is not merely of
infinite mean but almost surely infinite, and the proof gives this at every
point of the line at once. Neither main theorem uses it, but it shows that no
point of the line is free of crossing edges, so a disconnection argument has
to control long edges rather than exclude them. Its proof is the multi-scale
scheme used repeatedly below in its simplest instance: disjoint dyadic blocks,
moments uniform in the scale, Paley--Zygmund, and Borel--Cantelli across
independent scales. It is also the input to
\Cref{cor:diameter,cor:diametercont}.

\begin{prop}\label{prop:edges-io}
    Consider the weight-dependent random connection model on $\R$ with connection probability $\1_{[0,\beta]}((t\vee s)|x-y|)$. Then, for $\beta<\infty$, almost surely every point of $\R$ is crossed by infinitely many edges.
\end{prop}

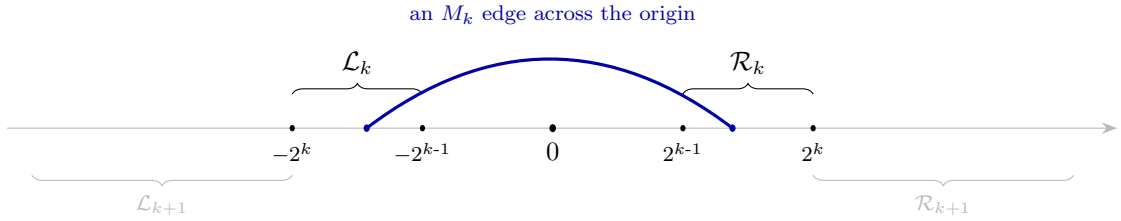
\begin{figure}[ht]\centering
\begin{tikzpicture}[xscale=0.82]
  \draw[zline] (-8.8,0) -- (9.1,0);
  \fill (0,0) circle (1.5pt); \node[below=2pt] at (0,0) {$0$};
  \foreach \x in {-4.2, -2.1, 2.1, 4.2}{\fill (\x,0) circle (1.2pt);}
  \node[below=2pt] at (-2.1,0) {\footnotesize $-2^{k\text{-}1}$};
  \node[below=2pt] at (-4.2,0) {\footnotesize $-2^{k}$};
  \node[below=2pt] at (2.1,0)  {\footnotesize $2^{k\text{-}1}$};
  \node[below=2pt] at (4.2,0)  {\footnotesize $2^{k}$};
  \draw[brc] (-2.1,0.45) -- (-4.2,0.45) node[midway,above=4pt]{$\mathcal L_k$};
  \draw[brc] (4.2,0.45)  -- (2.1,0.45)  node[midway,above=4pt]{$\mathcal R_k$};
  \fill[blue!65!black] (-3.0,0) circle (1.4pt);
  \fill[blue!65!black] (2.9,0) circle (1.4pt);
  \draw[outerarch] (-3.0,0) to[bend left=32] (2.9,0);
  \node[blue!65!black,above] at (0,1.25) {\footnotesize an $M_k$ edge across the origin};
  \draw[brc, gray!55] (-8.4,-0.62) -- (-4.2,-0.62) node[midway,below=4pt]{\footnotesize $\mathcal L_{k+1}$};
  \draw[brc, gray!55] (4.2,-0.62) -- (8.4,-0.62) node[midway,below=4pt]{\footnotesize $\mathcal R_{k+1}$};
\end{tikzpicture}
\caption{
$\mathcal L_k=(-2^{k},-2^{k-1}]$, $\mathcal R_k=[2^{k-1},2^{k})$ (to scale). $M_k$ counts
origin-crossing edges from $\mathcal L_k$ to $\mathcal R_k$; distinct scales use disjoint vertices,
so the $M_k$ are independent, and Paley--Zygmund plus Borel--Cantelli give $M_k\ge1$
infinitely often.}
\label{fig:edges-io}
\end{figure}

\begin{proof}
Let $X$ be the number of origin-crossing edges, i.e.\ edges $\{(x,t),(y,s)\}$ with $x\le 0<y$.
Averaging the connection indicator over the two independent uniform marks, two points at
positions $x\le 0<y$ at distance $d=y-x$ form an edge with probability
\begin{equation}\label{eq:pd}
   p(d):=\int_0^1\!\!\int_0^1\1\big[(t\vee s) d\le\beta\big] \mathrm dt \mathrm ds
        =\big(\min(1,\beta/d)\big)^2
        =\begin{cases}1,& d\le\beta,\\[2pt]\beta^2/d^2,& d>\beta.\end{cases}
\end{equation}

By the Mecke equation, with $u=-x\ge 0$ and $d=y-x$,
\[
   \E[X]=\int_{-\infty}^{0}\!\int_{0}^{\infty} p(y-x) \mathrm dy \mathrm dx
        =\int_0^\infty d p(d) \mathrm dd
        =\tfrac{\beta^2}{2}+\beta^2\!\int_\beta^\infty\frac{\mathrm dd}{d}=\infty .
\]
This is the harmonic (logarithmic) divergence, but an infinite mean does not by itself force $X=\infty$ almost surely. We make the argument across independent scales.

Fix $k_0$ with $2^{k_0-1}\ge\beta$. For $k\ge k_0$ put
$\cL_k=(-2^{k},-2^{k-1}]$ and $\cR_k=[2^{k-1},2^{k})$, and let $M_k$ be the number of edges with
left endpoint position in $\cL_k$ and right endpoint position in $\cR_k$ (see \Cref{fig:edges-io}). The strips
$\cL_k\times(0,1)$ and $\cR_k\times(0,1)$ are pairwise disjoint across $k$, so
$(M_k)_{k\ge k_0}$ are independent. For $x\in\cL_k$, $y\in\cR_k$ one has $d=y-x\in[2^k,2^{k+1})$,
so $p(d)\in(\beta^2/4^{k+1}, \beta^2/4^{k}]$ by \eqref{eq:pd}.

The rectangle $\cL_k\times\cR_k$ has area $(2^{k-1})^2=4^{k-1}$, so
\[
   \frac{\beta^2}{16} \le \E[M_k]=\int_{\cL_k}\!\!\int_{\cR_k}p(y-x) \mathrm dy \mathrm dx \le \frac{\beta^2}{4},
   \qquad\text{uniformly in }k.
\]

Write $\E[M_k^2]=\E[M_k]+\E[M_k(M_k-1)]$ and split the ordered
pairs of distinct edges by shared endpoints (two distinct edges share at most one).
\begin{itemize}
\item \emph{Four distinct points.} By Mecke this is the product integral, $\le(\E[M_k])^2\le\beta^4/16$.
\item \emph{Shared left endpoint} $(x,t)\in\cL_k$, $(y_1,s_1),(y_2,s_2)\in\cR_k$. With $d_i=y_i-x$,
      the mark integral factorises ($t\le\beta/d_1$, $t\le\beta/d_2$, $s_i\le\beta/d_i$):
      \[
        \frac{\beta}{\max(d_1,d_2)}\cdot\frac{\beta}{d_1}\cdot\frac{\beta}{d_2}
         \le \frac{\beta^3}{2^{k} d_1 d_2} \le \frac{\beta^3}{8^{k}},
      \]
      using $\max(d_1,d_2)\ge2^k$, $d_i\ge2^k$. Over the ordered position volume
      $2^{k-1}(2^{k-1})^2=8^{k-1}$ this is $\le\beta^3/8$.
\item \emph{Shared right endpoint.} Identical, $\le\beta^3/8$.
\end{itemize}
Hence $\E[M_k^2]\le \beta^2/4+ \beta^4/16+\beta^3/4=:K(\beta)<\infty$, uniformly in $k$.

For every $k\ge k_0$,
\[
   \P(M_k\ge 1) \ge \frac{(\E[M_k])^2}{\E[M_k^2]} \ge \frac{(\beta^2/16)^2}{K(\beta)} =: p>0 .
\]
The $\{M_k\ge 1\}$ are independent with $\sum_{k\ge k_0}\P(M_k\ge1)=\infty$, so by the second
Borel--Cantelli lemma $\P(M_k\ge1\text{ i.o.})=1$. Each such scale gives a distinct
origin-crossing edge, so $X=\infty$ almost surely. The edge produced at scale $k$ has left
endpoint at most $-2^{k-1}$ and right endpoint at least $2^{k-1}$, so every $x\in\R$ lies in
all but finitely many of these spans, and the same edges therefore cross every point of the
line.
\end{proof}

\begin{remark}\label{rem:edges-io-Z}
On $\Z$, in the simplified model of \Cref{sec:simplified}: $p(d)=x_m^2/d^2$ for $d\ge1$, $\E[X]=x_m^2\sum_{d\ge1}1/d=\infty$, the dyadic blocks
$\cL_k=(-2^{k},-2^{k-1}]$ and $\cR_k=[2^{k-1},2^{k})$ carry $4^{k-1}$ vertex pairs, and the shared-endpoint sums reproduce the bounds
with $\beta$ replaced by $x_m$; independence across scales again gives the conclusion, and since
$\Z$ is countable, almost surely every site is crossed by infinitely many edges.
\end{remark}

The second proposition concerns two edges whose endpoints interleave.

\begin{prop}\label{prop:crossing}
    Let $\{a,b\}$ and $\{c,d\}$ be edges whose endpoints interleave, $a<c<b<d$.
    Then $\{c,b\}$ is also an edge; in particular $a,b,c,d$ all belong to the
    same connected component.
\end{prop}

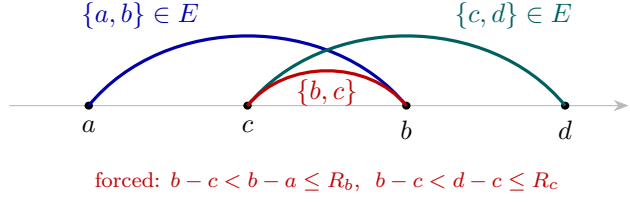
\begin{figure}[ht]\centering
\begin{tikzpicture}[xscale=1.05]
  \draw[zline] (-0.6,0) -- (7.2,0);
  \foreach \x/\lab in {0.4/a, 2.4/c, 4.4/b, 6.4/d}
     {\fill (\x,0) circle (1.5pt); \node[below=2pt] at (\x,0) {$\lab$};}
  \draw[outerarch] (0.4,0) to[bend left=52] (4.4,0);
  \draw[innerarch] (2.4,0) to[bend left=52] (6.4,0);
  \draw[forced]    (2.4,0) to[bend left=52] (4.4,0);
  \node[blue!65!black] at (1.05,1.2) {$\{a,b\}\in E$};
  \node[teal!75!black] at (5.75,1.2) {$\{c,d\}\in E$};
  \node[red!75!black, inner sep=1pt] at (3.4,0.18) {$\{b,c\}$};
  \node[red!75!black, align=center] at (3.4,-1.0)
     {\footnotesize forced: $b-c<b-a\le R_b$, $\ b-c<d-c\le R_c$};
\end{tikzpicture}
\caption{Two edges whose endpoints interleave,
$a<c<b<d$. Because $b-c$ is shorter than both $b-a$ and $d-c$, the edge $\{b,c\}$ is forced, so
$a\sim b\sim c\sim d$ all lie in one component.}
\label{fig:crossing}
\end{figure}

\begin{proof}
Let $\{(a,t_a),(b,t_b)\}$ and $\{(c,t_c),(d,t_d)\}$ be edges whose endpoints interleave,
$a<c<b<d$, as in \Cref{fig:crossing}. From the two edges,
\[
   t_b(b-a)\le(t_a\vee t_b)(b-a)\le\beta,\qquad t_c(d-c)\le(t_c\vee t_d)(d-c)\le\beta .
\]
Since $0<b-c<b-a$ and $0<b-c<d-c$,
\[
   t_b(b-c)<t_b(b-a)\le\beta,\qquad t_c(b-c)<t_c(d-c)\le\beta,
\]
so $(t_b\vee t_c)(b-c)\le\beta$ and $\{(b,t_b),(c,t_c)\}$ is an edge. The chain
$(a,t_a)\sim(b,t_b)\sim(c,t_c)\sim(d,t_d)$ puts all four endpoints in one component.
\end{proof}

The two propositions together come close to producing an infinite component.
Almost surely infinitely many edges cross the origin, and their spans are
unbounded. Any two such spans overlap, and if the edges share an endpoint or
interleave, all four endpoints lie in one component; a sequence of
interleaving edges with unbounded spans would therefore place the origin in
an infinite component. Nested pairs are the only exception, and a nested pair
still merges if any edge joins the two arches. What remains are nested pairs
with no edge between the arches; adding a balance condition on the overhangs
gives the \emph{rainbows} of \Cref{sec:rainbows}, and \Cref{thm:continuum}
says that for $\lambda\beta<1$ these are frequent enough to confine the
component of every vertex, at infinitely many scales.

\begin{remark}
\Cref{prop:crossing} is not invoked in the proofs of the main theorems; in
\Cref{sec:super} it is explicitly dispensable, see the remark following
\Cref{lem:link}. It does dictate the form of the one-sided gap problem of
\Cref{sec:reduction}: an edge from an overhang into the interior $(c,d)$
collapses onto the inner arch, so gap events restricted to the overhangs
alone cannot certify disconnection, and the marked model of
\Cref{def:marked} must carry the inner span and its boundary data.
The computation applies verbatim in the simplified model of
\Cref{sec:simplified}, where it gives $b-c\le R_b\wedge R_c$, i.e.\
$\{b,c\}\in E$.
\end{remark}

\subsection{Rainbows}\label{sec:rainbows}

A point $(x,t)$ connects to $(y,s)$ iff $(t\vee s)\abs{x-y}\le\beta$, i.e.\ iff
$\abs{x-y}\le\min(\beta/t,\beta/s)$, so $(x,t)$ acts as a vertex of radius
$\beta/t$ under the symmetric $\min$-rule. Since $t$ is uniform on $(0,1)$,
$\P(\beta/t>r)=\P(t<\beta/r)=\beta/r$ for $r\ge\beta$: the per-vertex radius is
Pareto with scale $\beta$. We write $R_x$ for the radius of a vertex $x$ and
work with this formulation from here on.

\begin{mydef}[Rainbow]\label{def:rainbow}
A \emph{rainbow} is a pair of edges $\{a,b\}$, $\{c,d\}$ with $a<c<d<b$ whose
two diagonals $\{a,c\}$ and $\{d,b\}$ are absent, and whose overhangs are
\emph{balanced}:
\[
   \abs{\ell-r}\ \le\ m,\qquad\text{where }
   \ell:=c-a,\quad m:=d-c,\quad r:=b-d
\]
are the left overhang, the inner span, and the right overhang. We call
$\{a,b\}$ the outer arch and $\{c,d\}$ the inner arch (\Cref{fig:rainbow}).

The definition uses only positions, radii and the $\min$-rule, so it applies
verbatim to the integer model of \Cref{sec:simplified} below. Since the
outer arch forces $R_a,R_b\ge b-a$, and $b-a$ exceeds both $\ell$ and $r$,
the two absent diagonals constrain only the inner radii:
\[
   R_a\ge b-a,\qquad R_b\ge b-a,\qquad R_c\in[m,\ell),\qquad R_d\in[m,r) ;
\]
in particular $\ell,r>m$. Balance says $\ell\le m+r$ and $r\le m+\ell$, so
these windows give
\[
   R_c<\ell\le m+r=b-c,\qquad R_d<r\le m+\ell=d-a :
\]
neither inner endpoint reaches the far outer endpoint, that is, the long
diagonals $\{c,b\}$ and $\{a,d\}$ are absent as well.
\end{mydef}

\begin{figure}[ht]\centering
\begin{tikzpicture}[xscale=1.0]
  \draw[zline] (-0.6,0) -- (9.3,0);
  \foreach \x/\lab in {0.5/a, 3.2/c, 5.2/d, 8.5/b}
     {\fill (\x,0) circle (1.5pt); \node[below=2pt] at (\x,0) {$\lab$};}
  \draw[outerarch] (0.5,0) to[bend left=48] (8.5,0);
  \draw[innerarch] (3.2,0) to[bend left=48] (5.2,0);
  \node[blue!65!black] at (4.5,2.15) {outer arch $\{a,b\}$};
  \node[teal!75!black] at (4.2,0.95) {inner arch $\{c,d\}$};
  \draw[brc] (0.5,-0.55) -- (3.2,-0.55) node[midway,below=4pt]{$\ell=c-a$};
  \draw[brc] (3.2,-0.55) -- (5.2,-0.55) node[midway,below=4pt]{$m=d-c$};
  \draw[brc] (5.2,-0.55) -- (8.5,-0.55) node[midway,below=4pt]{$r=b-d$};
  \node[align=left] at (4.5,-1.95)
     {\footnotesize $R_a,R_b\ge b-a;\quad R_c\in[m,\ell);\quad R_d\in[m,r);\qquad
      \ell>m,\ r>m;\qquad \abs{\ell-r}\le m$};
\end{tikzpicture}
\caption{A rainbow: outer arch $\{a,b\}$ over a nested inner
arch $\{c,d\}$, with left overhang $\ell$, inner span $m$, right overhang $r$. The non-edges
$\{a,c\},\{d,b\}$ cap the inner radii ($R_c<\ell$, $R_d<r$), forcing $\ell,r>m$; balance
$\abs{\ell-r}\le m$ caps them further, by $R_c<m+r$ and $R_d<m+\ell$, so that the long
diagonals $\{c,b\}$ and $\{a,d\}$ are absent too.}
\label{fig:rainbow}
\end{figure}

\section{The simplified model}\label{sec:simplified}
To avoid lower order random effects, we work in this section with the integer
counterpart of the radius picture of \Cref{sec:rainbows}, which we call the
simplified model. Let the vertex set $V$ be $\Z$,
where each vertex $i\in\Z$ carries an independent Pareto radius $R_i$ with shape
$1$ and scale $x_m\in(0,1)$ (playing the role of $\beta$),
\[
    \P(R_i>x)=\begin{cases}
    1,&x<x_m,\\
    \frac{x_m}{x},&x\geq x_m.
    \end{cases}
\]
Then, in analogy to the original model, an edge exists between two vertices $i$ and $j$, if and only if $|i-j|\leq R_i\wedge R_j$ and we write
\[
p_{ij}:=\P(|i-j|\leq R_i\wedge R_j).
\]
Due to the assumption that $x_m<1$, the probability that there exists an edge between a vertex $i$ and its neighbour $i+1$ is strictly less than 1.
As noted in \Cref{sec:intro}, this is the graph $G_{p,\delta}$ of Yukich
\cite{Yukich2006} on $\Z$ with $p=1$ and $\delta=x_m$. He takes $\delta=1$,
which for the same reason forces $R_i\ge1$ at every $i$ and leaves nothing to
prove about connectivity.

\Cref{thm:continuum} rests on the following theorem about the simplified
model, whose proof occupies the remainder of the paper up to and including
\Cref{sec:main}.

\begin{theorem}\label{thm:main}
For every $x_m\in(0,1)$, almost surely every connected component of the
simplified model is finite. In particular, no infinite component exists.
\end{theorem}

The route is the one sketched in \Cref{sec:intro}, and its steps are
\Cref{prop:rainbows-io} (rainbows at infinitely many independent scales),
\Cref{lem:window} and \Cref{lem:contain} (reduction to one gap problem per
overhang), \Cref{prop:cutpoints} and \Cref{lem:cutmarked} (the cut-point
certificates that solve it uniformly in the scale), and \Cref{lem:confine}
with \Cref{prop:scale} (confinement, with probability bounded below per
scale). One point of the assembly is not visible from the sketch: plain
Borel--Cantelli does not apply at the last step, because the containment
events reach across scales, so \Cref{prop:scalesio} runs through
Kochen--Stone and a zero--one law instead.

The basic disconnection notion is the following: if the event
\begin{equation}\label{eq:gap}
    \{\text{gap at }i\}:=\{\not\exists j\leq i\text{ s.t. }\exists k>i\text{ with } |k-j|\leq R_j\wedge R_k\}
\end{equation}
occurs, then any vertex to the left of $i$ cannot belong to the same connected component as any vertex to the right of $i$.
No site is a gap, however: infinitely many edges cross every site almost
surely (\Cref{rem:edges-io-Z}), so the event above is null for each $i$, and
with it the naive route to disconnection. Disconnection is therefore sought
inside a rainbow, where the outer arch is the only edge in $[a,b]$ allowed to cross the
cut.

In particular, if $\{a,b\},\{c,d\}$ form a rainbow (\Cref{def:rainbow}) and there
exist cut points $i\in(a,c)$ and $i'\in(d,b)$ such that no edge other than
the outer arch crosses either of them, then the inner and the outer arch
lie in disjoint connected components of the graph restricted to $[a,b]$
(and hence, by \Cref{lem:window} below, of the full graph): any path
between the arches inside the window would have to cross $i$ or $i'$ by an
edge, and the only such edge is the arch itself, which does not meet the
inner component. As the remark following \Cref{prop:crossing}
explains, such certificates cannot be confined to the overhangs; the
translation to the $(0,n)$ gap problem below therefore carries the
boundary vertices and the inner span with it.

\subsection{Formation of rainbows}\label{sec:formation}

Both \Cref{prop:scale} and \Cref{prop:scalecont} build on the following
construction.

\begin{mydef}[Scale blocks]\label{def:blocks}
For $k\ge1$ set $g=g_k:=8^{2k-1}$ and let
\[
   \mathsf a_k=(-9g,-8g],\quad \mathsf c_k=(-2g,-g],\quad
   \mathsf d_k=[g,2g),\quad \mathsf b_k=[8g,9g)
\]
be the four \emph{scale-$k$ blocks}, each of size $g$ (\Cref{fig:scales}).
A vertex is \emph{qualifying} if it lies in $\mathsf c_k$ or $\mathsf d_k$
with radius in $[4g,5g]$, or in $\mathsf a_k$ or $\mathsf b_k$ with radius
at least $32g$; the \emph{formation event} at scale $k$ is
\begin{align*}
   \cG_k:={}&\{\exists c\in\mathsf c_k:R_c\in[4g,5g]\}\cap\{\exists d\in\mathsf d_k:R_d\in[4g,5g]\}\\
          &\cap\{\exists a\in\mathsf a_k:R_a\ge32g\}\cap\{\exists b\in\mathsf b_k:R_b\ge32g\}.
\end{align*}
All four thresholds exceed $x_m$, since $x_m<1\le g$, so the Pareto tail
applies to each.
\end{mydef}

The proposition below is not itself invoked in the proof of
\Cref{thm:main}: \Cref{prop:scale} re-runs the construction of
\Cref{def:blocks} in a
strengthened form. We nevertheless state this formulation separately: it isolates
the unconditional half of the problem (rainbows \emph{form} at infinitely
many scales, regardless of what happens inside them) from the conditional
half, the disconnection of a formed rainbow, which occupies
\Cref{sec:reduction,sec:cutpoints,sec:main}, and its proof is the skeleton
on which \Cref{prop:scale} builds.

\begin{prop}\label{prop:rainbows-io}
Almost surely there are infinitely many \emph{origin-crossing} rainbows: infinitely many
rainbows (\Cref{def:rainbow}) whose inner arch $\{c,d\}$ satisfies $c<0<d$. Moreover the
probability that such a rainbow forms at a given geometric scale is bounded below uniformly in
the scale, so it does not decay as the scale grows.
\end{prop}

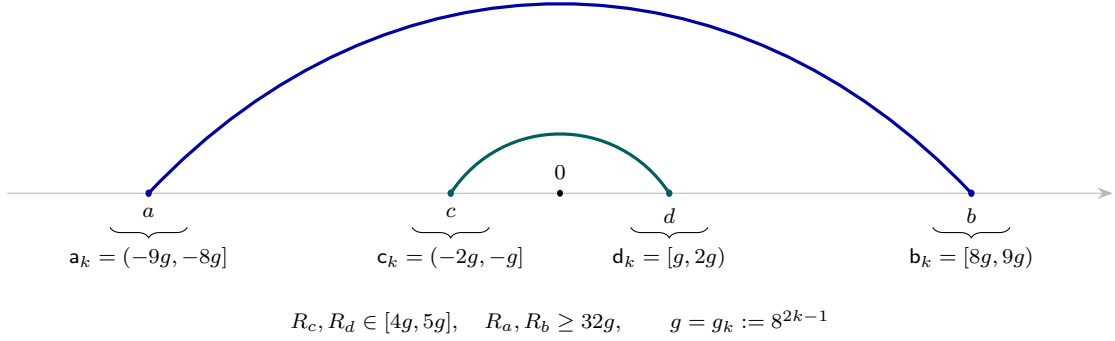
\begin{figure}[ht]\centering
\begin{tikzpicture}[xscale=0.85]
  \draw[zline] (-8.6,0) -- (8.6,0);
  \fill (0,0) circle (1.3pt); \node[above=2pt] at (0,0) {\footnotesize $0$};
  \fill[blue!65!black] (-6.4,0) circle (1.4pt); \node[below=2pt] at (-6.4,0) {\footnotesize $a$};
  \fill[teal!75!black] (-1.7,0) circle (1.4pt); \node[below=2pt] at (-1.7,0) {\footnotesize $c$};
  \fill[teal!75!black] (1.7,0)  circle (1.4pt); \node[below=2pt] at (1.7,0)  {\footnotesize $d$};
  \fill[blue!65!black] (6.4,0)  circle (1.4pt); \node[below=2pt] at (6.4,0)  {\footnotesize $b$};
  \draw[outerarch] (-6.4,0) to[bend left=42] (6.4,0);
  \draw[innerarch] (-1.7,0) to[bend left=52] (1.7,0);
  \draw[brc] (-7.0,-0.45) -- (-5.8,-0.45) node[midway,below=4pt]{\footnotesize $\mathsf a_k=(-9g,-8g]$};
  \draw[brc] (-2.3,-0.45) -- (-1.1,-0.45) node[midway,below=4pt]{\footnotesize $\mathsf c_k=(-2g,-g]$};
  \draw[brc] (1.1,-0.45)  -- (2.3,-0.45)  node[midway,below=4pt]{\footnotesize $\mathsf d_k=[g,2g)$};
  \draw[brc] (5.8,-0.45)  -- (7.0,-0.45)  node[midway,below=4pt]{\footnotesize $\mathsf b_k=[8g,9g)$};
  \node[align=center] at (0,-1.7)
     {\footnotesize $R_c,R_d\in[4g,5g],\quad R_a,R_b\ge 32g,\qquad g=g_k:=8^{2k-1}$};
\end{tikzpicture}
\caption{Four disjoint
blocks at geometric scale $g=8^{2k-1}$, each of size $g$, host the candidate vertices $a,c,d,b$. Moderate inner
radii ($[4g,5g]$) force the inner arch $\{c,d\}$ and forbid the short diagonals $\{a,c\}$, $\{d,b\}$;
large outer radii ($\ge 32g$) force the outer arch $\{a,b\}$. Confining the outer
vertices to blocks of width $g$ at distance $8g$ makes the overhangs comparable,
so the rainbow produced is balanced. Distinct scales use disjoint
vertices, so the scale events are independent.}
\label{fig:scales}
\end{figure}

\begin{proof}
For a rainbow at $(a,c,d,b)$ write
$\ell=c-a$, $m=d-c$, $r=b-d$, $L=\ell+m+r=b-a$. By \Cref{def:rainbow} the four radius conditions
involve the four distinct radii $R_a,R_b,R_c,R_d$, hence are independent, and
\begin{equation}\label{eq:rb-prob}
   \P\big(\text{rainbow at }(a,c,d,b)\big)
   =\frac{x_m}{L}\cdot\frac{x_m}{L}\cdot\frac{x_m(\ell-m)}{m\ell}\cdot\frac{x_m(r-m)}{mr}
   =\frac{x_m^4(\ell-m)(r-m)}{L^2 m^2 \ell r}.
\end{equation}

As motivation, sum \eqref{eq:rb-prob} over origin-crossing quadruples and keep
only $d-c=m$ with $\ell,r\in[2m,3m]$; such quadruples are balanced, since
$\abs{\ell-r}\le3m-2m=m$. There $\ell-m\ge m$, $r-m\ge m$, $L\le 7m$ and $\ell r\le 9m^2$,
so each term is $\ge x_m^4 m^2/(49m^2\cdot m^2\cdot 9m^2)=x_m^4/(441 m^4)$; the region has
$\ge m-1$ inner positions and $\ge m$ integer choices of each of $\ell,r$, i.e.\ $\ge m^3/2$ terms.
Hence
\[
   \E\big[\#\text{rainbows}\big]\ \ge\ \sum_{m\ge2}\frac{m^3}{2}\cdot\frac{x_m^4}{441 m^4}
   \ =\ \frac{x_m^4}{882}\sum_{m\ge2}\frac1m\ =\ \infty .
\]
As for the origin-crossing edges, this is a harmonic divergence, now in the inner length $m$, and
does not on its own force infinitely many rainbows.

We now construct events at independent scales, using the blocks and the
formation event $\cG_k$ of \Cref{def:blocks}.

Pick qualifying $a,c,d,b$. Since $\abs a,\abs b\in[8g,9g)$
while $\abs c,\abs d\in[g,2g)$,
\[
   \ell=\abs a-\abs c\in(6g,8g),\qquad
   r=\abs b-\abs d\in(6g,8g),\qquad
   m=\abs c+\abs d\in[2g,4g),
\]
and
\[
   a\le-8g<-2g<c\le-g<0<g\le d<2g<8g\le b,
\]
so $a<c<0<d<b$. The remaining conditions of \Cref{def:rainbow} hold:
\begin{itemize}
  \item $\{c,d\}\in E$: $ d-c=m<4g\le R_c\wedge R_d$;
  \item $\{a,b\}\in E$: $ b-a=|a|+|b|<18g\le32g\le R_a\wedge R_b$;
  \item $\{a,c\}\notin E$: $ c-a=|a|-|c|\ge 8g-2g=6g>5g\ge R_c\ge R_a\wedge R_c$ (the value of
        $R_a$ is irrelevant: $R_c$ alone cannot reach $a$);
  \item $\{d,b\}\notin E$: $ b-d=|b|-|d|\ge 6g>5g\ge R_d\ge R_b\wedge R_d$;
  \item balance: $\abs{\ell-r}<8g-6g=2g\le m$.
\end{itemize}
The inner radii land in the windows of \Cref{def:rainbow}: $R_c\in[4g,5g]\subseteq[m,\ell)$ since
$m<4g\le R_c$ and $R_c\le 5g<6g\le\ell$, and likewise $R_d\in[m,r)$. The same
bounds give the sharper proportions $\ell,r<8g\le4m$, used in
\Cref{sec:reduction}.

It remains to prove a uniform lower bound. The four events defining $\cG_k$ involve disjoint vertex sets, so
$\cG_k$ factorises. With per-vertex probabilities $x_m\big(\tfrac1{4g}-\tfrac1{5g}\big)=\tfrac{x_m}{20g}$
(inner blocks) and $\tfrac{x_m}{32g}$ (outer blocks), each block holding $g$
vertices, and using $(1-u)^n\le e^{-nu}$,
\begin{align*}
   \P(\exists c\in\mathsf c_k:R_c\in[4g,5g])&=1-\Big(1-\tfrac{x_m}{20g}\Big)^{g}\ge 1-e^{-x_m/20},\\
   \P(\exists a\in\mathsf a_k:R_a\ge32g)&=1-\Big(1-\tfrac{x_m}{32g}\Big)^{g}\ge 1-e^{-x_m/32},
\end{align*}
and likewise for $\mathsf d_k,\mathsf b_k$. Hence
\[
   \P(\cG_k)\ \ge\ (1-e^{-x_m/20})^2 (1-e^{-x_m/32})^2\ =:\ p_0\ >\ 0
   \qquad\text{uniformly in }k.
\]
The scale $g$ cancels, as each block holds $\Theta(g)$ candidate vertices, each qualifying with
probability $\Theta(1/g)$.

With $g_k=8^{2k-1}$, scale $k$ occupies the magnitude bands
$[8^{2k-1},2\cdot8^{2k-1})\cup[8^{2k},\tfrac98 8^{2k})$; since
$\tfrac98 8^{2k}<8^{2k+1}$ these are
disjoint across $k$, so the blocks (hence the radii on which $\cG_k$ depends) are disjoint and the
events $(\cG_k)_{k\ge1}$ are independent. As $\sum_k\P(\cG_k)=\infty$, the second Borel--Cantelli
lemma gives $\P(\cG_k\text{ i.o.})=1$. Each $\cG_k$ yields a rainbow with all four vertices in
scale-$k$ blocks, distinct across $k$; hence there are infinitely many origin-crossing rainbows
almost surely.
\end{proof}

\begin{remark}
The factorisation replaces the Paley--Zygmund step used for the origin-crossing edges: requiring
each of $a,b$ (resp.\ $c,d$) to have a sufficiently large radius individually makes the arch edge
automatic, so the existence of an arch reduces to two independent single-vertex events.
\end{remark}

\begin{remark}
The balance condition on rainbows dictates the widths in \Cref{def:blocks}. Outer blocks of
width $8g$ rather than $g$ would still produce nested arches with absent
diagonals, but the overhangs could then differ by more than the inner span,
and the reduction below would not apply. Narrowing them costs only a
constant: the outer-block probability drops from $1-e^{-x_m/4}$ to
$1-e^{-x_m/32}$, still uniform in $k$.
\end{remark}

\subsection{The window and the reduction}\label{sec:reduction}

The outer arch reaches across the whole span, so a vertex can leave $[a,b]$ only after attaching
to the outer arch from inside. Hence the connection question is decided inside $[a,b]$.

\begin{lemma}\label{lem:window}
On the rainbow event, whether the inner arch $\{c,d\}$ and outer arch $\{a,b\}$ lie in the same
component of the graph on $\Z$ depends only on $\{R_v: a\le v\le b\}$, and equals the same event
in the restriction of the graph to $[a,b]$.
\end{lemma}

\begin{figure}[ht]\centering
\begin{tikzpicture}[xscale=0.95]
  \draw[zline] (-2.9,0) -- (9.3,0);
  \foreach \x/\lab in {-2.3/w, 0.5/a, 3.4/u, 8.5/b}
     {\fill (\x,0) circle (1.5pt); \node[below=2pt] at (\x,0) {$\lab$};}
  \fill[gray, fill opacity=0.15] (0.5,-0.45) rectangle (8.5,2);
  \draw[cut] (0.5,-0.45) -- (0.5,2);
  \draw[cut] (8.5,-0.45) -- (8.5,2);
  \node[gray!60!black] at (6.9,1.8) {window $[a,b]$};
  \draw[outerarch] (0.5,0) to[bend left=44] (8.5,0);
  \draw[escape]    (3.4,0) to[bend left=42] (-2.3,0);
  \draw[forced]    (0.5,0) to[bend left=55] (3.4,0);
  \node[escape, align=center] at (0.2,-1.5)
     {\footnotesize edge leaving $[a,b]$:\\[-1pt]$R_u\ge|u-w|>u-a$};
  \node[red!75!black, fill=white, inner sep=1pt] at (3.4,0.8) {\footnotesize forced $\{a,u\}$};
\end{tikzpicture}
\caption{If a vertex $u$ has an
edge leaving the window to some $w\notin[a,b]$, then $R_u\ge|u-w|>\min(u-a,b-u)$, so $u$ already
attaches to an outer endpoint inside $[a,b]$ (here $\{a,u\}$). Hence the connection question is
settled by the radii in $[a,b]$.}
\label{fig:window}
\end{figure}
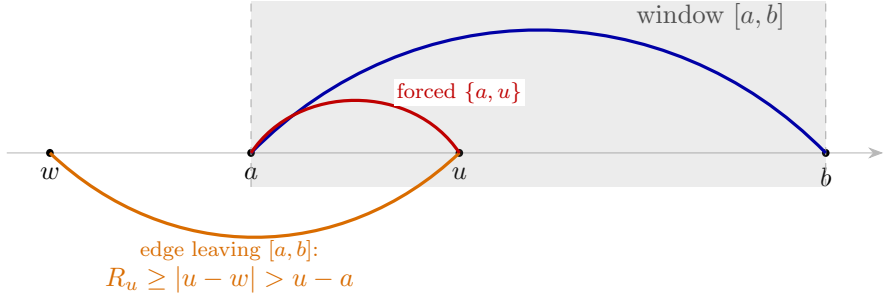

\begin{proof}
A path inside $[a,b]$ is a path in $\Z$, giving one inclusion. Conversely, let $P$ join the
inner arch to the outer arch in $\Z$. If $P\subseteq[a,b]$ we are done; otherwise let $u$ be the
first vertex of $P$ (from the inner arch) whose next edge $\{u,w\}$ has $w\notin[a,b]$, so the
arc of $P$ up to $u$ stays in $[a,b]$ (see \Cref{fig:window}). Since $w\notin[a,b]$,
\[
   R_u\ge\abs{u-w}>\min(u-a, b-u).
\]
As $R_a,R_b\ge b-a\ge\max(u-a,b-u)$: if $u-a\le b-u$ then $R_u\ge u-a$ and $R_a\ge u-a$, so
$\{a,u\}\in E$; otherwise $\{u,b\}\in E$. Either way $u$ is joined to an outer vertex by an edge
inside $[a,b]$, so the inner arch meets the outer arch within $[a,b]$.
\end{proof}

\paragraph{Reduction to one overhang per side.} Since $R_a,R_b\ge b-a$, a vertex
$u\in(a,b)$ is joined to an outer endpoint by an edge as soon as
$R_u\ge\min(u-a,b-u)$; call such a $u$ \emph{outer-attached}. By
\Cref{lem:window}, the two arches share a component iff the connected component
of $\{c,d\}$ in the graph restricted to the open window $(a,b)$ contains an
outer-attached vertex. On the containment event
\begin{equation}\label{eq:B}
   B:=\set{ R_v<\min(v-a, b-v)\ \text{for all }v\in(c,d) },
\end{equation}
no interior vertex is outer-attached, and in addition:
(i) the long diagonals $\{a,d\},\{c,b\}$ are deterministically absent, by
the balance inequalities of \Cref{def:rainbow};
(ii) neither $c$ nor $d$ is outer-attached, for the same reason;
(iii) an overhang vertex can only be outer-attached towards its \emph{near}
endpoint, since $u-a\le\ell\le m+r\le b-u$ for $u\in(a,c)$, and symmetrically.

This does not yet split the problem into two independent one-sided events: the
inner cluster may enter an overhang through an edge into the interior $(c,d)$
(which by \Cref{prop:crossing} collapses onto the inner arch, but is
not an edge of the overhang graph), and the two overhangs may be joined by a
direct edge. We therefore record a \emph{sufficient} certificate. Define
\begin{align}
  \mathcal D^L:=&\{a\not\sim[c,d]\ \text{by a path whose intermediate vertices lie in }(a,c)\},
\\
  \mathcal D^R:=&\{[c,d]\not\sim b\ \text{by a path whose intermediate vertices lie in }(d,b)\},
\\
  \mathcal N:=&\{\{u,w\}\notin E\ \text{for all }u\in(a,c),\ w\in(d,b)\}.
\end{align}
Then, on the rainbow event, $B\cap\mathcal N\cap\mathcal
D^L\cap\mathcal D^R$ implies that the arches lie in disjoint components (see
\Cref{fig:reduction}).
Indeed, suppose a path joins the inner arch to $\{a,b\}$ inside the window and,
say, ends at $a$; let $y$ be the vertex before $a$. If $y\in(c,d)$, then $y$ is
outer-attached, contradicting $B$; $y\in\{c,d\}$ contradicts the rainbow
(resp.\ (i)); if $y\in(d,b)$ then by (iii) $R_y\ge y-a\ge b-y$, so $\{y,b\}\in
E$ and we may instead let the path end at $b$ from within $(d,b)$; and if
$y\in(a,c)$, follow the path backwards from $y$ while it stays in $(a,c)$ --
the first vertex outside is in $[c,d]$ (contradicting $\mathcal D^L$), in
$(d,b)$ (contradicting $\mathcal N$), or does not exist because the path came
from $[c,d]$ directly (again contradicting $\mathcal D^L$). A mirrored argument applies on the right.
Conditionally on all radii of $[c,d]$ (the interior ones capped by $B$, the
endpoint windows given by the rainbow), $\mathcal D^L$ and $\mathcal D^R$
depend on disjoint sets of radii and are independent. The expected number of
cross edges is
\[
   x_m^2\!\!\sum_{u\in(a,c), w\in(d,b)}\!\!(w-u)^{-2}
   \ \le\ x_m^2\log\Big(1+\frac{\ell\wedge r}{m}\Big),
\]
by summing $(m+p+q)^{-2}$ over the distances $p,q\ge1$ to the near inner
endpoints. This is \emph{not} bounded over all rainbows: balance constrains only the
difference of the overhangs, so taking $m$ fixed and $\ell=r\to\infty$ stays
within \Cref{def:rainbow} and gives logarithmic growth. The sum is
$O(x_m^2)$ only under the additional shape constraint $\ell\wedge r\le Cm$,
which holds in particular for the rainbows produced by \Cref{def:blocks},
where $\ell,r<4m$. For such rainbows the Harris
inequality \cite{Harris1960}, together with $1-z\ge e^{-2z}$ on
$[0,\tfrac12]$, gives
$\P(\mathcal N)\ge\big(1+\tfrac{\ell\wedge r}{m}\big)^{-2x_m^2}
\ge 5^{-2x_m^2}$. The
interior event $B$ is the no-inner-vertex-reaches-beyond event; by the
vertex-containment estimate below, $\P(B)\ge(4/9)^{2x_m}$ uniformly in the
rainbow size.

\begin{figure}[ht]\centering
\begin{tikzpicture}[xscale=1.0]
  \node at (0.5,-0.3){$a$};
  \node at (3.0,-0.3){$c$};
  \node at (6.5,-0.3){$d$};
  \node at (9.0,-0.3){$b$};
  \draw[zline] (-0.6,0) -- (9.7,0);
  \draw[outerarch] (0.5,0) to[bend left=42] (9.0,0);
  \draw[innerarch] (3.0,0) to[bend left=52] (6.5,0);
  \draw[pathconn] (0.5,0) to[bend left=20] (3.0,0);
  \draw[pathconn] (6.5,0) to[bend left=20] (9.0,0);
  \node[green!45!black, fill=white, inner sep=1pt] at (1.75,-0.4) {\footnotesize $a\sim c$ in $(a,c)$};
  \node[green!45!black, fill=white, inner sep=1pt] at (7.75,-0.4) {\footnotesize $d\sim b$ in $(d,b)$};
  \fill[violet!70!black] (4.5,0) circle (1.6pt); \node[violet!70!black,below=2pt] at (4.5,0){\footnotesize $v$};
  \draw[excluded] (0.5,0) to[bend left=60] (4.5,0);
  \draw[excluded] (4.5,0) to[bend right=40] (3.0,0);
  \node[red!70!black, align=center] at (3.0,2.3)
     {\footnotesize over-the-top $a\!\sim\!v\!\sim\!c$ (requires $R_v\ge v-a$):\\[-1pt]is excluded on $B$};
\end{tikzpicture}
\caption{On the containment event $B$ the arches can
meet only through an overhang, $a\sim c$ in $(a,c)$ or $d\sim b$ in $(d,b)$ (wavy). The dashed
over-the-top route, an interior $v$ with $R_v\ge v-a$ joining $a$ to $c$ across the span (and the
long diagonal $\{a,d\}$), is what $B$ forbids. On $B$ and with no cross
edge, disconnection follows from $\mathcal D^L\cap\mathcal D^R$.}
\label{fig:reduction}
\end{figure}

\paragraph{One side, by symmetry.} The radii are i.i.d.\ and the rule
$\abs{i-j}\le R_i\wedge R_j$ is invariant under reflection, so the two events
$\mathcal D^L$ and $\mathcal D^R$ have the same law after exchanging the roles of
$\ell$ and $r$, and it suffices to analyse one of them. We treat the left side and
put $x=a$, $y=c$, so that the overhang has length $n:=\ell$ and the far overhang has
length $s:=r$; the right side is the mirror image with $n:=r$, $s:=\ell$. In either
case $\{x,y\}\notin E$.

In translating the overhang to $(0,n)$ one must retain the boundary information
imposed by the rainbow and by the containment event $B$: the outer endpoint $x=a$
carries $R_a\ge b-a$, so within the window its radius never binds ($a\sim u$ iff
$R_u\ge u-a$); the inner endpoint $y=c$ carries $R_c\in[m,\ell)$; and, by the
definition of $\mathcal D^L$, the target is not the single vertex $c$ but the whole
inner stretch $[c,d]$, whose interior radii are capped by $B$ and whose far endpoint
carries $R_d\in[m,r)$. This leads to the following one-sided model.

\begin{mydef}\label{def:marked}
Fix integers $m\ge 1$ and $n,s>m$ with $|n-s|\le m$ (the overhang lengths of a
rainbow and its inner span), and set $D_k:=\min(n+k, s+m-k)$ for
$1\le k\le m-1$ (the caps of $B$, as in \Cref{lem:contain} with $\ell=n$, $r=s$).
The \emph{marked model} on the vertex set $\{0,1,\dots,n+m\}$ has independent radii $R^\sharp_v$ for $v\in \{0,\dots,n+m\}$ with Pareto distribution of shape 1 and scale $x_m\in (0,1)$ as follows
\begin{itemize}
  \item $R^\sharp_0=\infty$ (image of the outer endpoint $a$: an edge $\{0,u\}$
        exists iff $R^\sharp_u\ge u$);
  \item $R^\sharp_u$ i.i.d.\ Pareto for $1\le u\le n-1$ (interior of the overhang);
  \item $R^\sharp_n$ Pareto conditioned on $[m,n)$ (near inner endpoint, image of $c$);
  \item $R^\sharp_{n+k}$ Pareto conditioned on $[x_m,D_k)$ for $1\le k\le m-1$
        (interior of the rainbow, capped by $B$);
  \item $R^\sharp_{n+m}$ Pareto conditioned on $[m,s)$ (far inner endpoint, image of $d$),
\end{itemize}
with the edge rule $\{u,v\}\in E$ iff $|u-v|\le R^\sharp_u\wedge R^\sharp_v$. For
$1\le i\le n-1$ the \emph{marked gap event} is
\[
   \{\text{gap}^\sharp\text{ at }i\}
   :=\{\not\exists  0\le j\le i<k\le n+m \text{ with } |k-j|\le R^\sharp_j\wedge R^\sharp_k\}.
\]
\end{mydef}
The marked model is put to work only in \Cref{lem:cutmarked}; before that,
\Cref{lem:contain} bounds the probability of the containment event $B$ that
produces the conditioned radii above, and \Cref{prop:cutpoints} solves the
gap problem in the unmarked model, in the form that \Cref{lem:cutmarked}
then perturbs.

Conditionally on the rainbow event and on $B$, the radii of $[a,d]$ are independent
with the marked laws above: the constraints $R_a\ge b-a$, $R_c\in[m,\ell)$,
$R_d\in[m,r)$ and the caps of $B$ concern disjoint coordinates, and $R_a\ge b-a$
never binds for targets in $[a,d]$ (as $d-a\le b-a$), so it may be replaced by
$R_a=\infty$. Moreover a marked gap certifies the one-sided disconnection event:
any path from $a$ to $[c,d]$ whose intermediate vertices lie in $(a,c)$ must cross
the cut at $i$ by an edge, which the marked gap forbids. Hence, conditionally on the rainbow
event intersected with $B$,
\[
   \{\text{gap}^\sharp\text{ at some }i\in[1,n-1]\}\ \subseteq\ \mathcal D^L ,
\]
and symmetrically for $\mathcal D^R$; intersecting the two sides and the
no-cross-edge event $\mathcal N$ then yields disconnection, by the reduction above.
The implication runs in one direction only; a lower bound on the
disconnection probability needs no more.

\begin{lemma}\label{lem:contain}
For a rainbow with inner length $m=d-c$ and overhangs $\ell=c-a>m$, $r=b-d>m$, the containment
event $B=\set{R_v<\min(v-a, b-v)\ \text{for all }v\in(c,d)}$ satisfies
\[
   \P(B) = \prod_{k=1}^{m-1}\Big(1-\frac{x_m}{\min(\ell+k, r+m-k)}\Big) \ge \Big(\tfrac49\Big)^{2x_m},
\]
uniformly in $\ell,r,m$ subject to $\ell,r>m$.
\end{lemma}

\begin{figure}[ht]\centering
\begin{tikzpicture}[xscale=1.0]
  \draw[zline] (-0.6,0) -- (9.7,0);
  \foreach \x/\lab in {0.5/a, 3.0/c, 6.5/d, 9.0/b}
     {\fill (\x,0) circle (1.5pt); \node[below=2pt] at (\x,0) {$\lab$};}
  \foreach \x in {3.6,4.2,4.8,5.4,6.0}{\fill[black!70] (\x,0) circle (1.1pt);}
  \draw[outerarch] (0.5,0) to[bend left=44] (9.0,0);
  \draw[innerarch] (3.0,0) to[bend left=44] (6.5,0);
  \fill[violet!70!black] (4.8,0) circle (1.7pt);
  \node[violet!70!black, above=1pt] at (4.8,0.05) {\footnotesize $v=c+k$};
  \draw[reach,{Stealth[length=1.6mm]}-{Stealth[length=1.6mm]}] (0.5,-0.55) -- (4.8,-0.55)
        node[midway,below=2pt]{\footnotesize $v-a=\ell+k$};
  \draw[reach,{Stealth[length=1.6mm]}-{Stealth[length=1.6mm]}] (4.8,-0.55) -- (9.0,-0.55)
        node[midway,below=2pt]{\footnotesize $b-v=r+m-k$};
  \node[align=center] at (4.75,-2.15)
     {\footnotesize containment $B$: every interior $v$ has $R_v<D_k=\min(\ell+k, r+m-k)$};
\end{tikzpicture}
\caption{Each interior vertex $v=c+k$ sits
at distance $\ell+k$ from $a$ and $r+m-k$ from $b$; it reaches an outer endpoint iff
$R_v\ge D_k:=\min(\ell+k,r+m-k)$. The event $B$ that no interior vertex reaches out has
$\mathbb P(B)=\prod_{k}(1-x_m/D_k)\ge(4/9)^{2x_m}$.}
\label{fig:contain}
\end{figure}

\begin{proof}
Write $D_k=\min(\ell+k, r+m-k)$, the distance from the interior vertex $c+k$ to the nearer
outer endpoint (since $(c+k)-a=\ell+k$ and $b-(c+k)=r+m-k$; see \Cref{fig:contain}). That vertex reaches an outer
endpoint iff $R_{c+k}\ge D_k$, which has  probability $x_m/D_k$ (as $D_k\ge x_m$); the interior radii are
independent, giving the product formula. If $m=1$ there are no interior vertices and $\P(B)=1$,
so assume $m\ge2$.

Each term is non-increasing in $\ell$ and in $r$, so the sum is
largest at $\ell=r=m$, where $D_k=\min(m+k,2m-k)$ is symmetric under $k\mapsto m-k$ and equals
$m+k$ for $k\le m/2$. Hence
\[
   \sum_{k=1}^{m-1}\frac1{D_k} \le 2\sum_{k=1}^{\lfloor m/2\rfloor}\frac1{m+k}.
\]
The summand $k\mapsto 1/(m+k)$ is convex, so by the Hermite--Hadamard inequality
$\frac1{m+k}\le\int_{k-1/2}^{k+1/2}\frac{\mathrm dx}{m+x}$, whence with
$\lfloor m/2\rfloor+\tfrac12\le\tfrac{m+1}{2}$,
\[
   2\sum_{k=1}^{\lfloor m/2\rfloor}\frac1{m+k}
    \le 2\int_{1/2}^{\lfloor m/2\rfloor+1/2}\frac{\mathrm dx}{m+x}
    \le 2\log\frac{(3m+1)/2}{m+1/2}
    = 2\log\frac{3m+1}{2m+1} < 2\log\tfrac32 = \log\tfrac94,
\]
where the last inequality follows because $2(3m+1)<3(2m+1)$.

For $1\le k\le m-1$, $D_k\ge\min(\ell+1,r+1)\ge m+2\ge4$, so
$z_k:=x_m/D_k\le x_m/4<\tfrac12$. On $[0,\tfrac12]$ one has $1-z\ge\euler^{-2z}$ (with equality at
$z=0$ and $\tfrac{\mathrm d}{\mathrm dz}\log(1-z)=-\tfrac1{1-z}\ge-2$). Therefore
\[
   \P(B)=\prod_{k=1}^{m-1}(1-z_k) \ge \exp\!\Big(-2\sum_{k=1}^{m-1}z_k\Big)
         \ge \exp\!\big(-2x_m\log\tfrac94\big)=\Big(\tfrac49\Big)^{2x_m}. \qedhere
\]
\end{proof}

We now focus on the gap events. Using the translation invariance of the
model and to simplify notation, we work on $(0,n)$: all radii of $[0,n]$ are
i.i.d.\ Pareto and $\{\text{gap at }i\}|_{(0,n)}$ denotes the event
\eqref{eq:gap} for the graph restricted to $[0,n]$, so that only pairs
$0\le j\le i<k\le n$ are considered. We write
$c(x_m)$, $C(x_m)$ for positive finite constants depending only on $x_m$
that may change from occurrence to occurrence. 

\subsection{Existence of gaps: cut points}\label{sec:cutpoints}

We first treat the unmarked model, where the certificate probabilities
take their simplest form: it is the template that both the marked variant
(\Cref{lem:cutmarked}, whose proof \Cref{prop:scale}(ii) transfers to the
block-decorated overhangs) and the continuum version
(\Cref{prop:cutcont}) perturb, and the comparison with the classical
cut-point construction (\Cref{rem:cutpoints}) is most transparent in this
form.

The existence of a gap in $(0,n)$ can be established without any control of
the total number of gap positions: it suffices to count gaps certified by
an event whose probabilities factorise exactly. For
$0\le i\le n-2$ define the \emph{cut-point certificate}
\begin{equation}
   \mathcal C_i:=\{R_j<i+1-j\ \text{for all }0\le j\le i\}\cap\{R_{i+1}<1\}.
\end{equation}
On $\mathcal C_i$ no vertex $j\le i$ can reach past $i+1$, and the vertex
$i+1$ connects to nothing, so no pair $(j,k)$ with $j\le i+1<k$ connects:
$\mathcal C_i\subseteq\{\text{gap at }i+1\}|_{(0,n)}$. The advantage of considering this event is that $\mathcal C_i$ is a \emph{cap event}: an
event of the form $\mathcal E_c:=\bigcap_v\{R_v<c_v\}$ determined by a
deterministic vector of \emph{caps} $c=(c_v)_v\in(0,\infty]^V$, one constraint
per vertex ($c_v=\infty$ meaning no constraint). Two features are used throughout.
Since the radii are independent and $\P(R<t)=1-x_m/t$ for $t\ge x_m$, the
probability of a cap event with all caps $\ge x_m$ is an exact product,
$\P(\mathcal E_c)=\prod_v(1-x_m/c_v)$ (every cap occurring below is $\ge1$);
and cap events are closed under
intersection, the caps combining vertex by vertex,
\[
   \mathcal E_c\cap\mathcal E_{c'}=\mathcal E_{c\wedge c'},
   \qquad (c\wedge c')_v=\min(c_v,c'_v),
\]
so joint probabilities are exact products as well. Explicitly $\mathcal C_i$
has caps $c^{(i)}_j=i+1-j$ for $0\le j\le i$, $c^{(i)}_{i+1}=1$ and
$c^{(i)}_v=\infty$ for $v\ge i+2$. With
\[
   G(k):=\prod_{d=1}^{k}\Big(1-\frac{x_m}{d}\Big)
   =\frac{\Gamma(k+1-x_m)}{\Gamma(1-x_m) k!}
   \ \sim\ \frac{k^{-x_m}}{\Gamma(1-x_m)},
\]
we have $\P(\mathcal C_i)=(1-x_m) G(i+1)$. Additionally, for $i<i'$ with $s:=i'-i$,
\begin{equation}\label{eq:cut_joint}
  \P(\mathcal C_i\cap\mathcal C_{i'})=\P(\mathcal C_i) (1-x_m) G(s-1),
\end{equation}
since the caps of $\mathcal C_i\cap\mathcal C_{i'}$ are the vertex-wise minima
of $c^{(i)}$ and $c^{(i')}$, namely $i+1-j$ for $j\le i$, then $1$ at $i+1$,
$i'+1-j$ for $i+2\le j\le i'$, and $1$ at $i'+1$. The first two groups
reproduce $\P(\mathcal C_i)$; in the third, the substitution $d=i'+1-j$ turns
$\prod_{j=i+2}^{i'}(1-x_m/(i'+1-j))$ into $\prod_{d=1}^{s-1}(1-x_m/d)=G(s-1)$,
and the fourth contributes $1-x_m$.

\begin{prop}\label{prop:cutpoints}
Let $N_{\mathcal C}:=\#\{0\le i\le n-2:\ \mathcal C_i\ \text{occurs}\}$. Then
\[
   \E[N_{\mathcal C}]=(1-x_m)\sum_{k=1}^{n-1}G(k)
   \ \sim\ \frac{n^{1-x_m}}{\Gamma(1-x_m)}\ \longrightarrow\ \infty
   \qquad(x_m<1),
\]
\[
   \E[N_{\mathcal C}^2]\ \le\ 2\big(\E N_{\mathcal C}\big)^2
   +3 \E N_{\mathcal C},
\]
and consequently, for all $n\ge2$,
\[
  \P\big(\exists i\in[1,n-1]:\ \{\text{gap at }i\}|_{(0,n)}\big)\ \ge\
  \P(N_{\mathcal C}\ge1)
  \ \ge\ \frac{1}{2+3/\E N_{\mathcal C}}\ \ge\ c(x_m)\ >\ 0,
\]
with $\liminf_{n\to\infty}\P(N_{\mathcal C}\ge1)\ge\tfrac12$.
\end{prop}

\begin{proof}
Summing $\P(\mathcal C_i)=(1-x_m)G(i+1)$ over
$0\le i\le n-2$ and substituting $k=i+1$,
\[
   \E N_{\mathcal C}=\sum_{i=0}^{n-2}\P(\mathcal C_i)
   =(1-x_m)\sum_{i=0}^{n-2}G(i+1)
   =(1-x_m)\sum_{k=1}^{n-1}G(k).
\]
Since $G(k)\sim k^{-x_m}/\Gamma(1-x_m)$, and
$\sum_{k=1}^{n-1}k^{-x_m}\sim n^{1-x_m}/(1-x_m)$ for $x_m<1$, the two factors
$1-x_m$ cancel and $\E N_{\mathcal C}\sim n^{1-x_m}/\Gamma(1-x_m)\to\infty$.

For the second moment, write $N_{\mathcal C}=\sum_{i=0}^{n-2}\1_{\mathcal C_i}$
and expand the square, splitting the double sum according to $i=i'$, $i<i'$
and $i>i'$ and using that the last two blocks are equal:
\[
   N_{\mathcal C}^2
   =\sum_{i}\sum_{i'}\1_{\mathcal C_i}\1_{\mathcal C_{i'}}
   =\sum_{i}\1_{\mathcal C_i}^2+2\sum_{i<i'}\1_{\mathcal C_i}\1_{\mathcal C_{i'}}
   =N_{\mathcal C}+2\sum_{i<i'}\1_{\mathcal C_i\cap\mathcal C_{i'}},
\]
where the last equality uses $\1_A^2=\1_A$ and $\1_A\1_B=\1_{A\cap B}$. The
diagonal therefore contributes $N_{\mathcal C}$, and taking
expectations,
\begin{equation}\label{eq:cut_secondmoment}
   \E[N_{\mathcal C}^2]=\E N_{\mathcal C}
   +2\!\!\sum_{0\le i<i'\le n-2}\!\!\P(\mathcal C_i\cap\mathcal C_{i'}).
\end{equation}
It remains to bound the off-diagonal sum. Fix $i$ and put $s=i'-i$; by
\eqref{eq:cut_joint}, and since $G\ge0$,
\[
  \sum_{i'=i+1}^{n-2}\P(\mathcal C_i\cap\mathcal C_{i'})
  =\P(\mathcal C_i) (1-x_m)\sum_{s=1}^{n-2-i}G(s-1)
  \ \le\ \P(\mathcal C_i) (1-x_m)\sum_{k=0}^{n-2}G(k).
\]
This last sum starts at $k=0$ with $G(0)=1$ (an empty product), whereas
$\E N_{\mathcal C}=(1-x_m)\sum_{k=1}^{n-1}G(k)$ starts at $k=1$; separating
the $k=0$ term,
\[
  (1-x_m)\sum_{k=0}^{n-2}G(k)
  =(1-x_m)+(1-x_m)\sum_{k=1}^{n-2}G(k)
  \ \le\ (1-x_m)+\E N_{\mathcal C}\ \le\ 1+\E N_{\mathcal C}.
\]
Summing over $i$ and using $\sum_i\P(\mathcal C_i)=\E N_{\mathcal C}$ gives
$2\sum_{i<i'}\P(\mathcal C_i\cap\mathcal C_{i'})
\le 2 \E N_{\mathcal C}(\E N_{\mathcal C}+1)$, so that, by
\eqref{eq:cut_secondmoment},
\[
  \E[N_{\mathcal C}^2]\ \le\ \E N_{\mathcal C}
  +2\E N_{\mathcal C}(\E N_{\mathcal C}+1)
  =2(\E N_{\mathcal C})^2+3\E N_{\mathcal C}.
\]

Since $N_{\mathcal C}=N_{\mathcal C}\1_{\{N_{\mathcal C}\ge1\}}$,
Cauchy-Schwarz gives
$\E N_{\mathcal C}\le(\E[N_{\mathcal C}^2])^{1/2} \P(N_{\mathcal C}\ge1)^{1/2}$,
that is, $\P(N_{\mathcal C}\ge1)\ge(\E N_{\mathcal C})^2/\E[N_{\mathcal C}^2]$
(the Paley-Zygmund inequality at $\theta=0$). Combined with the second-moment
bound,
\[
  \P(N_{\mathcal C}\ge1)\ \ge\
  \frac{(\E N_{\mathcal C})^2}{2(\E N_{\mathcal C})^2+3\E N_{\mathcal C}}
  =\frac{1}{2+3/\E N_{\mathcal C}} .
\]
As $\mathcal C_i\subseteq\{\text{gap at }i+1\}|_{(0,n)}$ with $i+1\in[1,n-1]$, the event
$\{N_{\mathcal C}\ge1\}$ is contained in $\{\exists i\in[1,n-1]:\text{gap at }i\}|_{(0,n)}$,
which gives the first bound. Finally $\E N_{\mathcal C}\ge(1-x_m)G(1)=(1-x_m)^2>0$
for every $n\ge2$, so the right-hand side is bounded below by a positive
$c(x_m)$ uniformly in $n$; and $\E N_{\mathcal C}\to\infty$ makes
$1/(2+3/\E N_{\mathcal C})\to\tfrac12$, giving the stated $\liminf$.
\end{proof}

\begin{lemma}\label{lem:cutmarked}
In the marked model of \Cref{def:marked}, define, for $n/3\le i\le 2n/3$,
\[
   \mathcal C^\sharp_i:=\{R^\sharp_j<i+1-j\ \text{for }1\le j\le i\}
   \cap\{R^\sharp_{i+1}<1\}
   \cap\{R^\sharp_k<k\ \text{for }i+2\le k\le n-1\}.
\]
Then $\mathcal C^\sharp_i\subseteq\{\text{gap}^\sharp\text{ at }i+1\}$, and
the certificate count
\[
   N^\sharp_{\mathcal C}:=\#\{i\in[n/3,2n/3]:\ \mathcal C^\sharp_i\
   \text{occurs}\}
\]
satisfies $\P(N^\sharp_{\mathcal C}\ge1)\ge c(x_m)>0$; consequently
\[
   \P\big(\exists i:\ \text{gap}^\sharp\text{ at }i\big)\ \ge\ c(x_m)\ >\ 0,
\]
uniformly in $n$ and in the marked parameters $m,s$. In particular, on the
rainbow event intersected with $B$, each one-sided disconnection event of
the reduction satisfies
$\P(\mathcal D^L), \P(\mathcal D^R)\ \ge\ c(x_m)$, uniformly over
rainbows.
\end{lemma}

\begin{proof}
We begin with the inclusion $\mathcal C^\sharp_i\subseteq
\{\text{gap}^\sharp\text{ at }i+1\}$, which we verify for every integer $i$
with $1\le i\le n-2$. Consider a pair $(j,k)$ with $0\le j\le i+1<k\le n+m$.
If $j=0$, an edge requires $R^\sharp_k\ge k$; this is excluded by the third
group of caps for $k\le n-1$, and is impossible for $k\ge n$, since every
marked radius is capped below its position ($R^\sharp_n<n$,
$R^\sharp_{n+\kappa}<D_\kappa\le n+\kappa$ for $1\le\kappa\le m-1$, and
$R^\sharp_{n+m}<s\le n+m$ by balance). If
$1\le j\le i$, then $R^\sharp_j<i+1-j<k-j$; if $j=i+1$, then
$R^\sharp_{i+1}<1\le k-j$. Hence no pair crosses the cut at $i+1$.

\begin{figure}[ht]\centering
\begin{tikzpicture}[xscale=1.02, yscale=1.32]
  \draw[zline] (-0.7,0) -- (10.7,0);
  \foreach \x/\lab in {0/0, 0.8/1, 2.1/j, 3.9/i, 4.7/{i+1}, 5.5/{i+2}, 7.6/{n-1}, 8.4/n, 10.1/{n+m}}
     {\fill (\x,0) circle (1.4pt); \node[below=3pt] at (\x,0) {\footnotesize $\lab$};}
  \draw[cut] (4.7,-0.5) -- (4.7,2.15);
  \node[gray!65!black] at (4.7,2.42) {\footnotesize cut at $i+1$};
  \draw[excluded] (0,0) to[bend left=45] coordinate[pos=0.8](A2) (10.1,0);
  \node[red!70!black, above=3pt] at (A2) {\footnotesize (a$_2$)};
  \draw[excluded] (0,0) to[bend left=40] coordinate[pos=0.72](A1) (7.6,0);
  \node[red!70!black, above=3pt] at (A1) {\footnotesize (a$_1$)};
  \draw[excluded] (2.1,0) to[bend left=46] coordinate[pos=0.30](Bm) (6.6,0);
  \node[red!70!black, above=3pt] at (Bm) {\footnotesize (b)};
  \draw[excluded] (4.7,0) to[bend left=70] coordinate[midway](Cm) (5.5,0);
  \node[red!70!black, above=3pt] at (Cm) {\footnotesize (c)};
  \draw[brc] (0.75,-0.62) -- (3.85,-0.62) node[midway,below=4pt]{\footnotesize $R^\sharp_j<i+1-j$};
  \draw[brc] (4.55,-0.62) -- (4.85,-0.62) node[midway,below=4pt]{\footnotesize $R^\sharp_{i+1}<1$};
  \draw[brc] (5.55,-0.62) -- (7.55,-0.62) node[midway,below=4pt]{\footnotesize $R^\sharp_k<k$};
  \draw[brc] (8.45,-0.62) -- (10.05,-0.62) node[midway,below=4pt]{\footnotesize position caps};
  \node[anchor=north west, align=left] at (-0.4,-1.55)
    {\footnotesize (a$_1$) $j=0$, $k\le n-1$: an edge needs $R^\sharp_k\ge k$, excluded by the third cap group\\[1pt]
     \footnotesize (a$_2$) $j=0$, $k\ge n$: an edge needs $R^\sharp_k\ge k$, excluded by the position caps\\[1pt]
     \footnotesize (b) $1\le j\le i$:\ \ $R^\sharp_j<i+1-j<k-j$\\[1pt]
     \footnotesize (c) $j=i+1$:\ \ $R^\sharp_{i+1}<1\le k-j$};
\end{tikzpicture}
\caption{The inclusion $\mathcal C^\sharp_i\subseteq\{\text{gap}^\sharp\text{
at }i+1\}$. A pair $(j,k)$ with $j\le i+1<k$ would have to span the dashed
cut. The four dashed arcs exhaust the possibilities: $j=0$, which splits
according to whether $k$ falls inside the overhang or beyond it, and the two
remaining ranges of $j$. Below the axis are the three cap groups of
$\mathcal C^\sharp_i$ and, on the right, the caps that every radius of the
marked model carries anyway. The vertex $0$ has $R^\sharp_0=\infty$ and is
capped by neither.}
\label{fig:cutmarked-inclusion}
\end{figure}

We turn to the probability bound, first treating the finitely many short
overhangs. Since $n>m\ge1$ in \Cref{def:marked}, we have $n\ge2$

First, suppose
$2\le n\le5$. On the event
$\mathcal E:=\{R^\sharp_j<1\ \text{for all }1\le j\le n-1\}$ there is a
gap$^\sharp$ at $1$: let $0\le j\le 1<k\le n+m$. If $j=1$, then
$R^\sharp_1<1\le k-1$. If $j=0$, then $R^\sharp_0=\infty$ and an edge only requires
$R^\sharp_k\ge k$, which fails for $2\le k\le n-1$ since $R^\sharp_k<1$ on
$\mathcal E$, and fails for $k\ge n$ by the caps on the marked radii listed
above. As $R^\sharp_1,\dots,R^\sharp_{n-1}$ are i.i.d.\ Pareto,
$\P(\mathcal E)=(1-x_m)^{n-1}\ge(1-x_m)^{4}$, which is the asserted bound in
this range.

Suppose from now on that $n\ge6$. Every integer of $[n/3,2n/3]$ is then an
admissible cut position, since $\lfloor 2n/3\rfloor\le n-2$, and there are at
least $n/3-1\ge n/9$ of them, so the count $N^\sharp_{\mathcal C}$ of the
statement runs over at least $n/9$ positions.

We first compute the one-point probabilities. Every cap of
$\mathcal C^\sharp_i$ is at a vertex of $\{1,\dots,n-1\}$, whose radii are
i.i.d.\ Pareto, so the probability is again an exact product (the marked
radii are unconstrained by $\mathcal C^\sharp_i$ and do not enter here):
\[
   \P(\mathcal C^\sharp_i)
   =(1-x_m) G(i)\prod_{k=i+2}^{n-1}\Big(1-\frac{x_m}k\Big)
   =(1-x_m) \frac{G(i) G(n-1)}{G(i+1)}
   =\frac{(1-x_m) G(n-1)}{1-x_m/(i+1)} ,
\]
where the last step used that $G(i)/G(i+1)=(1-x_m/(i+1))^{-1}$. Since
$1-x_m\le 1-x_m/(i+1)\le1$, we see that $\P(\mathcal C^\sharp_i)$ lies between $(1-x_m)G(n-1)$ and $G(n-1)$
whatever $i$ is, so it depends on the cut position only through a bounded
factor. As $G(n-1)\ge c(x_m) n^{-x_m}$ we obtain
$\P(\mathcal C^\sharp_i)\ge c(x_m) n^{-x_m}$ and hence
\[
   \E N^\sharp_{\mathcal C}\ \ge\ \frac n9 c(x_m) n^{-x_m}
   \ =\ c(x_m) n^{1-x_m}.
\]

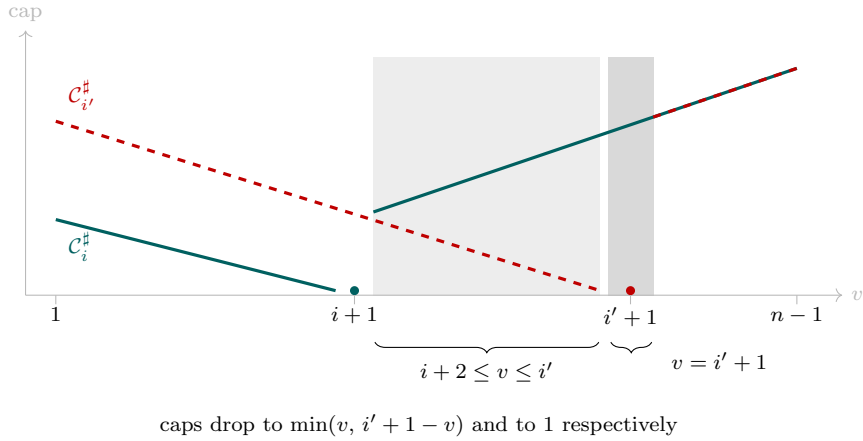
\begin{figure}[ht]\centering
\begin{tikzpicture}[xscale=1.0, yscale=1.0]
  \fill[gray!14] (4.3,0) rectangle (7.3,3.15);
  \fill[gray!30] (7.4,0) rectangle (8.0,3.15);
  \draw[->, gray!55] (-0.3,0) -- (10.5,0) node[right]{\footnotesize $v$};
  \draw[->, gray!55] (-0.3,0) -- (-0.3,3.5) node[above]{\footnotesize cap};
  \draw[innerarch] (0.1,1.00) -- (3.8,0.06);
  \draw[innerarch] (4.3,1.10) -- (9.9,3.00);
  \fill[teal!75!black] (4.05,0.06) circle (1.7pt);
  \draw[forced, dashed] (0.1,2.30) -- (7.3,0.06);
  \draw[forced, dashed] (8.0,2.355) -- (9.9,3.00);
  \fill[red!75!black] (7.7,0.06) circle (1.7pt);
  \node[red!75!black, anchor=south west] at (0.15,2.32) {\footnotesize $\mathcal C^\sharp_{i'}$};
  \node[teal!75!black, anchor=north west] at (0.15,0.96) {\footnotesize $\mathcal C^\sharp_i$};
  \foreach \x/\lab in {0.1/1, 4.05/{i+1}, 7.7/{i'+1}, 9.9/{n-1}}
     {\draw[gray!55] (\x,0) -- (\x,-0.12); \node[below=1pt] at (\x,0) {\footnotesize $\lab$};}
  \draw[brc] (4.32,-0.62) -- (7.28,-0.62)
     node[midway,below=4pt]{\footnotesize $i+2\le v\le i'$};
  \draw[brc] (7.44,-0.62) -- (7.98,-0.62);
  \node[anchor=west] at (8.12,-0.88) {\footnotesize $v=i'+1$};
  \node[anchor=north] at (4.9,-1.45)
     {\footnotesize caps drop to $\min(v,\,i'+1-v)$ and to $1$ respectively};
\end{tikzpicture}
\caption{The joint bound. Cap profiles of $\mathcal C^\sharp_i$ (solid) and
$\mathcal C^\sharp_{i'}$ (dashed) over the vertices $1,\dots,n-1$; each
descends to its pinch at $1$, then jumps back up and rises as $v$. The caps of
the intersection are the vertex-wise minimum, so they follow the lower of the
two curves. On $1\le v\le i+1$ that is the solid profile, and on
$v\ge i'+2$ the two coincide; only the two shaded groups change. Because
$i'\le2i$ over the range of cut positions used, the dashed profile is the
lower one across the whole of the first shaded group.}
\label{fig:cutmarked-joint}
\end{figure}

We next bound the joint probabilities. Fix $n/3\le i<i'\le2n/3$. The caps of
$\mathcal C^\sharp_i\cap\mathcal C^\sharp_{i'}$ are the vertex-wise minima of
the two cap vectors; comparing them with the caps of $\mathcal C^\sharp_i$
vertex by vertex, for $1\le v\le i$ the two caps are $i+1-v$ and $i'+1-v$,
with minimum $i+1-v$; at $v=i+1$ the minimum of $1$ and $i'-i$ is $1$; for
$i+2\le v\le i'$ the cap $v$ is replaced by $\min(v, i'+1-v)$; at $v=i'+1$ the
cap $i'+1$ is replaced by $1$; and for $i'+2\le v\le n-1$ both caps equal $v$.
Only the two middle groups change, so
\begin{multline*}
  \P(\mathcal C^\sharp_i\cap\mathcal C^\sharp_{i'})
  =\P(\mathcal C^\sharp_i)\cdot\frac{1-x_m}{1-x_m/(i'+1)}
  \prod_{v=i+2}^{i'}\frac{1-x_m/\min(v, i'+1-v)}{1-x_m/v}\\
  \le\ C(x_m) \P(\mathcal C^\sharp_i) (1-x_m) G(i'-i-1),
\end{multline*}
where we used that $1-x_m/\min(v,i'+1-v)\le1-x_m/(i'+1-v)$, so that after the
substitution $d=i'+1-v$ the numerators contribute
$\prod_{d=1}^{i'-i-1}(1-x_m/d)=G(i'-i-1)$, together with
$\prod_{v=i+2}^{i'}(1-x_m/v)^{-1}=G(i+1)/G(i')\le C(x_m)$ for
$n/3\le i<i'\le2n/3$.

It remains to sum. Exactly as in the proof of \Cref{prop:cutpoints},
$\E[(N^\sharp_{\mathcal C})^2]=\E N^\sharp_{\mathcal C}
+2\sum_{i<i'}\P(\mathcal C^\sharp_i\cap\mathcal C^\sharp_{i'})$, and by the
last display the inner sum over $i'>i$ is at most
$C(x_m) \P(\mathcal C^\sharp_i) (1-x_m)\sum_{k=0}^{n}G(k)$. Since
$(1-x_m)\sum_{k=0}^{n}G(k)\le C(x_m) n^{1-x_m}\le
C(x_m) \E N^\sharp_{\mathcal C}$ by the first-moment bound, summing over $i$
gives $2\sum_{i<i'}\P(\mathcal C^\sharp_i\cap\mathcal C^\sharp_{i'})
\le C(x_m)(\E N^\sharp_{\mathcal C})^2$. The diagonal term is of the same
order, because $\E N^\sharp_{\mathcal C}\ge c(x_m) n^{1-x_m}\ge c(x_m)$ for
$n\ge6$ and so $\E N^\sharp_{\mathcal C}\le
C(x_m)(\E N^\sharp_{\mathcal C})^2$. Hence
$\E[(N^\sharp_{\mathcal C})^2]\le C(x_m)(\E N^\sharp_{\mathcal C})^2$, and
Paley--Zygmund gives $\P(N^\sharp_{\mathcal C}\ge1)\ge c(x_m)>0$; with the
inclusion of the first paragraph this is the asserted bound for $n\ge6$.

It remains to transfer this to the disconnection events. On the rainbow event
intersected with $B$ the law of the
radii of $[a,d]$ is the marked law (see after \Cref{def:marked}),
and $\{\exists \text{gap}^\sharp\}\subseteq\mathcal D^L$; the mirror
argument gives $\mathcal D^R$.
\end{proof}

The final clause of \Cref{lem:cutmarked} completes the reduction of
\Cref{sec:reduction}.

\begin{remark}\label{rem:cutpoints}
\begin{enumerate}[(i)]\item \Cref{prop:cutpoints} is the classical cut-point construction of
one-dimensional $1/d^2$ percolation, cf.\ Newman and Schulman
\cite{NewmanSchulman1986}; in the radius model the certificate is a product
over vertices. 
\item The certified
positions form a strictly thinner set than the gaps:
$\E N_{\mathcal C}\asymp n^{1-x_m}$, while the number of gap positions is
expected to be of order $n^{1-x_m^2+o(1)}$; existence holds for
every $x_m<1$ either way, and the certified exponent $1-x_m>0$ is
the criticality condition of the model.
\end{enumerate}
\end{remark}

\subsection{All components are finite}\label{sec:main}

We now assemble the main result. Throughout this subsection $\{a,b\},\{c,d\}$
is a rainbow with the usual labels $\ell=c-a$, $m=d-c$,
$r=b-d$ (so $|\ell-r|\le m$, $R_a,R_b\ge b-a$, $R_c<\ell$, $R_d<r$), and for
cut positions $i\in(a,c)$ and $i'\in(d,b)$ we write
\[
  G^L(i):=\{\text{no pair in }[a,i]\times(i,d]\text{ connects}\},\qquad
  G^R(i'):=\{\text{no pair in }[c,i']\times(i',b]\text{ connects}\};
\]
in overhang coordinates these are the marked gap events of
\Cref{def:marked} and its mirror image.

\begin{lemma}\label{lem:confine}
Let $\{a,b\},\{c,d\}$ be a rainbow and suppose that $\mathcal N$
holds (no edge between the open overhangs $(a,c)$ and $(d,b)$) and that
$G^L(i)$ and $G^R(i')$ hold for some $i\in(a,c)$, $i'\in(d,b)$. Then, in the
full graph on $\Z$, the connected component of every vertex $v\in(i,i']$ is
contained in $(i,i']$. In particular every component meeting $(i,i']$ is
finite, and the inner and the outer arch lie in distinct components.
\end{lemma}

\begin{proof}
It suffices to show that no edge $\{u,w\}$ with $u\in(i,i']$ and
$w\notin(i,i']$ exists, since any path leaving $(i,i']$ contains such an
edge. Suppose $\{u,w\}\in E$. Exits beyond the window $[a,b]$ are handled
by the escape argument of \Cref{lem:window}: an edge $\{u,w\}$ with
$w\notin[a,b]$ forces $R_u$ to exceed the distance from $u$ to an outer
endpoint, hence (as $R_a,R_b\ge b-a$) forces an edge from $u$ to that endpoint;
in the cases below, the balance inequalities steer this to the endpoint
whose gap event it contradicts.

Consider first a right exit, $w>i'$.
If $u\in[c,i']$ and $w\le b$, the pair lies in $[c,i']\times(i',b]$,
contradicting $G^R(i')$. If $u\in[c,i']$ and $w>b$, then
$R_u\ge w-u>b-u$, and $R_b\ge b-a\ge b-u$, so $\{u,b\}\in E$ -- again a pair
in $[c,i']\times(i',b]$, contradicting $G^R(i')$. If $u\in(i,c)$ and
$w\in(i',b)\subseteq(d,b)$, the pair lies in $(a,c)\times(d,b)$,
contradicting $\mathcal N$. If $u\in(i,c)$ and $w\ge b$, then
$R_u\ge w-u\ge b-u\ge b-c=m+r\ge\ell>u-a$ (balance), and
$R_a\ge b-a\ge u-a$, so $\{a,u\}\in E$ -- a pair in $[a,i]\times(i,d]$ (as
$u\in(i,c)\subseteq(i,d]$), contradicting $G^L(i)$.

A left exit, $w\le i$, is the mirror image of a right exit: reflecting the
line exchanges $a\leftrightarrow b$, $c\leftrightarrow d$,
$\ell\leftrightarrow r$ and $G^L(i)\leftrightarrow G^R(i')$, and preserves
$\mathcal N$ and the balance inequalities, so the four cases above apply
with the roles of the two gap events exchanged.

Hence the component of every $v\in(i,i']$ is contained in
$(i,i']\subseteq(a,b)$, which is finite. Since $c,d\in(i,i']$ while
$a,b\notin(i,i']$, the two arches lie in distinct components.
\end{proof}

\Cref{lem:confine} strengthens, and in the final assembly replaces, the
reduction certificate of \Cref{sec:reduction}: there,
$B\cap\mathcal N\cap\mathcal D^L\cap\mathcal D^R$ shows only that the two
arches lie in distinct components, while here the same balance
inequalities, run from explicit cut positions, trap the component of
\emph{every} vertex between the cuts, which is the form \Cref{thm:main}
needs. The proof of \Cref{thm:main} accordingly uses only
\Cref{lem:confine} together with the certificate events of
\Cref{prop:scale}.

\begin{prop}\label{prop:scale}
For $k\ge1$ let $g=g_k=8^{2k-1}$, with the blocks of
\Cref{def:blocks}, and let $\widetilde{\cG}_k$ be the event that in each of
the four blocks exactly one vertex qualifies (radius $\ge32g$ in
$\mathsf a_k,\mathsf b_k$; radius in $[4g,5g]$ in
$\mathsf c_k,\mathsf d_k$) while every other block vertex has radius $<g$
(``quiet blocks''); the selected vertices are called $a,c,d,b$. Define
\[
   D_k:=\widetilde{\cG}_k\cap B_k\cap\mathcal N_k
        \cap\mathcal Q^L_k\cap\mathcal Q^R_k,
\]
where $B_k$, $\mathcal N_k$ are the containment and no-cross-edge events of
the reduction for the selected rainbow at scale $k$, and $\mathcal Q^L_k$ is the event
that some cut position $\tau$ in the middle third of $(a,c)$ carries the
cut-point certificate: $R_u<\tau+1-u$ for all non-block $u\in(a,\tau]$,
$R_{\tau+1}<1$, and $R_w<w-a$ for all non-block $w\in(\tau+1,c)$
($\mathcal Q^R_k$ is the mirror image in $(d,b)$). Thus $\mathcal Q^L_k$ is
the analogue in $(a,c)$ of the certificate event
$\{N^\sharp_{\mathcal C}\ge1\}=\{\exists\,\text{gap}^\sharp\}$ of
\Cref{lem:cutmarked}, and not of the disconnection event $\mathcal D^L$: it
provides an explicit cut position $\tau$, which is what \Cref{lem:confine}
requires. Then:
\begin{enumerate}
\item[(i)] on $D_k$ the hypotheses of \Cref{lem:confine} hold for some
  $i\in(a,c)$, $i'\in(d,b)$, and $(-g,g)\subseteq(c,d)\subseteq(i,i']$;
\item[(ii)] $\P(D_k)\ge\delta(x_m)>0$, uniformly in $k$;
\item[(iii)] the events
  $E_k:=\widetilde{\cG}_k\cap\mathcal N_k\cap\mathcal Q^L_k\cap\mathcal Q^R_k$,
  $k\ge1$, are independent, being measurable with respect to the radii in
  the disjoint annuli $\{u: g_k\le|u|<9g_k\}$.
\end{enumerate}
\end{prop}

\begin{proof}
On $\widetilde{\cG}_k$ the selected vertices $a,c,d,b$ form a rainbow: the
qualifying radii are those of \Cref{def:blocks}, and the verification in the
proof of \Cref{prop:rainbows-io} applies unchanged (the quiet-block
condition only removes further candidates). In particular
$\ell,r\in(6g,8g)$, $m\in[2g,4g)$ and $|\ell-r|\le m$.

(i) Take $i:=\tau$ from $\mathcal Q^L_k$. The overhang has length
$n=\ell\in(6g,8g)$ and $\tau$ lies in its middle third, so $\tau-a\ge2g$ and
$c-\tau\ge2g$. Block vertices inside $(a,c)$ lie within $g$ of $a$ (the rest of
$\mathsf a_k$) or within $g$ of $c$ (the rest of $\mathsf c_k$); on
$\widetilde{\cG}_k$ they are quiet, so they satisfy the corresponding caps
($R_u<g\le\tau+1-u$ for $u$ near $a$, and $R_w<g<w-a$ for $w$ near $c$). Hence
\emph{every} $u\in(a,\tau]$ has $R_u<\tau+1-u$ and \emph{every}
$w\in(\tau+1,c)$ has $R_w<w-a$. We now check that $G^L(\tau)$ holds, i.e.\ that no pair in
$[a,\tau]\times(\tau,d]$ connects: for $u\in(a,\tau]$ and $w>\tau$,
$R_u<\tau+1-u\le w-u$; for $u=a$ and $w\in(\tau,c)$, $R_w<w-a$ (resp.\
$R_{\tau+1}<1$); for $u=a$ and $w=c$, $R_c<\ell=c-a$; for $u=a$ and
$w\in(c,d)$, $R_w<\min(w-a,b-w)\le w-a$ by $B_k$; for $u=a$ and $w=d$,
$R_d<r\le\ell+m=d-a$ by balance. The mirror argument gives that $G^R(i')$ holds from
$\mathcal Q^R_k$, and the no-cross-edge hypothesis of \Cref{lem:confine} is
exactly $\mathcal N_k$, one of the events defining $D_k$. Finally
$c\le-g<g\le d$ and $i<c$, $i'>d$.

(ii) Write
\[
   \mathcal F_k:=\sigma\big(R_v:\ v\in\mathsf a_k\cup\mathsf c_k
   \cup\mathsf d_k\cup\mathsf b_k\big)
\]
for the $\sigma$-algebra generated by the block radii, and condition on it.
The event $\widetilde{\cG}_k$ is $\mathcal F_k$-measurable, and on it the
block radii determine the selected vertices $a,c,d,b$, and with them which
caps and vertex ranges define the events
$B_k,\mathcal N_k,\mathcal Q^L_k,\mathcal Q^R_k$. On $\widetilde{\cG}_k$
every condition these events impose on a quiet block vertex holds
automatically, since a radius below $g$ meets caps that all exceed $g$, so
conditionally on $\mathcal F_k$ each of the four is a decreasing event of
the non-block radii, which are i.i.d.\ and independent of $\mathcal F_k$. By Harris \cite{Harris1960},
\[
  \P(D_k\mid\mathcal F_k) \1_{\widetilde{\cG}_k}\ \ge\
  \1_{\widetilde{\cG}_k}
  \P(B_k\mid\mathcal F_k) \P(\mathcal N_k\mid\mathcal F_k)
  \P(\mathcal Q^L_k\mid\mathcal F_k) \P(\mathcal Q^R_k\mid\mathcal F_k).
\]
Here $\P(B_k\mid\mathcal F_k)\ge(4/9)^{2x_m}$ by \Cref{lem:contain} (the quiet
interior block vertices satisfy their caps automatically, and dropping their
factors only increases the product);
$\P(\mathcal N_k\mid\mathcal F_k)\ge
e^{-16x_m^2}$ by Harris, since the expected number of cross edges is at most
$x_m^2 \ell\sum_{d\ge m}d^{-2}\le 8x_m^2$; and
$\P(\mathcal Q^{L}_k\mid\mathcal F_k),
\P(\mathcal Q^R_k\mid\mathcal F_k)\ge c(x_m)$ by the
proof of \Cref{lem:cutmarked}. The estimate in question is the moment bound
for the certificate count, $\P(N^\sharp_{\mathcal C}\ge1)\ge c(x_m)$; the
final clause of that lemma, which concerns $\mathcal D^L$ and
$\mathcal D^R$, plays no role here. The restriction of the caps to
non-block vertices affects neither moment. Indeed, first, the
certificate of $\mathcal Q^L_k$ imposes caps on the non-block vertices only,
so each one-point probability is the exact product of \Cref{lem:cutmarked}
with the factors of the block vertices deleted; deleting factors can only
increase the product, so the lower bound
$\E N^\sharp_{\mathcal C}\ge c(x_m)\,n^{1-x_m}$ is preserved. Second, the
two-point ratio
$\P(\mathcal C^\sharp_i\cap\mathcal C^\sharp_{i'})/\P(\mathcal C^\sharp_i)$
involves caps only at the positions of $[i+2,i'+1]$, which for $i,i'$ in the
middle third of the overhang lie at distance greater than $2g-1\ge g$ from
both ends of $(a,c)$, hence outside the block zones (contained in the
$g$-neighbourhoods of $a$ and of $c$): that computation is unchanged. Both
moment bounds therefore hold with the same constants, and Paley--Zygmund
yields the same $c(x_m)$; the mirror statement applies to
$\mathcal Q^R_k$. Finally
\[
  \P(\widetilde{\cG}_k)\ \ge\
  \Big(g\cdot\frac{x_m}{32g} \Big(1-\frac{x_m}g\Big)^{g}\Big)^2
  \Big(g\cdot\frac{x_m}{20g} \Big(1-\frac{x_m}g\Big)^{g}\Big)^2
  \ \ge\ \frac{x_m^4}{32^2\cdot20^2} e^{-8x_m}\ >\ 0,
\]
uniformly in $k$, whence (ii) with
$\delta(x_m):=\frac{x_m^4}{32^2\cdot20^2}e^{-8x_m}(4/9)^{2x_m}
e^{-16x_m^2}c(x_m)^2$.

(iii) All vertices entering $E_k$ (the four blocks and the two overhangs)
have $|u|\in[g_k,9g_k)$, and $9g_k<g_{k+1}=64g_k$, so the annuli are
disjoint and the radii independent.
\end{proof}

\begin{prop}\label{prop:scalesio}
Almost surely $D_k$ occurs for infinitely many $k$.
\end{prop}

\begin{proof}
Write $D_k=E_k\cap B_k$ with $E_k$ as in \Cref{prop:scale}(iii). For $j<k$,
the event $D_j$ is measurable with respect to the radii in
$[-9g_j,9g_j]$, and $9g_j\le 9g_k/64<g_k$, so $D_j$ and $E_k$ are
independent and
\[
  \P(D_j\cap D_k)\ \le\ \P(D_j\cap E_k)\ =\ \P(D_j) \P(E_k).
\]
On the other hand, write
\[
   \mathcal A_k:=\sigma\big(R_u:\ g_k\le\abs u<9g_k\big)
\]
for the $\sigma$-algebra generated by the annulus radii, so that $E_k$ is
$\mathcal A_k$-measurable by \Cref{prop:scale}(iii). The vertices of $(c,d)$
lying in the annulus are block vertices, and on $\widetilde{\cG}_k$ they are
quiet and meet their containment caps automatically; conditionally on
$\mathcal A_k$ the event $B_k$ is therefore a product over the radii of the
vertices $\abs v<g_k$, which are independent of $\mathcal A_k$, and
\[
   \P(B_k\mid\mathcal A_k) \1_{\widetilde{\cG}_k}\ \ge\
   (4/9)^{2x_m} \1_{\widetilde{\cG}_k}
\]
by \Cref{lem:contain}. Since $E_k\subseteq\widetilde{\cG}_k$, integrating
this against $\1_{E_k}$ gives $\P(D_k)\ge(4/9)^{2x_m}\P(E_k)$ and hence
\[
  \P(D_j\cap D_k)\ \le\ (9/4)^{2x_m} \P(D_j) \P(D_k)
  \qquad\text{for all }j<k .
\]
Since $\sum_k\P(D_k)=\infty$ by \Cref{prop:scale}(ii), the Kochen--Stone
lemma \cite{KochenStone1964} gives
\[
  \P(D_k\ \text{i.o.})\ \ge\
  \limsup_{K\to\infty}
  \frac{\big(\sum_{k\le K}\P(D_k)\big)^2}{\sum_{j,k\le K}\P(D_j\cap D_k)}
  \ \ge\ (4/9)^{2x_m}\ >\ 0 .
\]
Finally, $\{D_k\ \text{i.o.}\}$ is unchanged if the radii of finitely many
vertices are modified: a vertex $v$ lies in the annulus of at most one
scale, so at most one $E_k$ is affected; and for any fixed value of $R_v$
the containment cap of $B_k$ at $v$, which is at least $8g_k-|v|$ once $v$
is interior to scale $k$, is satisfied for all but finitely many $k$. Hence
$\1_{\{D_k\text{ i.o.}\}}$ is measurable with respect to
$\sigma(R_u:|u|>M)$ for every $M$, and the Kolmogorov zero--one law shows
that $\P(D_k\ \text{i.o.})\in\{0,1\}$; being positive, it equals $1$.
\end{proof}

We can now prove the theorem on the simplified model stated in
\Cref{sec:simplified}.

\begin{proof}[Proof of \Cref{thm:main}]
By \Cref{prop:scalesio} there is an event of probability one on which $D_k$
occurs for infinitely many $k$. Fix a realisation in this event and a vertex
$v\in\Z$, and choose $k$ with $D_k$ and $g_k>|v|$. By \Cref{prop:scale}(i),
$v\in(-g_k,g_k)\subseteq(i,i']$, and by \Cref{lem:confine} the component of
$v$ is contained in $(i,i']\subseteq(-9g_k,9g_k)$, hence finite. Since this
holds simultaneously for every $v\in\Z$, there is no infinite component.
\end{proof}

\begin{remark}
For $x_m\ge1$ no radius falls below $1$, so every nearest-neighbour edge
is present and $\Z$ is a single infinite component. That connectivity is
an artefact of the fixed vertex spacing rather than percolation. The continuum model, to which the
subcritical statement is transferred in \Cref{sec:continuum}, percolates for
$\lambda\beta\ge31$ (\Cref{thm:super}); the structural reason the two models
diverge is isolated in \Cref{rem:superlit}.
\end{remark}

\begin{corollary}\label{cor:diameter}
Let $x_m\in(0,1)$. Almost surely every connected component is finite, and
yet
\[
   \sup\{\operatorname{diam} C:\ C\ \text{a connected component}\}=\infty,
\]
where $\operatorname{diam} C:=\max C-\min C$ denotes the diameter of the
vertex set of $C$ in $\Z$.
\end{corollary}

\begin{proof}
The first statement is \Cref{thm:main}. By \Cref{rem:edges-io-Z}, the
$\Z$-analogue of \Cref{prop:edges-io}, almost surely infinitely many distinct
edges $\{j,k\}$ with $j\le0<k$ are present. For each $L$ only finitely many
pairs $j\le0<k$ satisfy $k-j\le L$, so these edges have unbounded length.
The component containing such an edge contains both of its endpoints, hence
has diameter at least $k-j$. The supremum of the component diameters is
therefore almost surely infinite, although each individual component is
finite.
\end{proof}

\section{The continuum model}\label{sec:continuum}

We now return to the weight-dependent random connection model and prove
\Cref{thm:continuum}, by transferring \Cref{thm:main} to it. Throughout this section
$\mathcal X$ is a Poisson point process on $\R\times(0,1)$ of intensity
$\lambda>0$; a point $\mathbf x=(x,t)\in\mathcal X$ is identified with its
position $x$ and its radius $R_x:=\beta/t$, and two points are joined by an
edge iff $\abs{x-y}\le R_x\wedge R_y$, which is the connection rule of
\Cref{sec:setup}. Note that $R_x>\beta$ for every point, and that for
$r\ge\beta$ the points of radius greater than $r$ form a Poisson process of
intensity $\lambda\beta/r$ on the line.

The map $(x,t)\mapsto(\lambda x,t)$ sends $\mathcal X$ to a Poisson process
of intensity one and transforms the connection rule $(t\vee s)\abs{x-y}\le\beta$
into $(t\vee s)\abs{x'-y'}\le\lambda\beta$; it is an isomorphism of the two
random graphs. \emph{We therefore fix $\lambda=1$ and $\beta<1$ for the rest
of this section}; $\beta$ now plays the role that $x_m$ played in the
simplified model.

\Cref{thm:continuum} does not follow from \Cref{thm:main} by a
comparison argument. Coarse-graining $\R$ into boxes of side one and giving
box $i$ the maximal radius $R^*_i$ found in it produces radii with the
correct tail, but continuum connectivity only implies
$\abs{i-j}\le R^*_i\wedge R^*_j+1$ for the boxes, and
$R^*_i+1\ge\beta+1>1$ deterministically: the dominating box graph contains
every nearest-neighbour edge and is connected. Finer boxes fail the same way
at their own scale. As the Pareto tail exponent is exactly one, there is no
room in the exponent to absorb the additive error (the same criticality
that blocks a domination argument within the models also blocks one between
them). We therefore re-run the argument of
\Cref{sec:formation,sec:reduction,sec:cutpoints,sec:main} directly for the
Poisson model. Three ingredients need Poisson replacements: the
correlation inequality, the cut-point certificates and the zero--one law,
and they occupy most of this section. Everything else transfers verbatim or
becomes simpler.

\paragraph{Deterministic ingredients.} \Cref{prop:crossing} was
proved in the continuum to begin with. The proofs of \Cref{lem:window}, of
the reduction certificate of \Cref{sec:reduction} and of \Cref{lem:confine}
use only the symmetric rule $\abs{x-y}\le R_x\wedge R_y$, the ordering of the
positions and the balance inequalities, never that positions are integers;
they hold for every locally finite configuration of positions with radii,
and we use them for $\mathcal X$ without further comment. We keep the
notation of \Cref{sec:main}: for a rainbow $\{a,b\},\{c,d\}$ formed
by points of $\mathcal X$ and cut positions $i\in(a,c)$, $i'\in(d,b)$, the
events $G^L(i)$, $G^R(i')$ and $\mathcal N$ refer to pairs of points of
$\mathcal X$ (outer endpoints included). Since $\mathcal X$ is almost surely
locally finite, a component whose positions are confined to a bounded
interval is finite.

\paragraph{Poisson toolbox.} We use three standard facts (see
\cite{LastPenrose2017}): restrictions of $\mathcal X$ to disjoint measurable
regions of $\R\times(0,1)$ are independent; conditionally on the positions,
the marks are independent uniforms (marking theorem); and for measurable
$S\subseteq\R\times(0,1)$,
\[
   \P\big(\mathcal X\cap S=\emptyset\big)=e^{-\abs S},
\]
with $\abs S$ the Lebesgue measure of $S$. All certificates below are void
events of explicit regions, so their probabilities and joint probabilities
are exact exponentials (the continuum counterpart of the exact products of
\Cref{prop:cutpoints}). The role of the discrete tail sums is played by
\[
   h(s):=\int_0^s\P(R>u) \mathrm du
   =\begin{cases} s,& 0\le s\le\beta,\\[2pt]
     \beta\big(1+\log(s/\beta)\big),& s\ge\beta,\end{cases}
\]
so that $e^{-h(s)}=e^{-\beta}(\beta/s)^{\beta}$ for $s\ge\beta$.

We call an event $A$ of the marked Poisson process \emph{decreasing} if it is
preserved under removing points: if $\mu'\in A$ and $\mu\subseteq\mu'$, then
$\mu\in A$. All certificates below ($B$, $\mathcal N$ and the cut
certificates) are decreasing: each forbids the presence of points in a given
region of $\R\times(0,1)$.

\begin{lemma}\label{lem:fkg}
Let $A_1,\dots,A_k$ be decreasing events of $\mathcal X$. Then
\[
   \P(A_1\cap\dots\cap A_k)\ \ge\ \prod_{i=1}^k\P(A_i).
\]
The same holds for the restriction of $\mathcal X$ to any measurable region,
in particular conditionally on the configuration in the complementary
region.
\end{lemma}

\begin{proof}
For $k=2$ this is Theorem~20.4 of Last and Penrose \cite{LastPenrose2017},
the Harris--FKG inequality for Poisson processes on a general
$\sigma$-finite phase space (here $\R\times(0,1)$ with Lebesgue intensity) applied with their set $B=\emptyset$, so that both indicator functions
are decreasing everywhere; indicators are bounded, hence square-integrable.
For $k>2$ iterate, noting that an intersection of decreasing events is
decreasing. For the second statement, the restriction of $\mathcal X$ to the
complement of a region $S_0$ is itself a Poisson process and is independent
of $\mathcal X\cap S_0$, so the inequality applies to it conditionally on
$\mathcal X\cap S_0$.
\end{proof}

\paragraph{Cut certificates.} The following is the Poisson counterpart of
\Cref{prop:cutpoints} and \Cref{lem:cutmarked}. For
$a<c$ with $\ell:=c-a$, measurable $W\subseteq(a,c)$ and $\tau\in(a,c)$
define
\begin{equation}\label{eq:cutcont}
  \mathcal C^W_\tau:=\{\text{no point } u\in W\cap(a,\tau]\ \text{with}\ R_u\ge\tau-u\}
  \cap\{\text{no point } u\in W\cap(\tau,c)\ \text{with}\ R_u\ge u-a\}.
\end{equation}
This is stated directly in the one-sided geometry in which it is used: $a$ is
the position of the outer endpoint, whose radius never binds inside the
window, and the second group of caps prevents edges from $a$ that jump the
cut. The set $W$ allows the certificate to ignore designated zones near the
ends of the overhang (in the application: the block zones, whose points are
controlled separately).
Since $R_u\ge\beta$ always, the first group of caps forbids \emph{any} point
of $W\cap(\tau-\beta,\tau]$: an empty buffer of length $\beta$ replaces the
cap $R_{i+1}<1$ of \Cref{prop:cutpoints}. This is the only place where
$\beta<1$ enters the construction, through $1-\beta>0$ in
\eqref{eq:ENcont} below, exactly as $x_m<1$ entered \Cref{prop:cutpoints}.

\begin{prop}\label{prop:cutcont}
Let $a<c$ with $\ell=c-a\ge48$, and let $W\subseteq(a,c)$ be measurable with
$[a+\ell/3, a+2\ell/3]\subseteq W$. Let
\[
   N^W:=\#\big\{j\in\Z:\ \tau_j:=a+j\in[a+\ell/3, a+2\ell/3]\ \text{and}\
   \mathcal C^W_{\tau_j}\ \text{occurs}\big\}.
\]
Then
\begin{equation}\label{eq:ENcont}
  \E N^W\ \ge\ c_0(\beta) \ell^{1-\beta},
  \qquad
  \E\big[(N^W)^2\big]\ \le\ C_0(\beta) \big(\E N^W\big)^2,
  \qquad
  \P(N^W\ge1)\ \ge\ c(\beta)>0,
\end{equation}
uniformly in $\ell$ and $W$. Moreover $\{N^W\ge1\}$ is decreasing and
measurable with respect to $\mathcal X\cap(W\times(0,1))$.
\end{prop}

\begin{proof}
The event $\mathcal C^W_\tau$ is the void event of
\[
  U^W_\tau:=\big\{(u,t):u\in W\cap(a,\tau],\ t\le 1\wedge\tfrac{\beta}{\tau-u}\big\}
   \cup \big\{(u,t):u\in W\cap(\tau,c),\ t\le\tfrac{\beta}{u-a}\big\},
\]
using $R_u\ge r\Leftrightarrow t\le\beta/r$ and, in the second group,
$u-a\ge\ell/3\ge16>\beta$. As $\mathcal X$ is Poisson of unit intensity,
$\P(\mathcal C^W_\tau)=e^{-\abs{U^W_\tau}}$. Moreover an
intersection of void events is the void event of the union of their
regions, so the joint probabilities are exact as well:
\[
   \P\big(\mathcal C^W_\tau\cap\mathcal C^W_{\tau'}\big)
   \ =\ e^{-\abs{U^W_\tau\cup U^W_{\tau'}}}
   \qquad\text{for all cut positions }\tau,\tau' .
\]
This is the continuum counterpart of the cap events of
\Cref{sec:cutpoints}, with the union of regions in place of the
vertex-wise minimum of caps.

For the one-point bound, removing the restriction to $W$ only enlarges the
region, so for $\tau=\tau_j$,
\[
  \abs{U^W_\tau}\ \le\ \int_0^{\tau-a}\Big(1\wedge\frac{\beta}s\Big)\mathrm ds
  +\int_{\tau-a}^{\ell}\frac{\beta}{s} \mathrm ds
  \ =\ h(\tau-a)+\beta\log\frac{\ell}{\tau-a}
  \ \le\ h(\tau-a)+\beta\log3,
\]
whence, by $e^{-h(x)}=e^{-\beta}(\beta/x)^\beta$ and $\tau-a\le\ell$,
\[
  \P(\mathcal C^W_\tau)\ \ge\ 3^{-\beta}e^{-\beta}\beta^{\beta} \ell^{-\beta}
  \ =:\ c_1(\beta) \ell^{-\beta}.
\]
The grid contains at least $\ell/3-2\ge\ell/6$ positions, so
$\E N^W\ge(c_1(\beta)/6) \ell^{1-\beta}=:c_0(\beta) \ell^{1-\beta}$.

For the two-point bound, let $\tau=\tau_j<\tau'=\tau_{j'}$ and
$s:=\tau'-\tau\ge1>\beta$. The strip $(\tau,\tau']$ lies in the middle third
of $(a,c)$, hence in $W$; there the region $U^W_{\tau'}$ contains
$\{(u,t):u\in(\tau,\tau'],\ t\le1\wedge\beta/(\tau'-u)\}$, of measure
$h(s)$, while $U^W_\tau$ meets the strip only in its second group
$\{t\le\beta/(u-a)\}$, of measure
$\beta\log\frac{\tau'-a}{\tau-a}\le\beta\log2$, because
$\tau'-a\le2\ell/3\le2(\tau-a)$. Therefore
$\abs{U^W_\tau\cup U^W_{\tau'}}\ge\abs{U^W_\tau}+h(s)-\beta\log2$, i.e.
\[
  \P\big(\mathcal C^W_\tau\cap\mathcal C^W_{\tau'}\big)
  \ \le\ 2^{\beta} \P(\mathcal C^W_\tau) e^{-h(s)} .
\]
We now sum. Writing $N^W=\sum_j\1_{\mathcal C^W_{\tau_j}}$ and expanding the
square as in the proof of \Cref{prop:cutpoints}, the diagonal terms
contribute $N^W$ itself, so
\[
  \E\big[(N^W)^2\big]=\E N^W
  +2\sum_{j<j'}\P\big(\mathcal C^W_{\tau_j}\cap\mathcal C^W_{\tau_{j'}}\big).
\]
For fixed $j$ the positions $\tau_{j'}$ with $j'>j$ are at distances
$s=1,2,\dots$ along the grid, at most $\lceil\ell\rceil$ of them, so the last
display and
\[
  \sum_{s=1}^{\lceil\ell\rceil}e^{-h(s)}
  =e^{-\beta}\beta^\beta\sum_{s=1}^{\lceil\ell\rceil}s^{-\beta}
  \le C_2(\beta) \ell^{1-\beta}
\]
give $\sum_{j'>j}\P(\mathcal C^W_{\tau_j}\cap\mathcal C^W_{\tau_{j'}})
\le2^{\beta}C_2(\beta) \P(\mathcal C^W_{\tau_j}) \ell^{1-\beta}$. Summing
over $j$,
\[
  \E\big[(N^W)^2\big]\ \le\ \E N^W+2^{1+\beta} C_2(\beta) \E N^W \ell^{1-\beta} .
\]
Both terms are at most a constant multiple of $(\E N^W)^2$. Indeed the
one-point bound gives $\ell^{1-\beta}\le\E N^W\!/c_0(\beta)$, which handles
the second; and since $\ell\ge48$ and $\beta<1$ we have $\ell^{1-\beta}\ge1$,
so $\E N^W\ge c_0(\beta)$ and hence
$\E N^W\le(\E N^W)^2/c_0(\beta)$, which handles the first. Therefore
\[
  \E\big[(N^W)^2\big]
  \ \le\ \frac{1+2^{1+\beta}C_2(\beta)}{c_0(\beta)} \big(\E N^W\big)^2 .
\]
The Paley--Zygmund inequality gives
$\P(N^W\ge1)\ge(\E N^W)^2/\E[(N^W)^2]\ge c(\beta)$. Finally each
$\mathcal C^W_{\tau_j}$ is the void event of a subset of $W\times(0,1)$, so
$\{N^W\ge1\}$ is a finite union of such events: decreasing and measurable
with respect to $\mathcal X\cap(W\times(0,1))$.
\end{proof}

\begin{remark}\label{rem:cutcont}
The correspondence with \Cref{sec:cutpoints} is exact:
$e^{-h(\tau-a)}=e^{-\beta}(\beta/(\tau-a))^{\beta}$ replaces
$(1-x_m)G(i+1)\sim i^{-x_m}/\Gamma(1-x_m)$, certified cuts appear at density
$\ell^{-\beta}$, and their count diverges precisely when $\beta<1$.
\end{remark}

\begin{lemma}\label{lem:containcont}
For a rainbow with labels $\ell,m,r$ (so $\ell,r>m$, the windows of \Cref{def:rainbow} being non-empty) and
$\ell\wedge r\ge\beta$, the containment event
$B=\set{R_v<\min(v-a, b-v)\ \text{for every point }v\in(c,d)}$ satisfies
\[
   \P(B)\ =\ \exp\Big(-\beta\int_0^m\frac{\mathrm ds}{\min(\ell+s, r+m-s)}\Big)
   \ \ge\ \Big(\tfrac49\Big)^{\beta}.
\]
\end{lemma}

\begin{proof}
$B$ is the void event of
$\{(u,t):u\in(c,d),\ t\le\beta/D(u-c)\}$ with
$D(s)=\min(\ell+s,r+m-s)\ge\ell\wedge r\ge\beta$, which gives the formula.
Each $1/D(s)$ is nonincreasing in $\ell$ and in $r$, so the integral is
largest at $\ell=r=m$, where, by the symmetry $s\mapsto m-s$,
\[
  \int_0^m\frac{\mathrm ds}{\min(m+s, 2m-s)}
  \ =\ 2\int_0^{m/2}\frac{\mathrm ds}{m+s}\ =\ 2\log\tfrac32\ =\ \log\tfrac94 .
  \qedhere
\]
\end{proof}

\paragraph{Disconnection at a fixed scale.} Fix $k\ge1$, $g=g_k=8^{2k-1}$
and the blocks of \Cref{def:blocks},
\[
   \mathsf a_k=(-9g,-8g],\quad \mathsf c_k=(-2g,-g],\quad
   \mathsf d_k=[g,2g),\quad \mathsf b_k=[8g,9g),
\]
now regarded as intervals carrying the point process. Let
$\widetilde{\cG}_k$ be the event that each block contains exactly one point
with radius in its qualifying window ($\ge32g$ for
$\mathsf a_k,\mathsf b_k$; in $[4g,5g]$ for $\mathsf c_k,\mathsf d_k$) and every
other point has radius $< g$ (``quiet blocks"); the four qualifiers are called $a,c,d,b$. As
in \Cref{prop:rainbows-io}, on $\widetilde{\cG}_k$ these points form a
rainbow with $m\in(2g,4g)$ and $\ell,r\in(6g,8g)$.

\begin{prop}\label{prop:scalecont}
Define
\[
   D_k:=\widetilde{\cG}_k\cap B_k\cap\mathcal N_k
        \cap\mathcal Q^L_k\cap\mathcal Q^R_k,
\]
where $B_k$ and $\mathcal N_k$ are the containment and no-cross-edge events
of the selected rainbow, $\mathcal Q^L_k:=\{N^{W_L}\ge1\}$ is the
certified-cut event of \Cref{prop:cutcont} for the overhang $(a,c)$ with
$W_L:=(-8g,-2g]$, and $\mathcal Q^R_k$ is its mirror image in $(d,b)$ with
$W_R:=[2g,8g)$. Then:
\begin{enumerate}
\item[(i)] on $D_k$ the hypotheses of \Cref{lem:confine} hold for some
  $i\in(a,c)$, $i'\in(d,b)$, and $(-g,g)\subseteq(c,d)\subseteq(i,i']$;
\item[(ii)] $\P(D_k)\ge\delta(\beta)>0$, uniformly in $k$;
\item[(iii)] the events
  $E_k:=\widetilde{\cG}_k\cap\mathcal N_k\cap\mathcal Q^L_k\cap\mathcal Q^R_k$,
  $k\ge1$, are independent, $E_k$ being measurable with respect to the
  restriction of $\mathcal X$ to the annulus
  $A_k:=\{u:g_k\le\abs u<9g_k\}\times(0,1)$.
\end{enumerate}
\end{prop}

\begin{proof}
(i) Note first that the middle third of $(a,c)$ is contained in
$(-7g,-3g)\subseteq W_L$: indeed $a+\ell/3\ge-9g+2g=-7g$ and
$a+2\ell/3=c-\ell/3\le-g-2g=-3g$, so \Cref{prop:cutcont} applies with
$W=W_L$ (and $\ell>6g\ge48$). Let $\tau$ be a certified cut position of
$\mathcal Q^L_k$; then $\tau-a\ge\ell/3>2g$ and $c-\tau\ge\ell/3>2g$. We
verify $G^L(\tau)$, i.e.\ that no pair of points in $[a,\tau]\times(\tau,d]$
is joined. A point $u\in(a,\tau]$ with $u\in W_L$ has $R_u<\tau-u<w-u$ for
every $w>\tau$ by the certificate. A block point $u\in(a,\tau]$ lies in
$(a,-8g]$ (the $\mathsf c_k$-zone lies to the right of $\tau$, since
$\tau\le c-2g<-3g$), so $u-a\le g$ and, being quiet, $R_u<g\le\tau-u$. It
remains to exclude edges from $a$: for $w\in(\tau,c)\cap W_L$ the
certificate gives $R_w<w-a$; for a quiet block point $w\in(-2g,c)$ we have
$R_w<g<6g<w-a$; for $w=c$, $R_c<\ell=c-a$ (rainbow); for a point
$w\in(c,d)$, $R_w<\min(w-a,b-w)\le w-a$ by $B_k$; and for $w=d$,
$R_d<r\le\ell+m=d-a$ by balance. Hence $G^L(\tau)$ holds; the mirror
argument gives $G^R(\tau')$ for a certified position $\tau'$ of
$\mathcal Q^R_k$, and $\mathcal N_k$ holds by definition. Finally
$c\le-g<g\le d$ and $\tau<c$, $\tau'>d$.

(ii) Write
\[
   \mathcal F_k:=\sigma\big(\mathcal X\cap
   ((\mathsf a_k\cup\mathsf c_k\cup\mathsf d_k\cup\mathsf b_k)\times(0,1))\big)
\]
for the $\sigma$-algebra generated by the configuration in the four blocks,
and condition on it; the rest of $\mathcal X$ is an independent Poisson
process. The event $\widetilde{\cG}_k$ is $\mathcal F_k$-measurable, and on
it the qualifiers $a,c,d,b$ are determined. On
$\widetilde{\cG}_k$ the quiet block points satisfy all caps that concern
them: those in $(c,d)$ have $R<g<\ell\wedge r\le\min(v-a,b-v)$, so they
satisfy the caps of $B_k$; and having radius $<g<m$ they cannot carry a
cross edge, so $\mathcal N_k$ reduces to the pairs of non-block points.
Conditionally, $B_k$, $\mathcal N_k$, $\mathcal Q^L_k$, $\mathcal Q^R_k$ are
therefore decreasing events of the non-block configuration, and
\Cref{lem:fkg} gives
\[
  \P(D_k\mid\mathcal F_k) \1_{\widetilde{\cG}_k}\ \ge\
  \1_{\widetilde{\cG}_k} 
  \P(B_k\mid\mathcal F_k) \P(\mathcal N_k\mid\mathcal F_k) 
  \P(\mathcal Q^L_k\mid\mathcal F_k) \P(\mathcal Q^R_k\mid\mathcal F_k).
\]
Here $\P(B_k\mid\mathcal F_k)\ge(4/9)^{\beta}$ by \Cref{lem:containcont}, since
restricting the void region to non-block positions only increases the
probability. For $\mathcal N_k$, condition further on the positions of the
non-block points: write
\[
   \mathcal P_k:=\mathcal F_k\vee\sigma\big(x:\ (x,t)\in\mathcal X,\
   x\notin\mathsf a_k\cup\mathsf c_k\cup\mathsf d_k\cup\mathsf b_k\big).
\]
Given $\mathcal P_k$ the radii are independent (marking theorem), the
no-edge events of the cross pairs are decreasing in them, and each pair at
distance $w-u>m$ connects with probability
$(\beta/(w-u))^2\le\tfrac12$, so by Harris \cite{Harris1960} and
$1-z\ge e^{-2z}$ on $[0,\tfrac12]$,
\[
  \P(\mathcal N_k\mid\mathcal P_k)\ \ge\
  \exp\Big(-2\beta^2\!\!\sum_{\text{cross pairs}}\!\!(w-u)^{-2}\Big);
\]
the sum is $\mathcal P_k$-measurable, so taking conditional expectations
given $\mathcal F_k$ and using Jensen and the Mecke equation, on
$\widetilde{\cG}_k$,
$\E\big[\sum_{\text{cross pairs}}(w-u)^{-2}\bigm|\mathcal F_k\big]
 \le\int_{(a,c)}\!\int_{(d,b)}(w-u)^{-2} \mathrm dw \mathrm du
 \le\ell/m\le4$, whence
$\P(\mathcal N_k\mid\mathcal F_k)\1_{\widetilde{\cG}_k}
 \ge e^{-8\beta^2}\1_{\widetilde{\cG}_k}$. The certified-cut events
satisfy $\P(\mathcal Q^{L}_k\mid\mathcal F_k),\P(\mathcal Q^R_k\mid\mathcal F_k)\ge
c(\beta)$ by \Cref{prop:cutcont}. Finally, in each block the points of
radius $\ge g$ form a Poisson process of mean $\beta$, of which the
qualifying window carries mean $\beta/32$ (outer blocks) resp.\
$\beta/20$ (inner blocks); ``exactly one qualifier and no other point of
radius $\ge g$'' therefore has probability $(\beta/32)e^{-\beta}$ resp.\
$(\beta/20)e^{-\beta}$, and, the blocks being disjoint,
\[
  \P(\widetilde{\cG}_k)=\Big(\frac{\beta}{32}\Big)^2\Big(\frac{\beta}{20}\Big)^2
  e^{-4\beta}\ >\ 0,
\]
uniformly in $k$. Together,
$\delta(\beta):=\big(\tfrac{\beta}{32}\big)^2\big(\tfrac{\beta}{20}\big)^2
e^{-4\beta}(4/9)^{\beta}e^{-8\beta^2}c(\beta)^2$.

(iii) The blocks, the overhangs $(a,c)\subseteq(-9g,-g)$ and
$(d,b)\subseteq(g,9g)$, and the sets $W_L,W_R$ all have positions of
absolute value in $[g,9g)$, and $9g_k<g_{k+1}=64g_k$: the annuli $A_k$ are
disjoint, so the restrictions of $\mathcal X$ to them are independent.
\end{proof}

\begin{prop}\label{prop:scalesiocont}
Almost surely $D_k$ occurs for infinitely many $k$.
\end{prop}

\begin{proof}
For $j<k$ the event $D_j$ is measurable with respect
to the restriction of $\mathcal X$ to $[-9g_j,9g_j]\times(0,1)$, and $E_k$
with respect to the annulus $A_k$; the regions are disjoint
($9g_j\le9g_k/64<g_k$), so
$\P(D_j\cap D_k)\le\P(D_j\cap E_k)=\P(D_j)\P(E_k)$. Conditionally on the
restriction of $\mathcal X$ to $A_k$, on $\widetilde{\cG}_k$ the caps of
$B_k$ at the quiet block points hold, and what remains of $B_k$ is the void
event of a region with positions in $(-g_k,g_k)$, disjoint from $A_k$; its
probability is at least $(4/9)^{\beta}$ by \Cref{lem:containcont}. Hence
$\P(D_k)\ge(4/9)^{\beta}\P(E_k)$ and
\[
  \P(D_j\cap D_k)\ \le\ (9/4)^{\beta} \P(D_j) \P(D_k),\qquad j<k .
\]
Since $\sum_k\P(D_k)=\infty$ by \Cref{prop:scalecont}(ii), the Kochen--Stone
lemma \cite{KochenStone1964} gives
$\P(D_k\text{ i.o.})\ge(4/9)^{\beta}>0$.

Set $\mathcal S_0:=[-9g_1,9g_1]\times(0,1)$ and
$\mathcal S_n:=\big([-9g_{n+1},9g_{n+1}]\setminus[-9g_n,9g_n]\big)\times(0,1)$
for $n\ge1$. The restrictions $\xi_n:=\mathcal X\cap\mathcal S_n$ are
independent and together generate $\mathcal X$. Fix $n$ and put
$M:=9g_{n+1}$. For $k$ with $g_k>M$ let $D^{(M)}_k$ be defined as $D_k$,
except that the caps of $B_k$ are imposed only on points with position
outside $[-M,M]$; since the annulus $A_k$ also lies outside $[-M,M]$, the
event $D^{(M)}_k$ is measurable with respect to $\sigma(\xi_m:m>n)$. The
central region $[-M,M]\times(0,1)$ contains almost surely finitely many
points, each of finite radius, while the cap of $B_k$ at a position
$v\in[-M,M]$ is at least $8g_k-M\to\infty$; hence there is a random $K$
such that for all $k\ge K$ every central point satisfies its scale-$k$ cap,
and $\1_{D_k}=\1_{D^{(M)}_k}$ for all $k\ge K$. Consequently
$\{D_k\text{ i.o.}\}=\{D^{(M)}_k\text{ i.o.}\}$ almost surely, and the
latter belongs to $\sigma(\xi_m:m>n)$. Thus $\{D_k\text{ i.o.}\}$ lies, up
to null sets, in the tail $\sigma$-field
$\bigcap_n\sigma(\xi_m:m>n)$ of the independent sequence $(\xi_n)$, which is
trivial by Kolmogorov's zero--one law; completion preserves triviality.
Being of positive probability, the event has probability one.
\end{proof}

\begin{proof}[Proof of \Cref{thm:continuum}]
By the scaling isomorphism it suffices to treat $\lambda=1$ and $\beta<1$.
By \Cref{prop:scalesiocont} there is an event of probability one on which
$D_k$ occurs for infinitely many $k$; intersect it with the
full-probability event that $\mathcal X$ is locally finite, and fix a
realisation. For a point $\mathbf v=(v,t)\in\mathcal X$ choose $k$ with
$D_k$ and $g_k>\abs v$. By \Cref{prop:scalecont}(i),
$v\in(-g_k,g_k)\subseteq(i,i']$, and by \Cref{lem:confine} the component of
$\mathbf v$ consists of points with positions in
$(i,i']\subseteq(-9g_k,9g_k)$, hence is finite. This holds simultaneously
for every point of $\mathcal X$, so there is no infinite component.
\end{proof}

\begin{corollary}\label{cor:diametercont}
Let $\lambda=1$ and $\beta<1$. Almost surely every component is finite and
yet $\sup\{\operatorname{diam}C:\ C\ \text{a component}\}=\infty$, where
$\operatorname{diam}C$ is the diameter of the set of positions of $C$.
\end{corollary}

\begin{proof}
By \Cref{prop:edges-io} there are almost surely infinitely many
origin-crossing edges. An origin-crossing edge of length at most $L$ has
both endpoints in $[-L,L]$, which contains finitely many points; hence the
lengths of the origin-crossing edges are unbounded, and the component of
such an edge has diameter at least its length. The rest is
\Cref{thm:continuum}.
\end{proof}

\begin{remark}\label{rem:contlit}
\begin{enumerate}[(i)] 
\item  The model is the boundary case $\gamma=0$ of the preferential
attachment family recalled before \Cref{thm:adrcm}.  No statement of
\cite{GracarLuechtrathMoerters2021} applies to the model: the standing
assumption there is $\gamma\in(0,1)$, and the profile is assumed
regularly varying at infinity, which the indicator profile is not.
\Cref{thm:continuum} shows that all components are finite for $\beta<1$;
the constant $1$ coincides with the classical constant of
\Cref{rem:lrp}(iii), and by the numerical study of \Cref{sec:numerics} it
does not appear to be sharp. For the upper bound the point is not the constant.
In $d=1$ with a profile of decay $\delta\ge2$ (indicator profile
included) the existence of a supercritical phase is left open in
\cite{GracarLuechtrathMoerters2021} (Remarks~(ii) and~(iv) following
Theorem~1.1); the effective decay is the boundary value $2$ (see (ii)
below), and the discrete counterpart of the model never percolates, at
any density (\Cref{thm:main}), so a supercritical phase could equally
well have been absent. \Cref{thm:super} shows that it is present, and
hence $\beta_c\in(0,\infty)$; by \Cref{rem:superlit}(i) it rests on a
density that only the continuum can reach.
\item The exponent criteria of
\cite{GracarLuechtrathMoench2022} decide finiteness of the percolation
threshold in $d=1$ whenever a certain effective decay differs from $2$;
for our kernel and profile,
$\int_0^1\!\int_0^1\1[(s\vee t)n\le1] \mathrm ds \mathrm dt=n^{-2}$,
the pseudo-scale-invariant boundary case in which those criteria
are silent. \Cref{thm:super} decides this boundary case, by exploiting
structure (the $\min$-rule and the exact scale invariance of
\Cref{rem:copies}) rather than decay exponents.
\item   The exponent $\gamma$ of \Cref{prop:gamma} moves the model off
this boundary. Computed as in Remark~1.2(ii) of
\cite{GracarLuechtrathMoench2022}, the integral of (ii) is of order
$n^{-2/\gamma}$ for $\gamma>1$ and vanishes for large $n$ when
$\gamma<1$, so the effective decay is $2/\gamma<2$ on one side of the
critical exponent and infinite on the other. The criteria apply here:
the standing assumptions of \cite{GracarLuechtrathMoench2022} ask of the
profile only monotonicity and integrability, which the indicator profile
satisfies, and Corollary~2.6 there removes a technical condition for
Poisson points. They give $\beta_c<\infty$ for $\gamma>1$ and
$\beta_c=\infty$ for $\gamma<1$; \Cref{prop:gamma} recovers the second
conclusion by an elementary comparison and sharpens the first to
$\beta_c=0$.

The comparison extends \Cref{thm:continuum,thm:super} to other kernels:
for a kernel $g$ with indicator profile in $d=1$, if $g(t,s)\le t\vee s$
pointwise then every edge of our graph is an edge of the $g$-graph on
the same points, and the $g$-graph has an infinite component whenever
$\lambda\beta\ge31$; if $g(t,s)\ge t\vee s$ then the $g$-graph is a
subgraph of ours, with no infinite component whenever $\lambda\beta<1$.
The first case with $g(t,s)=(t\vee s)^{1-\gamma}(t\wedge s)^{\gamma}$ is
the supercritical half of \Cref{thm:adrcm}.

The case $\gamma<\tfrac12$ of Theorem~1.4 of
\cite{GracarLuechtrathMoench2022}, recalled in \Cref{sec:setup},
concerns profiles of polynomial decay $\delta>2$; \Cref{thm:adrcm}
covers these profiles as well and closes this case, with the
supercritical phase provided by \Cref{cor:adrcmclass}.
 
\end{enumerate}
\end{remark}

\section{A supercritical phase in the continuum}\label{sec:super}

We now prove \Cref{thm:super}. Throughout this section $\X$ is again a
Poisson point process on $\R\times(0,1)$ of intensity $\lambda$; by the
scaling isomorphism recalled at the beginning of \Cref{sec:continuum} we may
and do fix $\lambda=1$ and assume $\beta\ge31$. As in \Cref{sec:continuum},
a point $\x=(x,t)\in\X$ is identified with its position $x$ and its radius
$R_x=\beta/t\in(\beta,\infty)$, and $\{\x,\y\}\in E$ iff
$\abs{x-y}\le R_x\wedge R_y$.

The proof is a renormalisation, but not a coarse-graining of space as in
Newman and Schulman \cite{NewmanSchulman1986}: we decompose the
\emph{mark space} into dyadic bands. The restrictions of
$\X$ to distinct bands are independent, and each band is a scaled copy of the same dense, bounded-range (hence subcritical)
one-dimensional Gilbert-type graph. The assembly rests on two observations.
First, by the $\min$-rule the connection threshold between a band-$k$ point
and \emph{any} point of larger radius exceeds $\beta2^k$, so consecutive
bands glue deterministically wherever the coarser band is locally dense
(\Cref{lem:link}); no analogue of \Cref{prop:crossing} and no correlation
inequality is needed, and every probabilistic estimate in this section is
an exact product over disjoint regions of $\R\times(0,1)$. Second, scale invariance forces the
failure probability of any single-scale event to stay bounded away from $0$
uniformly in the scale, so no fixed nested sequence of windows can survive
at all scales (\Cref{rem:tower}); the renormalisation must branch in space
as well as in scale. The incidence structure of dyadic windows across all
scales is the binary tiling of the hyperbolic half-plane, truncated at a
floor, and the assembly becomes a Peierls argument for independent site
percolation on this tiling.

\subsection{Mark bands}\label{sec:bands}

The starting point is the following decomposition.
For $k\in\N_0$ let
\[
   \X_k:=\big\{\x\in\X:\ R_x\in(\beta2^k, \beta2^{k+1}]\big\}
\]
be the \emph{band-$k$ points}. Since $R_x>\beta$ for every point, each point
lies in exactly one band. In mark coordinates, $R_x\in(\beta2^k,\beta2^{k+1}]$
means $t\in[2^{-k-1},2^{-k})$, so $\X_k$ is the restriction of $\X$
to $\R\times[2^{-k-1},2^{-k})$; by the restriction Theorem
\cite{LastPenrose2017} the processes $(\X_k)_{k\ge0}$ are independent, and
the positions of $\X_k$ form a Poisson process on $\R$ of intensity
$2^{-k-1}$.

\begin{lemma}\label{lem:bandedge}
Let $\x,\y\in\X$ with $R_x>\beta2^k$, $R_y>\beta2^k$ and
$\abs{x-y}\le\beta2^k$. Then $\{\x,\y\}\in E$. In particular a band-$k$
point is joined to every point of band $j\ge k$ at distance at most
$\beta2^k$ from it.
\end{lemma}

\begin{proof}
$\abs{x-y}\le\beta2^k<R_x\wedge R_y$.
\end{proof}

\begin{remark}\label{rem:copies}
The map $(x,t)\mapsto(2^kx, 2^{-k}t)$ preserves both the intensity measure
on $\R\times(0,1)$ and the connection rule $(t\vee s)\abs{x-y}\le\beta$, and
it maps band $0$ to band $k$. Each band is therefore a dilated copy of the
same graph: a Gilbert-type graph whose connection thresholds lie in
$(\beta2^k,\beta2^{k+1}]$. A single band has bounded range in one dimension,
hence almost surely no infinite component, for \emph{every} $\beta$; but
for large $\beta$ it is dense at its own scale, with $\beta/2$ expected
points per unit of connection range. The construction below uses
this densely-subcritical structure, available independently at
every scale. Scale invariance also has a cost, made precise in
\Cref{rem:tower}: no event tied to a single scale can have probability
tending to $1$ as the scale grows.
\end{remark}

\subsection{Windows, runs and deterministic linkage}\label{sec:windows}

We now renormalise. A \emph{site} is a dyadic window at some scale,
declared open when the corresponding band occupies every cell of the
window. This subsection shows that, once the site events are
given, all connectivity is automatic: an open site carries a chain of
band-$k$ points spanning its window (\Cref{lem:run}), and the chains of
neighbouring open sites (side by side at the same scale, or a window
and its parent) always lie in the same component of the graph
(\Cref{lem:link}). The randomness of the entire construction is thereby
confined to the site events, which are independent and open with a
probability that does not depend on the scale (\Cref{lem:density}).

For $k\in\N_0$ and $i\in\Z$ define the \emph{scale-$k$ window}
\[
   \Lambda_{k,i}:=\big[ 8\beta2^k i,\ 8\beta2^k(i+1) \big),
\]
subdivided into sixteen \emph{cells} of length $\beta2^k/2$. The windows
are dyadically nested, $\Lambda_{k+1,j}=\Lambda_{k,2j}\cup\Lambda_{k,2j+1}$,
each cell of $\Lambda_{k+1,j}$ has length $\beta2^k$, and each child window
is the disjoint union of eight aligned parent cells. Define the \emph{site
event}
\[
   O_{k,i}:=\big\{\text{every cell of }\Lambda_{k,i}\text{ contains the
   position of a band-}k\text{ point}\big\},
\]
and write $P_{k,i}$ for the set of band-$k$ points with position in
$\Lambda_{k,i}$.

\begin{lemma}\label{lem:density}
The events $\{O_{k,i}:k\in\N_0, i\in\Z\}$ are independent, and for all
$k,i$,
\[
   \P(O_{k,i})=\big(1-e^{-\beta/4}\big)^{16}
   \ \ge\ 1-16e^{-\beta/4}\ =:\ 1-\eps(\beta) .
\]
For $\beta\ge31$ one has $\eps(\beta)\le2^{-7}$.
\end{lemma}

\begin{proof}
$O_{k,i}$ is measurable with respect to the restriction of $\X$ to
$\Lambda_{k,i}\times[2^{-k-1},2^{-k})$, and these regions are pairwise
disjoint over all $(k,i)$, which gives the independence. The number of
band-$k$ points in a cell is Poisson with mean
$2^{-k-1}\cdot\beta2^k/2=\beta/4$, independently over the sixteen cells,
giving the product formula; the union bound gives the lower bound. Finally
$31/4>11\log2$, so $e^{-31/4}<2^{-11}$ and
$\eps(31)\le16\cdot2^{-11}=2^{-7}$.
\end{proof}

\begin{lemma}\label{lem:run}
On $O_{k,i}$: the positions of consecutive points of $P_{k,i}$ differ by
less than $\beta2^k$; all points of $P_{k,i}$ lie in one connected
component of the graph; and the leftmost and rightmost points of $P_{k,i}$
lie within $\beta2^k/2$ of the respective endpoints of $\Lambda_{k,i}$.
\end{lemma}

\begin{proof}
Let $x<y$ be the positions of two consecutive points of $P_{k,i}$, lying in
the cells with indices $m\le m'$. Every cell with index strictly between
$m$ and $m'$ is contained in $(x,y)$ and, on $O_{k,i}$, would contain the
position of a point of $P_{k,i}$; hence $m'\le m+1$ and
$y-x<2\cdot\beta2^k/2=\beta2^k$. By \Cref{lem:bandedge} consecutive points
are joined, so $P_{k,i}$ forms a chain. The extreme points lie in the first
and the last cell.
\end{proof}

\begin{lemma}\label{lem:link}
Let $k\in\N_0$, $i\in\Z$ and $j:=\lfloor i/2\rfloor$.
\begin{enumerate}
\item[(i)] On $O_{k,i}\cap O_{k,i+1}$, the sets $P_{k,i}$ and $P_{k,i+1}$
  lie in the same connected component of the graph.
\item[(ii)] On $O_{k,i}\cap O_{k+1,j}$, the sets $P_{k,i}$ and $P_{k+1,j}$
  lie in the same connected component of the graph.
\end{enumerate}
\end{lemma}

\begin{proof}
Both parts are illustrated in \Cref{fig:linkage}.

(i) The rightmost point of $P_{k,i}$ and the leftmost point of $P_{k,i+1}$
lie within $\beta2^k/2$ of the common window endpoint, hence within
$\beta2^k$ of each other, and both have radius exceeding $\beta2^k$;
\Cref{lem:bandedge} joins them.

(ii) Write $a$ for the left endpoint of $\Lambda_{k,i}$. Since
$\Lambda_{k,i}$ is a union of eight aligned parent cells, the interval
$J:=[a+\beta2^k, a+2\beta2^k)$ is a cell of $\Lambda_{k+1,j}$, and on
$O_{k+1,j}$ it contains the position $y$ of a point $\y\in P_{k+1,j}$. By
\Cref{lem:run} the leftmost point of $P_{k,i}$ has position at most
$a+\beta2^k/2<y$ and the rightmost has position at least
$a+8\beta2^k-\beta2^k/2>y$, so $y$ lies between the positions
$x_-\le y\le x_+$ of two consecutive points of $P_{k,i}$, and
$y-x_-\le x_+-x_-<\beta2^k$. Since $R_y>\beta2^{k+1}>\beta2^k$ and the
point at $x_-$ has radius exceeding $\beta2^k$, \Cref{lem:bandedge} joins
them.
\end{proof}

\begin{figure}[ht]\centering
\begin{tikzpicture}[xscale=0.79]
  \fill[orange!12] (1,-0.18) rectangle (2,1.32);
  \node[orange!60!black] at (1.5,-0.38) {\footnotesize $J$};
  \draw[zline] (-0.4,0) -- (16.7,0);
  \foreach \x in {0,8,16}{\draw[gray!80, thick] (\x,-0.22) -- (\x,0.22);}
  \foreach \x in {0.5,1,...,15.5}{\draw[gray!35] (\x,-0.09) -- (\x,0.09);}
  \foreach \m in {0,...,30}{
    \pgfmathsetmacro\pa{0.5*\m+0.25+0.07*(-1)^\m}
    \pgfmathsetmacro\pb{0.5*(\m+1)+0.25+0.07*(-1)^(\m+1)}
    \draw[teal!55!black, very thin] (\pa,0) to[bend left=55] (\pb,0);}
  \foreach \m in {0,...,31}{
    \pgfmathsetmacro\px{0.5*\m+0.25+0.07*(-1)^\m}
    \fill (\px,0) circle (1.05pt);}
  \draw[zline] (-0.4,1.15) -- (16.7,1.15);
  \foreach \x in {0,16}{\draw[gray!80, thick] (\x,0.93) -- (\x,1.37);}
  \foreach \x in {1,...,15}{\draw[gray!35] (\x,1.06) -- (\x,1.24);}
  \draw[gray!70, {Stealth[length=1.4mm]}-{Stealth[length=1.4mm]}]
     (11,0.30) -- (11.5,0.30) node[midway,above=0.5pt]{\scriptsize $\beta2^k\!/2$};
  \draw[gray!70, {Stealth[length=1.4mm]}-{Stealth[length=1.4mm]}]
     (11,0.80) -- (12,0.80) node[midway,above=0.5pt]{\scriptsize $\beta2^k$};
  \foreach \m in {0,...,14}{
    \pgfmathsetmacro\qa{\m+0.5+0.12*(-1)^\m}
    \pgfmathsetmacro\qb{(\m+1)+0.5+0.12*(-1)^(\m+1)}
    \draw[blue!45, very thin] (\qa,1.15) to[bend left=35] (\qb,1.15);}
  \foreach \m in {0,...,15}{
    \pgfmathsetmacro\qx{\m+0.5+0.12*(-1)^\m}
    \fill[blue!65!black] (\qx,1.15) circle (1.3pt);}
  \draw[escape] (1.38,1.15) -- (1.32,0);
  \node[orange!85!black] at (3.6,0.62) {\footnotesize $\abs{y-x_-}<\beta2^k$};
  \draw[forced] (7.68,0) to[bend left=55] (8.32,0);
  \node[red!75!black] at (8,0.52) {\footnotesize $<\beta2^k$};
  \node at (-1.0,0) {\footnotesize $\X_k$};
  \node at (-1.0,1.15) {\footnotesize $\X_{k+1}$};
  \draw[brc] (0,-0.6) -- (8,-0.6) node[midway,below=4pt]{\footnotesize $\Lambda_{k,2j}$};
  \draw[brc] (8,-0.6) -- (16,-0.6) node[midway,below=4pt]{\footnotesize $\Lambda_{k,2j+1}$};
  \draw[decorate, decoration={brace, amplitude=4pt}] (0,1.55) -- (16,1.55)
     node[midway,above=4pt]{\footnotesize $\Lambda_{k+1,j}$};
\end{tikzpicture}
\caption{Two
scale-$k$ windows inside their parent window; band-$k$ points on the lower
line (one per cell of length $\beta2^k/2$ on the site events), band-$(k+1)$
points on the upper line (one per parent cell of length $\beta2^k$). Both
cell grids are drawn: the parent cells are twice the child cells,
so eight aligned parent cells make up each child window.
Consecutive band-$k$ points are within $\beta2^k$ of each other and chain
into a run (teal); the runs of adjacent windows join across the boundary
(red); and the band-$(k+1)$ point in the parent cell $J$ lies inside the
child run's span, hence within $\beta2^k$ of a run point (orange). All
these edges are deterministic given the bands: the two radii exceed
$\beta2^k$ and the distance does not.}
\label{fig:linkage}
\end{figure}
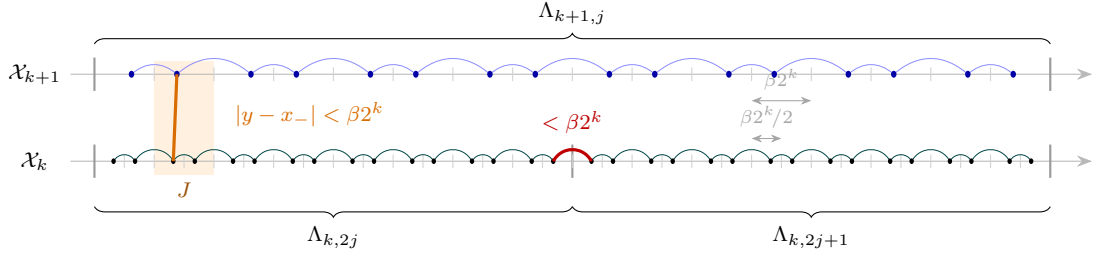

\begin{remark}
Part (ii) is where the $\min$-rule, the obstacle of the subcritical part,
works in our favour: the connection threshold between a band-$k$ point and
a band-$(k+1)$ point is governed by the \emph{smaller} of the two radii,
which still exceeds $\beta2^k$. A band-$k$ run therefore absorbs every
larger-radius point inside its span, and formation of the runs is the only
random ingredient of the construction.
\Cref{prop:crossing} would provide a probabilistic substitute for
(ii); it is not needed.
\end{remark}

\subsection{The renormalisation graph: a floored binary tiling}\label{sec:tiling}

Let $\mathsf H$ be the graph with vertex set
$V:=\{(k,i):k\in\N_0,\ i\in\Z\}$ and the edges
\[
  \{(k,i),(k,i+1)\}\quad\text{(horizontal)},\qquad
  \{(k,i),(k+1,\lfloor i/2\rfloor)\}\quad\text{(vertical)} .
\]
Both families are indexed by all of $V$, so each vertical edge is listed
once, from its lower endpoint: the vertex $(k,i)$ acquires one edge
upwards, to its parent $(k+1,\lfloor i/2\rfloor)$, and, if $k\ge1$, one
downwards to each of its two children $(k-1,2i)$ and $(k-1,2i+1)$.
Realising the vertex $(k,i)$ as the tile
$T_{k,i}:=\ov{\Lambda_{k,i}}\times[2^k,2^{k+1}]$ in the upper half-plane,
$\mathsf H$ is the adjacency graph of the binary tiling (sometimes called
the B\"or\"oczky tiling) of the hyperbolic half-plane, truncated at the
floor $k=0$: two tiles are $\mathsf H$-adjacent iff their boundaries share
a segment (see \Cref{fig:tiling}). Let $\mathsf H^\st$ be $\mathsf H$ together with the
\emph{diagonal} edges
\[
   \big\{(k,2j-1),(k+1,j)\big\}\quad\text{and}\quad
   \big\{(k,2j),(k+1,j-1)\big\},\qquad k\in\N_0,\ j\in\Z;
\]
two tiles are $\mathsf H^\st$-adjacent iff their boundaries share at least
a point. Site percolation on the unfloored tiling ($k\in\Z$) falls within
the theory of nonamenable planar graphs \cite{BenjaminiSchramm2001}; the
floor makes $\mathsf H$ amenable and that theory inapplicable, but a
Peierls argument suffices at the site densities we have.

\begin{figure}[ht]\centering
\begin{tikzpicture}[xscale=0.8, yscale=0.95]
  \foreach \i in {0,...,15}{\draw[gray!30] (\i,0) rectangle (\i+1,0.85);}
  \foreach \i in {0,...,7}{\draw[gray!30] (2*\i,0.85) rectangle (2*\i+2,1.8);}
  \foreach \i in {0,...,3}{\draw[gray!30] (4*\i,1.8) rectangle (4*\i+4,2.85);}
  \foreach \i in {0,...,1}{\draw[gray!30] (8*\i,2.85) rectangle (8*\i+8,4.0);}
  \draw[very thick] (0,0) -- (16,0);
  \node[below] at (8,0) {\footnotesize floor: $k=0$, radii in $(\beta,2\beta]$};
  \draw[gray!45, dashed] (0,4.0) rectangle (16,5.05);
  \draw[gray!75, dashed] (4,3.425) -- (8,4.52);
  \draw[gray!75, dashed] (12,3.425) -- (8,4.52);
  \fill[gray!60] (8,4.52) circle (1.3pt);
  \node[gray!70!black] at (8,5.42) {$\vdots$};
  \draw[-{Stealth[length=2mm]}] (-0.9,-0.15) -- (-0.9,5.55)
     node[above]{\footnotesize $k$};
  \node[left] at (-0.9,0.425) {\footnotesize $0$};
  \node[left] at (-0.9,1.325) {\footnotesize $1$};
  \node[left] at (-0.9,2.325) {\footnotesize $2$};
  \node[left] at (-0.9,3.425) {\footnotesize $3$};
  \node[left] at (-0.9,4.52) {\footnotesize $4$};
  \fill[teal!18] (0.5,0.425) -- (1.5,0.425) -- (1,1.325) -- cycle;
  \fill[blue!10] (1.5,0.425) -- (2.5,0.425) -- (3,1.325) -- (1,1.325) -- cycle;
  \foreach \i in {0,...,14}{\draw[black!75]
     (\i+0.5,0.425) to[bend right=28] (\i+1.5,0.425);}
  \foreach \i in {0,...,6}{\draw[black!75]
     (2*\i+1,1.325) to[bend right=22] (2*\i+3,1.325);}
  \foreach \i in {0,...,2}{\draw[black!75]
     (4*\i+2,2.325) to[bend right=16] (4*\i+6,2.325);}
  \draw[black!75] (4,3.425) to[bend right=11] (12,3.425);
  \foreach \i in {0,...,15}{
    \pgfmathsetmacro\pj{2*floor(\i/2)+1}
    \draw[black!75] (\i+0.5,0.425) -- (\pj,1.325);}
  \foreach \i in {0,...,7}{
    \pgfmathsetmacro\pj{4*floor(\i/2)+2}
    \draw[black!75] (2*\i+1,1.325) -- (\pj,2.325);}
  \foreach \i in {0,...,3}{
    \pgfmathsetmacro\pj{8*floor(\i/2)+4}
    \draw[black!75] (4*\i+2,2.325) -- (\pj,3.425);}
  \begin{scope}
    \clip (0,-0.5) rectangle (16,5.6);
    \draw[black!75, dashed] (0.5,0.425) to[bend left=28] (-0.5,0.425);
    \draw[black!75, dashed] (15.5,0.425) to[bend right=28] (16.5,0.425);
    \draw[black!75, dashed] (1,1.325) to[bend left=22] (-1,1.325);
    \draw[black!75, dashed] (15,1.325) to[bend right=22] (17,1.325);
    \draw[black!75, dashed] (2,2.325) to[bend left=16] (-2,2.325);
    \draw[black!75, dashed] (14,2.325) to[bend right=16] (18,2.325);
    \draw[black!75, dashed] (4,3.425) to[bend left=11] (-4,3.425);
    \draw[black!75, dashed] (12,3.425) to[bend right=11] (20,3.425);
    \draw[black!75, dashed] (8,4.52) to[bend left=8] (-8,4.52);
    \draw[black!75, dashed] (8,4.52) to[bend right=8] (24,4.52);
  \end{scope}
  \draw[innerarch] (0.5,0.425) to[bend right=28] (1.5,0.425);
  \draw[innerarch] (1.5,0.425) -- (1,1.325) -- (0.5,0.425);
  \draw[outerarch] (1.5,0.425) to[bend right=28] (2.5,0.425);
  \draw[outerarch] (1,1.325) to[bend right=22] (3,1.325);
  \draw[outerarch] (1.5,0.425) -- (1,1.325);
  \draw[outerarch] (2.5,0.425) -- (3,1.325);
  \draw[excluded] (1.5,0.425) -- (3,1.325);
  \draw[excluded] (2.5,0.425) -- (1,1.325);
  \foreach \i in {0,...,15}{\fill (\i+0.5,0.425) circle (1.3pt);}
  \foreach \i in {0,...,7}{\fill (2*\i+1,1.325) circle (1.3pt);}
  \foreach \i in {0,...,3}{\fill (4*\i+2,2.325) circle (1.3pt);}
  \foreach \i in {0,...,1}{\fill (8*\i+4,3.425) circle (1.3pt);}
  \node[below=2pt] at (0.5,0.425) {\scriptsize $\0$};
\end{tikzpicture}
\caption{Vertices are the dyadic windows, drawn as the tiles
$T_{k,i}$ of the binary tiling of the half-plane (light grey rectangles),
truncated at the floor $k=0$; the graph is drawn in black on top of the
tessellation, with an edge whenever two tiles share a boundary segment.
Horizontal edges, between siblings (same parent) and between cousins
(adjacent parents) alike, are drawn as slight arcs as a stylistic choice to distinguish them
from the tile borders; vertical (parent--child) edges are straight
segments. Dashed black arcs at the sides are edges to vertices outside
the field of view. The bounded faces are
triangles (a parent with its two children, shaded teal) and quadrilaterals
(two adjacent parents with the two children flanking a dyadic corner,
shaded blue); $\mathsf H^\st$ adds the two diagonals (dashed red) of every
quadrilateral face. The tiling is truncated below at the floor $k=0$ but
continues upward indefinitely, one tile per dyadic window at every scale
$k\in\N_0$ (dashed grey: the scale-$4$ tile). A site $(k,i)$ is declared
open iff $O_{k,i}$ occurs.}
\label{fig:tiling}
\end{figure}
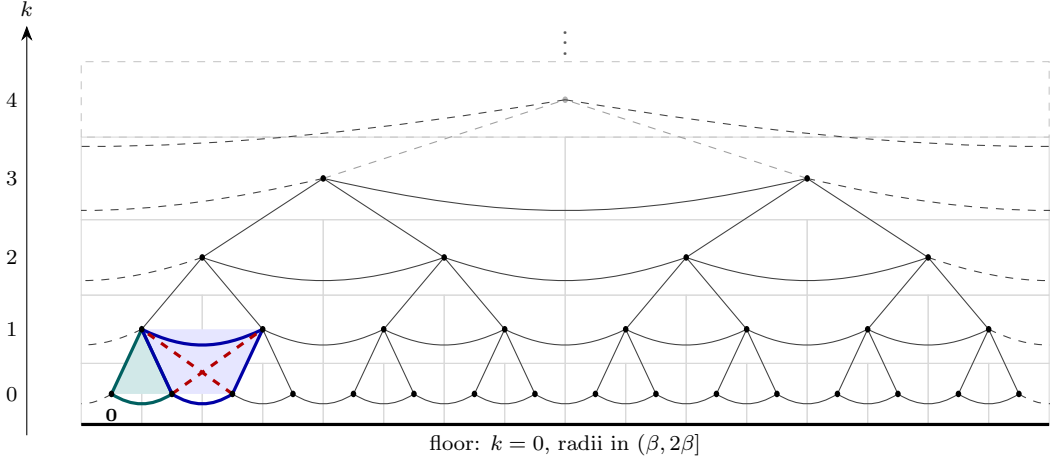

\begin{lemma}\label{lem:tiling}
$\mathsf H$ and $\mathsf H^\st$ are connected and locally finite, with
maximal degrees $5$ and $8$ respectively, and:
\begin{enumerate}
\item[(i)] $\mathsf H$ has exactly one end: for every finite $F\subseteq V$
  the graph $\mathsf H-F$ has exactly one infinite component;
\item[(ii)] $\mathsf H$ is planar, and the boundaries of its bounded faces
  are exactly the triangles $\{(k,2j),(k,2j+1),(k+1,j)\}$ and the
  quadrilaterals $\{(k,2j-1),(k,2j),(k+1,j),(k+1,j-1)\}$, $k\in\N_0$,
  $j\in\Z$; the edges of $\mathsf H^\st\setminus\mathsf H$ are exactly the
  diagonals of the quadrilateral faces.
\end{enumerate}
\end{lemma}

\begin{proof}
Degrees: a vertex $(k,i)$ with $k\ge1$ has two horizontal neighbours, one
parent and two children, so $\deg_{\mathsf H}\le5$; in $\mathsf H^\st$ it
gains exactly one upward diagonal (to $(k+1,\tfrac{i+1}2)$ for odd $i$, to
$(k+1,\tfrac i2-1)$ for even $i$) and two downward diagonals (to
$(k-1,2i-1)$ and $(k-1,2i+2)$), so $\deg_{\mathsf H^\st}\le8$. Vertices on
the floor $k=0$ have no children and no downward diagonals.

(ii) In the tiling, at least three tiles meet only at the tile corners on
the horocycles $y=2^{k+1}$, and these are of two types (the two kinds of
face are shaded in \Cref{fig:tiling}): at
$x=(2j+1)2^k$ (the midpoint of the lower edge of $T_{k+1,j}$) the three
tiles $T_{k,2j},T_{k,2j+1},T_{k+1,j}$ meet, giving a triangular face; at
$x=j2^{k+1}$ (a dyadic corner) the four tiles
$T_{k,2j-1},T_{k,2j},T_{k+1,j-1},T_{k+1,j}$ meet, giving a quadrilateral
face whose sides are the two horizontal edges and the two vertical edges
between them, its diagonals are precisely the two added edges of
$\mathsf H^\st$ at this corner. Corners on the floor line $y=1$ lie on the
unbounded face.

(i) Let $F\subseteq\{(k,i):k<K,\ \abs i<M\}$ with $M\ge1$. The set
$T:=\{(k,i):k\ge K\}$ is connected in $\mathsf H-F$: from any vertex climb
along vertical edges, and two vertices of $T$ have ancestors at a common
scale whose indices differ by at most one, where they join horizontally.
Any $(k,i)$ with $k<K$ and $\abs i\ge2M2^{K}$ reaches $T$ along its
ancestor path, whose indices have absolute value at least
$\abs i/2^{K}-1\ge2M-1\ge M$, hence avoid $F$. So all vertices outside the
finite set $F\cup\{(k,i):k<K,\abs i<2M2^K\}$ lie in one component of
$\mathsf H-F$, and every infinite component contains such vertices; hence
the infinite component is unique.
\end{proof}

By \Cref{lem:density,lem:link}, \Cref{thm:super} has been reduced to a
statement about independent site percolation on $\mathsf H$ at high
density. High density alone is not enough (\Cref{rem:tower} below
explains why the naive strategy of following one open site per scale
fails, and why the horizontal edges of $\mathsf H$ must carry part of the
load), so we argue by a Peierls contour bound. Its deterministic core is
the following lemma. Given $\omega\in\{0,1\}^{V}$, call a vertex
\emph{open} if $\omega_v=1$ and \emph{closed} otherwise; open clusters are
the connected components of the subgraph of $\mathsf H$ induced by the
open vertices. Write $\0:=(0,0)$.

\begin{lemma}\label{lem:contour}
Let $\omega\in\{0,1\}^V$, suppose $\0$ is open and its open cluster $C$ is
finite. Then there is a finite set $\Gamma\subseteq V$ with
$s:=\abs\Gamma\ge1$ such that:
\begin{enumerate}
\item[(i)] every vertex of $\Gamma$ is closed and $\mathsf H$-adjacent to
  $C$;
\item[(ii)] every infinite self-avoiding $\mathsf H$-path started at $\0$
  meets $\Gamma$;
\item[(iii)] $\Gamma$ is connected in $\mathsf H^\st$;
\item[(iv)] $(k^\ast,0)\in\Gamma$ for some $0\le k^\ast\le s-1$.
\end{enumerate}
\end{lemma}

Clause (iv) confines $\Gamma$ to the $s$ anchors $(0,0),\dots,(s-1,0)$, as
the count of \Cref{lem:count} requires.

\begin{proof}
Let $U$ be the set of vertices $w\notin C$ from which some path in
$\mathsf H$ to infinity avoids $C$; by \Cref{lem:tiling}(i), $U$ is the
unique infinite component of $\mathsf H-C$. Set
\[
   \Gamma:=\{w\in U:\ w\ \text{is $\mathsf H$-adjacent to}\ C\},
\]
the outer boundary of $C$ visible from infinity. A vertex of $\Gamma$ is
adjacent to the open cluster $C$ without belonging to it, hence closed;
this is (i). Every vertex of $\Gamma$ is $\mathsf H$-adjacent to $C$ and
$\deg_{\mathsf H}\le5$ by \Cref{lem:tiling}, so $\abs\Gamma\le5\abs C$ and
$\Gamma$ is finite; non-emptiness follows in the next paragraph.

We prove (ii) first and deduce (iv) from it. Let $P$ be any infinite
self-avoiding $\mathsf H$-path started at $\0$, let $w$ be the last vertex of $P$ lying
in $C$ and $w'$ its successor on $P$. The tail of $P$ from $w'$ avoids $C$,
so $w'\in U$; and $w'$ is adjacent to $w\in C$, so $w'\in\Gamma$, which is
(ii); in particular $\Gamma\ne\emptyset$. The vertical ray
$(0,0),(1,0),(2,0),\dots$ therefore meets $\Gamma$ in a vertex $(k^\ast,0)$
of the column $i=0$, and the floor ray $(0,0),(0,-1),(0,-2),\dots$ meets
$\Gamma$ in a vertex of height $0$. Every edge of $\mathsf H^\st$ changes
the height $k$ by at most one, so once (iii) is proved these two vertices
are joined inside $\Gamma$ by a path of at most $s-1$ edges, and
$k^\ast\le s-1$, which is (iv).

Clause (iii) is an application of a boundary-connectivity
lemma of Tim\'ar \cite[Lemma~2]{Timar2013}: if $G^{+}$ is a locally finite
graph containing $G$, and the cycle space of $G$ admits a generating set
consisting of cycles that are chordal in $G^{+}$ (a cycle is
\emph{chordal} in $G^{+}$ if any two of its vertices are
$G^{+}$-adjacent), then for every connected subset $C$ of $G$ and every
$x\in(V(G)\cup\mathrm{Ends}(G))\setminus C$, the outer boundary of $C$
visible from $x$ (the set of vertices outside $C$ that are
$G$-adjacent to $C$ and joined to $x$ by a $G$-path avoiding $C$) is
connected in $G^{+}$. We apply this with
$(G,G^{+})=(\mathsf H,\mathsf H^\st)$: the open cluster $C$ is connected
in $\mathsf H$, and for $x$ we take the unique end of $\mathsf H$
(\Cref{lem:tiling}(i)), for which visibility from $x$ coincides with
membership in $U$, so the visible boundary is $\Gamma$.

It remains to exhibit the generating set. Since $\mathsf H$ is planar and
locally finite, every cycle encloses a bounded region containing finitely
many bounded faces and equals the modulo-$2$ sum of their boundaries, so
the bounded face boundaries generate the cycle space of $\mathsf H$; by
\Cref{lem:tiling}(ii) these are triangles, which are chordal, and
quadrilaterals, chordal in $\mathsf H^\st$ because $\mathsf H^\st$
contains both diagonals.
\end{proof}

\begin{lemma}\label{lem:count}
For every $v\in V$ and $s\ge1$, the number of sets $\Gamma\subseteq V$ with
$v\in\Gamma$, $\abs\Gamma=s$ and $\Gamma$ connected in $\mathsf H^\st$ is
at most $64^{s-1}$.
\end{lemma}

\begin{proof}
Each such $\Gamma$ carries a spanning tree of the subgraph of
$\mathsf H^\st$ induced on it, and a depth-first traversal of that tree is
a closed walk from $v$ of length $2(s-1)$ in $\mathsf H^\st$ visiting
exactly the vertices of $\Gamma$. The walk determines $\Gamma$, and the
number of walks of length $2(s-1)$ from $v$ is at most
$\Delta^{2(s-1)}\le8^{2(s-1)}=64^{s-1}$ by \Cref{lem:tiling}.
\end{proof}

\begin{prop}\label{prop:peierls}
Let the vertices of $\mathsf H$ be open independently, each with
probability at least $1-\eps$, where $\eps\le2^{-7}$. Then
\[
   \P\big(\0\ \text{belongs to an infinite open cluster}\big)\ \ge\
   1-5\eps\ \ge\ \tfrac{123}{128} .
\]
\end{prop}

\begin{proof}
If $\0$ is open with finite open cluster, \Cref{lem:contour} produces a set
$\Gamma$ of some size $s\ge1$, all of whose vertices are closed, connected
in $\mathsf H^\st$ and containing a vertex $(k^\ast,0)$ with
$k^\ast\le s-1$. For a fixed candidate set of size $s$ the probability that
all its vertices are closed is at most $\eps^s$, by independence. Summing
over the at most $s$ admissible anchors $(k^\ast,0)$ and, by
\Cref{lem:count}, over the at most $64^{s-1}$ candidate sets through each
anchor,
\[
   \P\big(\0\ \text{open, open cluster finite}\big)\ \le\
   \sum_{s\ge1}s 64^{s-1}\eps^{s}\ =\ \frac{\eps}{(1-64\eps)^{2}}
   \ \le\ 4\eps,
\]
using $64\eps\le\tfrac12$. Hence the probability in question is at least
$1-\eps-4\eps$.
\end{proof}

\subsection{Proof of \texorpdfstring{\Cref{thm:super}}{Theorem 2}}\label{sec:superproof}

\begin{proof}[Proof of \Cref{thm:super}]
By the scaling isomorphism we fix $\lambda=1$ and $\beta\ge31$. Declare the
site $(k,i)$ of $\mathsf H$ open iff $O_{k,i}$ occurs. By
\Cref{lem:density} the sites are open independently, each with probability
at least $1-\eps(\beta)$ and $\eps(\beta)\le2^{-7}$, so by
\Cref{prop:peierls}, with probability at least $1-5\eps(\beta)>0$ the site
$\0$ belongs to an infinite open cluster $K\subseteq V$. On this event, for
any two sites of $K$ there is a finite path in $K$ joining them, and
\Cref{lem:link}, applied to each (horizontal or vertical) edge of that
path, shows that all the sets $P_{k,i}$, $(k,i)\in K$, lie in a single
connected component of the graph. These sets are nonempty and pairwise
disjoint (distinct sites use disjoint bands or disjoint windows) so
that component contains infinitely many points. Hence an infinite component
exists with positive probability. The event that the graph contains an
infinite component is invariant under the position shifts
$(x,t)\mapsto(x+c,t)$, $c\in\R$, under which the law of $\X$ is ergodic
(see \cite{LastPenrose2017}); its probability is therefore $0$ or $1$, and
being positive it equals $1$.
\end{proof}

A variant of the construction proves the supercritical part of
\Cref{thm:adrcm} for profiles other than the indicator.

\begin{corollary}\label{cor:adrcmclass}
Fix $\gamma\in(0,1)$ and a profile function $\rho$ with
$\rho\ge c\1_{[0,1]}$ for some $c\in(0,1]$. In $d=1$ the
weight-dependent random connection model with kernel
$(t\vee s)^{1-\gamma}(t\wedge s)^{\gamma}$ and profile $\rho$ almost
surely contains an infinite connected component once
$\lambda\beta\ge\beta_0(c)$.
\end{corollary}

Every non-trivial profile satisfies the hypothesis after a rescaling of
$\beta$. The map $(x,t)\mapsto(\lambda x,t)$ sends the model at
intensity $\lambda$ and parameter $\beta$ to the model at intensity one
and parameter $\lambda\beta$, for every kernel and profile, so we take
$\lambda=1$ and prove percolation for $\beta\ge\beta_0(c)$. For $c=1$
the graph contains the model with indicator profile and the same
kernel; the kernel is pointwise below $t\vee s$, so this in turn
contains the graph of \Cref{thm:super}, an infinite component exists
for $\beta\ge31$, and $\beta_0(1)=31$. Assume from now on $c\in(0,1)$.

Realise the edges of the graph $G$ of the corollary by a family
$(U_e)$ of independent uniform random variables on $(0,1)$, indexed by
the pairs $e$ of points of $\X$ and independent of $\X$: the pair
$e=\{(x,t),(y,s)\}$ is an edge of $G$ when
$U_e\le\rho\big(\beta^{-1}(t\vee s)^{1-\gamma}(t\wedge s)^{\gamma}
\abs{x-y}\big)$. Call a pair $e\in E$ a \emph{candidate} and a
candidate with $U_e\le c$ \emph{retained}. Since
$(t\vee s)^{1-\gamma}(t\wedge s)^{\gamma}\le t\vee s$ and
$\rho\ge c\1_{[0,1]}$, every retained candidate is an edge of $G$, so
it suffices to find an infinite component in the graph $G'$ of
retained candidates, in which every pair of $E$ is kept independently
with probability $c$.

Windows, cells, bands and the tiling $\mathsf H$ are as before. Set
$N:=\lceil\beta/8\rceil$, define
\[
   O^N_{k,i}:=\big\{\text{every cell of }\Lambda_{k,i}\text{ contains
   the positions of at least }N\text{ band-}k\text{ points}\big\},
\]
and on $O^N_{k,i}$ let $Q_{k,i}$, the \emph{designated points} of
$\Lambda_{k,i}$, consist of the $N$ leftmost band-$k$ points of each
cell. Two band-$k$ points in the same cell or in adjacent cells lie
within $\beta2^k$ of each other, so by \Cref{lem:bandedge} every such
pair is a candidate. Let the \emph{window event} $W_{k,i}$ be the
intersection of $O^N_{k,i}$ with the events that, for each of the
sixteen cells, the retained candidates connect the designated points
of the cell, and that each of the fifteen pairs of consecutive cells
is joined by a retained candidate between designated points; on
$W_{k,i}$ the whole of $Q_{k,i}$ lies in one connected component of
$G'$.

Each edge $e=\{u,v\}$ of $\mathsf H$ is assigned $N^2$ pairs and a
\emph{link event}. For a horizontal edge $\{(k,i),(k,i+1)\}$ take the
pairs formed by the designated points of the last cell of
$\Lambda_{k,i}$ and those of the first cell of $\Lambda_{k,i+1}$;
these cells are adjacent, so the pairs are candidates. For a vertical
edge $\{(k,i),(k+1,j)\}$ recall from the proof of \Cref{lem:link}(ii)
that the interval $J=[a+\beta2^k,a+2\beta2^k)$, with $a$ the left
endpoint of $\Lambda_{k,i}$, is a cell of $\Lambda_{k+1,j}$; it is
also the union of the third and fourth cells of $\Lambda_{k,i}$. Take
the pairs formed by the designated points of the third cell of
$\Lambda_{k,i}$ and the designated points of the cell $J$ of
$\Lambda_{k+1,j}$: the positions of such a pair both lie in $J$, hence
within $\beta2^k$ of each other, both radii exceed $\beta2^k$, and the
pairs are candidates by \Cref{lem:bandedge}. In both cases, writing
$O^N_u:=O^N_{k,i}$ (and similarly $W_u:=W_{k,i}$ later on) for the vertex $u=(k,i)$ of
$\mathsf H$, set
\[
   L_e:=O^N_u\cap O^N_v\cap\big\{\text{at least one of the $N^2$ pairs
   of $e$ is retained}\big\}.
\]
The window event and both assignments are illustrated in
\Cref{fig:thinned}.

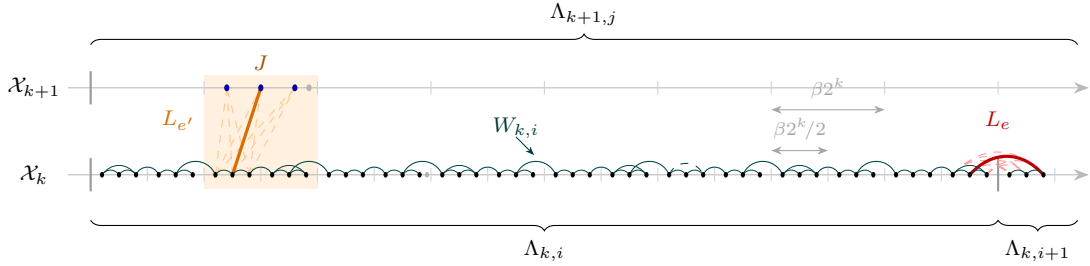
\begin{figure}[ht]\centering
\color{black}
\begin{tikzpicture}[xscale=0.75]
  \fill[orange!12] (2,-0.18) rectangle (4,1.32);
  \node[orange!60!black] at (3.0,1.48) {\footnotesize $J$};
  \draw[zline] (-0.4,0) -- (17.6,0);
  \draw[gray!80, thick] (0,-0.22) -- (0,0.22);
  \draw[gray!80, thick] (16,-0.22) -- (16,0.22);
  \foreach \x in {1,...,15,17}{\draw[gray!35] (\x,-0.09) -- (\x,0.09);}
  \draw[zline] (-0.4,1.15) -- (17.6,1.15);
  \draw[gray!80, thick] (0,0.93) -- (0,1.37);
  \foreach \x in {2,4,...,16}{\draw[gray!35] (\x,1.06) -- (\x,1.24);}
  \foreach \a in {2.2,2.5,2.8}{\foreach \b in {2.4,3.0,3.6}{
    \draw[orange!45, thin, dashed] (\a,0) -- (\b,1.15);}}
  \draw[escape] (2.5,0) -- (3.0,1.15);
  \foreach \a/\b in {15.2/16.2,15.2/16.5,15.2/16.8,15.5/16.2,
                     15.5/16.5,15.8/16.5,15.8/16.8}{
    \draw[red!40, thin, dashed] (\a,0) to[bend left=40] (\b,0);}
  \draw[red!40, thin, dashed] (15.8,0) to[bend left=60] (16.2,0);
  \draw[forced] (15.5,0) to[bend left=40] (16.8,0);
  \foreach \m in {0,...,16}{
    \draw[teal!55!black, very thin] (\m+0.2,0) to[bend left=45] (\m+0.5,0);
    \draw[teal!55!black, very thin] (\m+0.5,0) to[bend left=45] (\m+0.8,0);}
  \foreach \m in {0,3,6,9,12,15}{
    \draw[teal!55!black, very thin] (\m+0.2,0) to[bend left=45] (\m+0.8,0);}
  \draw[teal!55!black, thin, dashed] (10.2,0) to[bend left=60] (10.8,0);
  \foreach \m in {0,...,14}{
    \pgfmathsetmacro\pa{\m+0.8-0.3*mod(\m,2)}
    \draw[teal!55!black, thin] (\pa,0) to[bend left=60] (\m+1.2,0);}
  \foreach \m in {0,...,16}{
    \fill[black] (\m+0.2,0) circle (1.05pt);
    \fill[black] (\m+0.5,0) circle (1.05pt);
    \fill[black] (\m+0.8,0) circle (1.05pt);}
  \fill[gray!60] (5.93,0) circle (1.05pt);
  \foreach \b in {2.4,3.0,3.6}{\fill[blue!65!black] (\b,1.15) circle (1.3pt);}
  \fill[gray!60] (3.85,1.15) circle (1.3pt);
  \draw[gray!70, {Stealth[length=1.4mm]}-{Stealth[length=1.4mm]}]
     (12,0.32) -- (13,0.32) node[midway,above=0.5pt]{\scriptsize $\beta2^k\!/2$};
  \draw[gray!70, {Stealth[length=1.4mm]}-{Stealth[length=1.4mm]}]
     (12,0.86) -- (14,0.86) node[midway,above=0.5pt]{\scriptsize $\beta2^k$};
  \node[black] at (-1.0,0) {\footnotesize $\X_k$};
  \node[black] at (-1.0,1.15) {\footnotesize $\X_{k+1}$};
  \node[teal!45!black] at (7.5,0.62) {\footnotesize $W_{k,i}$};
  \draw[teal!45!black, thin, -{Stealth[length=1.2mm]}]
     (7.5,0.5) -- (7.82,0.26);
  \node[red!75!black] at (16.0,0.72) {\footnotesize $L_{e}$};
  \node[orange!85!black] at (1.55,0.72) {\footnotesize $L_{e'}$};
  \draw[brc] (0,-0.6) -- (16,-0.6) node[midway,below=4pt]{\footnotesize $\Lambda_{k,i}$};
  \draw[brc] (16,-0.6) -- (17.4,-0.6) node[midway,below=4pt]{\footnotesize $\Lambda_{k,i+1}$};
  \draw[decorate, decoration={brace, amplitude=4pt}] (0,1.72) -- (17.4,1.72)
     node[midway,above=4pt]{\footnotesize $\Lambda_{k+1,j}$};
\end{tikzpicture}
\caption{ The window event $W_{k,i}$ and the link events, drawn for
$N=3$. On the lower line the sixteen cells of $\Lambda_{k,i}$ and the
first cell of $\Lambda_{k,i+1}$ each hold the positions of their $N$
designated band-$k$ points (black); a cell may contain further
band-$k$ points (gray), which are not designated. On $W_{k,i}$ the
retained candidates connect the designated points of each cell and
join each of the fifteen pairs of consecutive cells (teal); one lost
candidate is shown dashed, harmless since its cell remains connected
without it. The link event of the horizontal edge
$e=\{(k,i),(k,i+1)\}$ asks that of the $N^2$ pairs between the last
cell of $\Lambda_{k,i}$ and the first cell of $\Lambda_{k,i+1}$
(red, lost candidates dashed) at least one is retained (solid). The
link event of the vertical edge $e'=\{(k,i),(k+1,j)\}$ asks the same
of the $N^2$ pairs between the third cell of $\Lambda_{k,i}$ and the
cell $J$ of $\Lambda_{k+1,j}$ (orange): $J$ is at once a cell of the
parent window and the union of the third and fourth cells of
$\Lambda_{k,i}$, so all these positions lie within $\beta2^k$ of each
other. Of the band-$(k+1)$ points only the designated points of $J$
are drawn. }
\label{fig:thinned}
\end{figure}

\begin{lemma}\label{lem:thinned}
Let $c\in(0,1)$ and
$\eps_1:=16e^{-\beta/32}+32N(1-c)^{N/2}+15(1-c)^{N^2}$.
\begin{enumerate}
\item[(i)] If $N(1-c)^{N/2}\le\tfrac12$, then $\P(W_{k,i})\ge1-\eps_1$
  for all $k,i$.
\item[(ii)] $\P\big(O^N_u\cap O^N_v\setminus L_e\big)\le(1-c)^{N^2}$
  for every edge $e=\{u,v\}$ of $\mathsf H$.
\item[(iii)] No pair of points is assigned twice: not to two distinct
  window events, not to two distinct edges of $\mathsf H$, and not to
  a window event and an edge.
\end{enumerate}
\end{lemma}

\begin{proof}
(i) As in \Cref{lem:density}, the number of band-$k$ points in a fixed
cell is Poisson with mean $\mu=\beta/4$, independently over the
sixteen cells. For $Z$ Poisson with mean $\mu$, Markov's inequality
applied to $2^{-Z}$ gives
\[
   \P\big(Z\le\tfrac{\mu}{2}\big)\ \le\ 2^{\mu/2}\,\E\big[2^{-Z}\big]
   \ =\ e^{-\mu(1-\ln2)/2}\ \le\ e^{-\mu/8},
\]
and $N-1<\beta/8=\mu/2$, so a fixed cell contains fewer than $N$
band-$k$ points with probability at most $e^{-\beta/32}$; hence
$\P(O^N_{k,i})\ge1-16e^{-\beta/32}$. Conditionally on $\X$, on
$O^N_{k,i}$, the retained candidates between the designated points of
a fixed cell form the binomial random graph on $N$ vertices with edge
probability $c$. If that graph is disconnected, some set of $j\le N/2$
of its vertices is joined to none of the other $N-j$, an event of
probability at most
$\binom Nj(1-c)^{j(N-j)}\le\big(N(1-c)^{N/2}\big)^{j}$ for each fixed
$j$, so the graph is disconnected with probability at most
$\sum_{j\ge1}\big(N(1-c)^{N/2}\big)^{j}\le2N(1-c)^{N/2}$ when
$N(1-c)^{N/2}\le\tfrac12$. A fixed pair of consecutive cells fails to
be joined when all $N^2$ candidates between their designated points
are lost, an event of probability $(1-c)^{N^2}$. The union bound over
the sixteen cells and the fifteen consecutive pairs proves (i).

(ii) On $O^N_u\cap O^N_v$ the $N^2$ pairs of $e$ are defined, and all
of them are lost with probability $(1-c)^{N^2}$, the uniforms being
independent of $\X$.

(iii) A pair of a window event joins two band-$k$ points of one
window, and distinct windows have disjoint point sets (disjoint
intervals at equal scale, disjoint bands at distinct scales), so the
window determines the event. A pair of a horizontal edge joins
band-$k$ points of two distinct windows, so it coincides with no
window pair, and the two windows determine the edge. A pair of a
vertical edge joins points of two consecutive bands, so it coincides
with no window pair and no horizontal pair; the lower band determines
$k$, and the window of its band-$k$ endpoint is the child window,
which determines the edge.
\end{proof}

\begin{proof}[Proof of \Cref{cor:adrcmclass}]
Call a vertex $u$ of $\mathsf H$ \emph{good} if $W_u$ holds together
with $L_e$ for every edge $e$ of $\mathsf H$ incident to $u$. Since
$W_u\subseteq O^N_u$ and $\deg_{\mathsf H}\le5$ (\Cref{lem:tiling}),
\Cref{lem:thinned}(i),(ii) and the union bound give
\[
   \P(u\ \text{good})\ \ge\ 1-\eps_2,\qquad
   \eps_2:=\eps_1+5\big(16e^{-\beta/32}+(1-c)^{N^2}\big).
\]
If the good vertices contain an infinite cluster in $\mathsf H$, then
$G'$ has an infinite component: for every edge of $\mathsf H$ between
good vertices the link event produces a retained candidate joining a
designated point of one window to a designated point of the other,
the window events connect the designated points within each window,
and the sets $Q_u$ are nonempty and pairwise disjoint over the
cluster.

The good-vertex events are not independent, but their dependence has
range two. The event $\{u\ \text{good}\}$ is measurable with respect
to the restriction of $\X$ to the regions
$\Lambda_{k,i}\times[2^{-k-1},2^{-k})$ of $u$ and of its neighbours,
together with the uniforms of the pairs assigned to $W_u$ and to the
edges at $u$. If $u$ and $v$ are at graph distance at least $3$ in
$\mathsf H$, no vertex is a neighbour of both and no edge is incident
to both, so the two collections of regions are disjoint and, by
\Cref{lem:thinned}(iii), so are the two collections of pairs.
Conditionally on $\X$ the two events are then independent, and their
conditional probabilities are functions of the restrictions of $\X$
to disjoint regions, hence the events are independent; the same holds
for every family with pairwise distances at least $3$.

Suppose $\0$ is good and its good cluster is finite.
\Cref{lem:contour}, applied to the good-vertex configuration,
produces a set $\Gamma$ of $s\ge1$ vertices, none good, connected in
$\mathsf H^\st$ and containing an anchor $(k^\ast,0)$ with
$k^\ast\le s-1$. A ball of radius $2$ in $\mathsf H$ contains at most
$1+5+5\cdot4=26$ vertices, so repeatedly selecting a vertex of
$\Gamma$ and discarding the at most $25$ others within distance $2$
of it leaves a set $\Gamma'\subseteq\Gamma$ of at least $s/26$
vertices with pairwise distances at least $3$, and
\[
   \P\big(\text{no vertex of }\Gamma\text{ is good}\big)\ \le\
   \prod_{u\in\Gamma'}\P(u\ \text{not good})\ \le\ \eps_2^{\,s/26}.
\]
Summing over the at most $s$ anchors and, by \Cref{lem:count}, the at
most $64^{s-1}$ candidate sets through each anchor,
\[
   \P\big(\0\ \text{good, good cluster finite}\big)\ \le\
   \sum_{s\ge1}s\,64^{s-1}\eps_2^{\,s/26}\ \le\ 4\eps_2^{\,1/26}
\]
once $64\eps_2^{1/26}\le\tfrac12$. Since $N=\lceil\beta/8\rceil$, for
fixed $c$ every term of $\eps_1$ and $\eps_2$ tends to $0$ as
$\beta\to\infty$, so there is $\beta_0(c)<\infty$ such that
$\beta\ge\beta_0(c)$ implies $N(1-c)^{N/2}\le\tfrac12$,
$64\eps_2^{1/26}\le\tfrac12$ and $\eps_2+4\eps_2^{1/26}<1$. For such
$\beta$ the origin belongs to an infinite good cluster with positive
probability, so $G'$, and with it $G$, contains an infinite component
with positive probability. The existence of an infinite component in
$G$ is invariant under the position shifts $(x,t)\mapsto(x+r,t)$,
$r\in\R$; since $G$ is a measurable function of $\X$ independently
marked by the edge variables, whose law is ergodic under these shifts
by the marking theorem and the ergodicity of $\X$ (see
\cite{LastPenrose2017}), its probability is $0$ or $1$, and being
positive it equals $1$.
\end{proof}

\begin{remark}\label{rem:tower}
The spatial branching built into \Cref{prop:peierls} is forced. For the
vertical linkage a scale-$k$ window must have length of order $\beta2^k$,
and any event guaranteeing a spanning run in such a window fails with
probability at least $c(\beta)>0$ \emph{uniformly in $k$}: by
\Cref{rem:copies} the window length, the point spacings and the connection
range all carry the same factor $2^k$, so nothing improves with the scale.
For a fixed nested sequence of windows, one per scale, the spanning events
are independent (they use disjoint bands), so they all hold simultaneously
with probability $\prod_k(1-c(\beta))=0$: a single tower of scales dies
almost surely. The infinite cluster must be able to route around a broken
window through siblings at the same scale, which the horizontal edges of
$\mathsf H$ provide; the anchored-contour bound of
\Cref{lem:contour} is the corresponding replacement for a branching
argument along the (unique-parent) vertical tree.
\end{remark}

\begin{remark}\label{rem:superlit}
\begin{enumerate}[(i)]
\item \Cref{thm:main} and \Cref{thm:super} are two sides of one mechanism.
The construction of this section needs, at every scale, a band that is
locally dense (many points per unit of its own connection range). In the
continuum, band $k$ carries $\lambda\beta/2$ expected points per unit of
connection range \emph{at every scale} $k$, so the single parameter
$\lambda\beta$ makes all bands dense simultaneously, and percolation
follows once it is large. On the lattice this is impossible: band $k$ of
the simplified model (the vertices with $R_v\in(x_m2^k,x_m2^{k+1}]$)
contains on average $x_m/2<\tfrac12$ points per unit of connection range,
for every $k$ and every admissible $x_m$ (the unit vertex spacing caps
the density that the radius law can produce, uniformly over scales). No
band is ever dense at its own scale, the renormalisation of this section
has nothing to feed on, and indeed by \Cref{thm:main} the lattice model
never percolates. In this precise sense $\lambda\beta$ is the density the
lattice cannot exceed, and the supercritical phase is a purely
continuum phenomenon. 
\item The constant $31$ is what the crude bounds
give ($16e^{-\lambda\beta/4}\le2^{-7}$ against the walk-counting constant
$64$ of \Cref{lem:count}); no attempt has been made to optimise it, and
the true $\beta_c$ at intensity one is unknown beyond
$\beta_c\in[1,31]$. The simulations of \Cref{sec:numerics} place it near
$2$; they also give $p_c(\mathsf H)\approx0.70$ for the tiling of
\Cref{sec:tiling}, so optimising the constants in the present proof could
lower $31$ to about $15$ and not beyond.
\end{enumerate}
\end{remark}

\paragraph{Acknowledgements.}
The authors thank Lukas L\"uchtrath for reminding them of some
immediate consequences of their results, now included in the paper.

\appendix
\crefalias{section}{appendix}
\Crefname{appendix}{Appendix}{Appendices}

\section{A numerical study of the critical value}\label{sec:numerics}

This appendix summarises a numerical study of the critical value at
intensity one. The results place $\beta_c$ close to $2$, and all of them
are compatible with $\beta_c\in[1.5,2.5]$; the rigorous bounds are
$1\le\beta_c\le31$ (\Cref{thm:continuum,thm:super}). The Julia code, the
seeds and the run manifests of the runs reported here are included in the
ancillary files of the arXiv submission of this paper.

\paragraph{Setup.} We simulate the model of \Cref{sec:setup} at $\lambda=1$
on a window $[0,L]$, in the radius formulation of \Cref{sec:rainbows}: the
positions form a Poisson process of intensity one, each point carries an
independent mark $t$ uniform on $(0,1)$ and the radius $R=\beta/t$, and two
points are joined iff $\abs{x-y}\le R_x\wedge R_y$. Marks are stored rather
than radii, so a single sample realises the graphs for all $\beta$ at once,
and these graphs increase in $\beta$. Components are computed by a
union-find pass over the points in decreasing order of radius: when a point
is processed, its edges to all previously processed points reach exactly as
far as its own radius, so every edge is discovered once, at its endpoint of
smaller radius.

\paragraph{Estimators.} Two estimators are used. First, on a fixed sample
the event that a single component covers the centred subwindow of length
$L/2$ is increasing in $\beta$ and therefore has a per-sample threshold,
which bisection locates in about thirteen component passes; the median of
this threshold over replicates is the value of $\beta$ at which the
spanning probability equals $\tfrac12$. Second, a variant of the algorithm
of Newman and Ziff \cite{NewmanZiff2001} replays the edges of a sample in
increasing order of activation value $(t_x\vee t_y)\abs{x-y}$ and produces
in one pass the whole curve $\theta_L(\beta)$, the fraction of subwindow
points lying in the largest component, of which we record the median over
replicates. Window lengths run from $L=10^4$ to $L=10^9$, with replicate
counts decreasing from $1024$ to $12$. Since the radius tail exponent is
exactly $1$, every finite-size effect decays like $1/\log L$ rather than
like a power of $L$; all extrapolations are therefore taken in the variable
$1/\log L$.

\paragraph{Controls.} Since all finite-size effects here are logarithmic,
the procedure is checked against the three statements the paper proves. At
$\beta=0.5$, subcritical by \Cref{thm:continuum}, every observable reads
subcritical at every $L$. At $\beta=31$, supercritical by
\Cref{thm:super}, every window is fully connected: every radius exceeds $31$, so consecutive points at distance at most $31$ are always joined, and point-free stretches of length $31$ have intensity $e^{-31}$ per unit length, so none occurs at these window sizes. The percolation density at $\beta=31$ is below one by an amount of the same order, so full connectivity is the expected reading. The third check is the
simplified model of \Cref{sec:simplified}, which spans finite windows for
$x_m$ close to one but never percolates (\Cref{thm:main}): run on it, the
procedure extrapolates the spanning threshold to $x_m=1$ and produces no
threshold in the interior of $(0,1)$ (\Cref{fig:threshold-drift}, right
panel).

\begin{figure}[!ht]
\centering
\includegraphics[width=\textwidth]{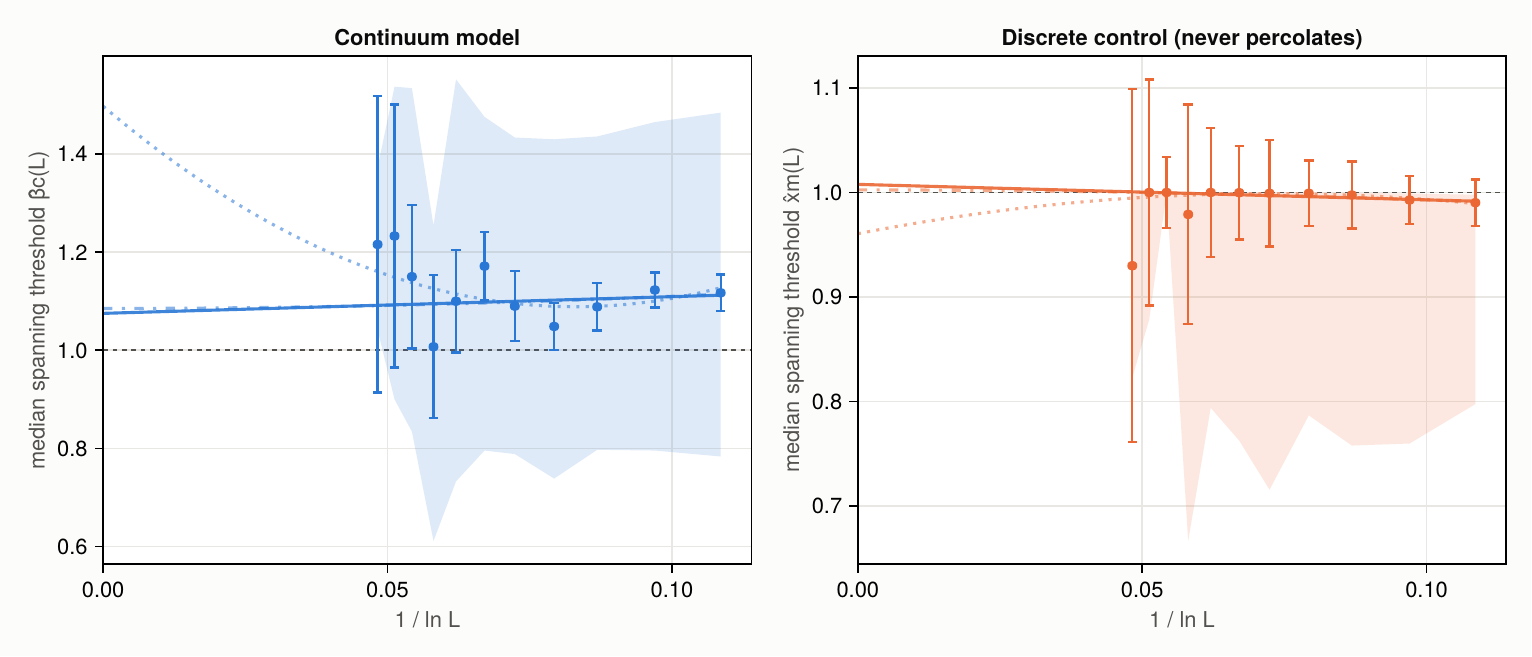}
\caption{Median per-sample spanning threshold against $1/\log L$, with
interquartile bands and three extrapolating fits, for $L=10^4$ to $10^9$.
Left: the continuum model; the threshold extrapolates to about $1.08$,
below $\beta_c$, since finite windows are spanned below criticality by
single points of very large radius. Right: the simplified model on $\Z$,
which spans finite windows but never percolates (\Cref{thm:main}); the
same procedure extrapolates its threshold to $x_m=1$ and produces no
threshold in the interior of $(0,1)$.}
\label{fig:threshold-drift}
\end{figure}

\paragraph{Results.} \Cref{fig:theta-curves} shows $\theta_L(\beta)$. For
$\beta\le1.5$ the curves decay with $L$ at an accelerating rate, while at
$\beta=2$ they are stable at $\theta\approx0.94$ and at $\beta=3$ at
$\theta\approx0.9985$; the crossings of consecutive curves lie in
$[1.7,2.4]$. In the same direction, at $\beta=1.5$ windows of length
$10^9$ still contain interior stretches connected to neither end of the
window, while at $\beta=2$ such stretches no longer occur. Together this
places $\beta_c$ close to $2$. The drift of the near-critical curves is
logarithmic and has not stopped completely at the largest sizes, so a
critical value slightly above $2.5$ is not excluded.

\begin{figure}[!ht]
\centering
\includegraphics[width=0.72\textwidth]{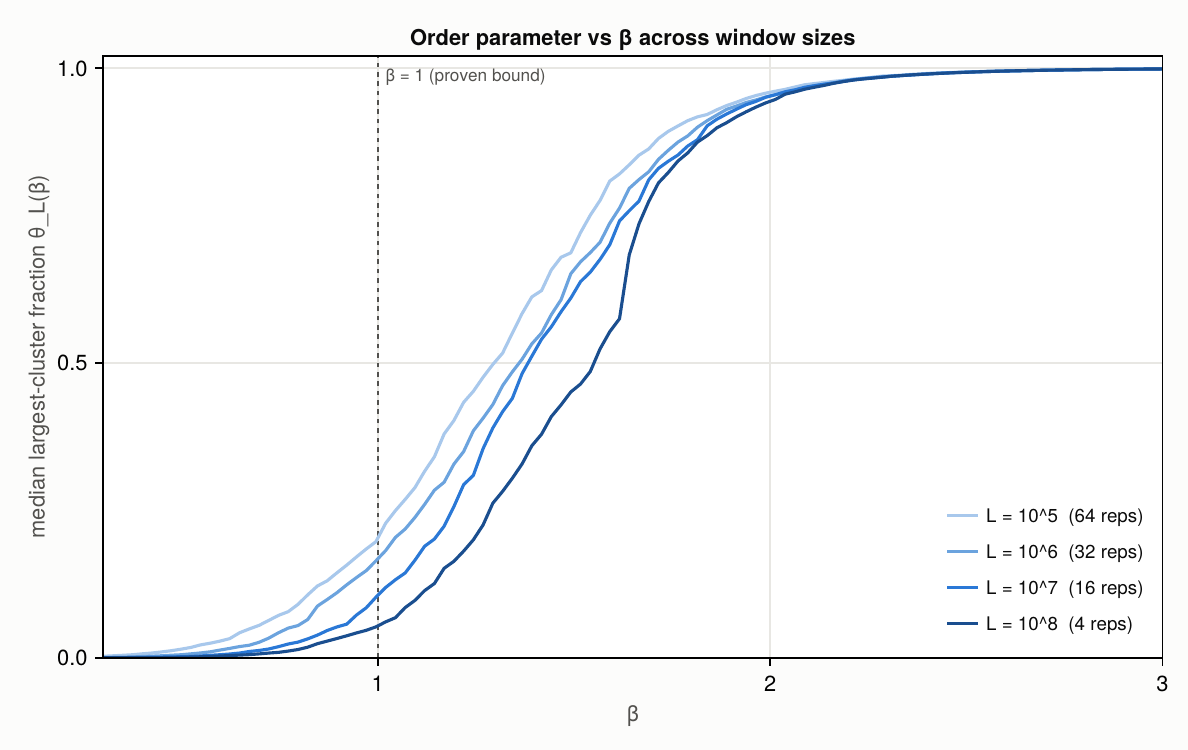}
\caption{The largest-component fraction $\theta_L(\beta)$ of the central
subwindow, median over replicates, for window lengths $L=10^5$ to $10^8$
(darker with increasing $L$; $64$ down to $4$ replicates). The curves
decay with $L$ for $\beta\le1.5$, are stable in $L$ from $\beta=2$ on, and
sharpen with $L$ while crossing at heights $0.7$ to $0.9$. The dashed line
marks the proven lower bound $\beta_c\ge1$.}
\label{fig:theta-curves}
\end{figure}

The spanning threshold itself does not converge to $\beta_c$: its medians
are close to $1.1$ for every $L$ and extrapolate to about $1.08$
(\Cref{fig:threshold-drift}, left panel). A finite window can be spanned
below criticality by a single point of very large radius, the mechanism of
\Cref{cor:diameter}, and the left tail that such points produce in the
distribution of the per-sample threshold does not shrink with $L$. The
estimate of $\beta_c$ therefore comes from the $\theta$ curves, which
measure the size of the largest component and not only its extent.

The curves of \Cref{fig:theta-curves} sharpen with $L$ while their
crossing heights stay between $0.7$ and $0.9$. For a continuously
vanishing percolation density the crossing heights would decrease with
$L$; stable heights are what a density with a jump would produce, and for
$\beta_c\approx2$ they give $\theta(\beta_c^{+})\approx0.9$. This is the
question raised in \Cref{rem:lrp}(v).

A separate simulation of site percolation on the floored binary tiling
$\mathsf H$ of \Cref{sec:tiling} gives $p_c(\mathsf H)\approx0.70$, while
the Peierls bound of \Cref{prop:peierls} asks for site density
$1-2^{-7}$. With the site density $(1-e^{-\beta/4})^{16}$ of
\Cref{lem:density}, the renormalisation of \Cref{sec:super} therefore
produces a percolating site configuration from $\beta\approx15$ onwards:
optimising the constants in the given proof could lower the $31$ of
\Cref{thm:super} to about $15$, and any bound below that requires a
different argument; compare \Cref{rem:superlit}(ii).

All runs are seeded from a single master seed with an independent stream
per task, and every result file is stored together with a manifest
recording the seed, the code version and the parameter grids. All figures
can be reproduced exactly from the ancillary files.

\printbibliography

@article{Harris1960,
  author  = {Harris, Theodore E.},
  title   = {A lower bound for the critical probability in a certain percolation process},
  journal = {Proceedings of the Cambridge Philosophical Society},
  volume  = {56},
  number  = {1},
  year    = {1960},
  pages   = {13--20},
  doi     = {10.1017/S0305004100034241}
}

@article{KochenStone1964,
  author  = {Kochen, Simon and Stone, Charles},
  title   = {A note on the {B}orel--{C}antelli lemma},
  journal = {Illinois Journal of Mathematics},
  volume  = {8},
  number  = {2},
  year    = {1964},
  pages   = {248--251},
  doi     = {10.1215/ijm/1256059668}
}

@article{Schulman1983,
  author  = {Schulman, Lawrence S.},
  title   = {Long range percolation in one dimension},
  journal = {Journal of Physics A: Mathematical and General},
  volume  = {16},
  number  = {17},
  year    = {1983},
  pages   = {L639--L641},
  doi     = {10.1088/0305-4470/16/17/001}
}

@article{NewmanSchulman1986,
  author  = {Newman, Charles M. and Schulman, Lawrence S.},
  title   = {One dimensional {$1/|j-i|^s$} percolation models: {T}he existence of a transition for {$s\le2$}},
  journal = {Communications in Mathematical Physics},
  volume  = {104},
  year    = {1986},
  pages   = {547--571},
  doi     = {10.1007/BF01211064}
}

@article{AizenmanNewman1986,
  author  = {Aizenman, Michael and Newman, Charles M.},
  title   = {Discontinuity of the percolation density in one dimensional {$1/|x-y|^2$} percolation models},
  journal = {Communications in Mathematical Physics},
  volume  = {107},
  year    = {1986},
  pages   = {611--647},
  doi     = {10.1007/BF01205489}
}

@book{LastPenrose2017,
  author    = {Last, G{\"u}nter and Penrose, Mathew},
  title     = {Lectures on the {P}oisson Process},
  series    = {Institute of Mathematical Statistics Textbooks},
  volume    = {7},
  publisher = {Cambridge University Press},
  address   = {Cambridge},
  year      = {2017},
  doi       = {10.1017/9781316104477}
}

@book{MeesterRoy1996,
  author    = {Meester, Ronald and Roy, Rahul},
  title     = {Continuum Percolation},
  series    = {Cambridge Tracts in Mathematics},
  volume    = {119},
  publisher = {Cambridge University Press},
  address   = {Cambridge},
  year      = {1996},
  doi       = {10.1017/CBO9780511895357}
}

@article{GracarGrauerLuechtrathMoerters2019,
  author  = {Gracar, Peter and Grauer, Arne and L{\"u}chtrath, Lukas and M{\"o}rters, Peter},
  title   = {The age-dependent random connection model},
  journal = {Queueing Systems},
  volume  = {93},
  year    = {2019},
  pages   = {309--331},
  doi     = {10.1007/s11134-019-09625-y}
}

@article{GracarLuechtrathMoerters2021,
  author  = {Gracar, Peter and L{\"u}chtrath, Lukas and M{\"o}rters, Peter},
  title   = {Percolation phase transition in weight-dependent random connection models},
  journal = {Advances in Applied Probability},
  volume  = {53},
  number  = {4},
  year    = {2021},
  pages   = {1090--1114},
  doi     = {10.1017/apr.2021.13}
}

@article{GracarHeydenreichMoenchMoerters2022,
  author  = {Gracar, Peter and Heydenreich, Markus and M{\"o}nch, Christian and M{\"o}rters, Peter},
  title   = {Recurrence versus transience for weight-dependent random connection models},
  journal = {Electronic Journal of Probability},
  volume  = {27},
  year    = {2022},
  pages   = {1--31},
  note    = {Paper no.~60},
  doi     = {10.1214/22-EJP748}
}

@article{Timar2013,
  author  = {Tim{\'a}r, {\'A}d{\'a}m},
  title   = {Boundary-connectivity via graph theory},
  journal = {Proceedings of the American Mathematical Society},
  volume  = {141},
  number  = {2},
  year    = {2013},
  pages   = {475--480},
  doi     = {10.1090/S0002-9939-2012-11333-4}
}

@article{BenjaminiSchramm2001,
  author  = {Benjamini, Itai and Schramm, Oded},
  title   = {Percolation in the hyperbolic plane},
  journal = {Journal of the American Mathematical Society},
  volume  = {14},
  number  = {2},
  year    = {2001},
  pages   = {487--507},
  doi     = {10.1090/S0894-0347-00-00362-3}
}

@article{GracarLuechtrathMoench2022,
  author  = {Gracar, Peter and L{\"u}chtrath, Lukas and M{\"o}nch, Christian},
  title   = {Finiteness of the percolation threshold for inhomogeneous long-range models in one dimension},
  journal = {Electronic Journal of Probability},
  volume  = {30},
  year    = {2025},
  pages   = {1--29},
  note    = {Paper no.~134},
  doi     = {10.1214/25-EJP1399}
}

@article{Yukich2006,
  author  = {Yukich, Joseph E.},
  title   = {Ultra-small scale-free geometric networks},
  journal = {Journal of Applied Probability},
  volume  = {43},
  number  = {3},
  year    = {2006},
  pages   = {665--677},
  doi     = {10.1239/jap/1158784937}
}

@article{Berger2002,
  author  = {Berger, Noam},
  title   = {Transience, recurrence and critical behavior for long-range percolation},
  journal = {Communications in Mathematical Physics},
  volume  = {226},
  number  = {3},
  year    = {2002},
  pages   = {531--558},
  doi     = {10.1007/s002200200617}
}

@article{Biskup2004,
  author  = {Biskup, Marek},
  title   = {On the scaling of the chemical distance in long-range percolation models},
  journal = {The Annals of Probability},
  volume  = {32},
  number  = {4},
  year    = {2004},
  pages   = {2938--2977},
  doi     = {10.1214/009117904000000577}
}

@article{Biskup2011,
  author  = {Biskup, Marek},
  title   = {Graph diameter in long-range percolation},
  journal = {Random Structures \& Algorithms},
  volume  = {39},
  number  = {2},
  year    = {2011},
  pages   = {210--227},
  doi     = {10.1002/rsa.20349}
}

@article{BiskupLin2019,
  author  = {Biskup, Marek and Lin, Jeffrey},
  title   = {Sharp asymptotic for the chemical distance in long-range percolation},
  journal = {Random Structures \& Algorithms},
  volume  = {55},
  number  = {3},
  year    = {2019},
  pages   = {560--583},
  doi     = {10.1002/rsa.20849}
}

@article{ImbrieNewman1988,
  author  = {Imbrie, John Z. and Newman, Charles M.},
  title   = {An intermediate phase with slow decay of correlations in one dimensional {$1/|x-y|^2$} percolation, {I}sing and {P}otts models},
  journal = {Communications in Mathematical Physics},
  volume  = {118},
  number  = {2},
  year    = {1988},
  pages   = {303--336},
  doi     = {10.1007/BF01218582}
}

@article{AizenmanChayesChayesNewman1988,
  author  = {Aizenman, Michael and Chayes, Jennifer T. and Chayes, Lincoln and Newman, Charles M.},
  title   = {Discontinuity of the magnetization in one-dimensional {$1/|x-y|^2$} {I}sing and {P}otts models},
  journal = {Journal of Statistical Physics},
  volume  = {50},
  number  = {1--2},
  year    = {1988},
  pages   = {1--40},
  doi     = {10.1007/BF01022985}
}

@article{DuminilCopinGarbanTassion2024,
  author  = {Duminil-Copin, Hugo and Garban, Christophe and Tassion, Vincent},
  title   = {Long-range models in {1D} revisited},
  journal = {Annales de l'Institut Henri Poincar{\'e}, Probabilit{\'e}s et Statistiques},
  volume  = {60},
  number  = {1},
  year    = {2024},
  pages   = {232--241},
  doi     = {10.1214/22-AIHP1355}
}

@misc{DingSly2013,
  author      = {Ding, Jian and Sly, Allan},
  title       = {Distances in critical long range percolation},
  year        = {2013},
  eprint      = {1303.3995},
  eprinttype  = {arxiv},
  eprintclass = {math.PR}
}

@article{Baeumler2023,
  author  = {B{\"a}umler, Johannes},
  title   = {Distances in {$1/|x-y|^{2d}$} percolation models for all dimensions},
  journal = {Communications in Mathematical Physics},
  volume  = {404},
  number  = {3},
  year    = {2023},
  pages   = {1495--1570},
  doi     = {10.1007/s00220-023-04861-z}
}

@article{BaeumlerBerger2024,
  author  = {B{\"a}umler, Johannes and Berger, Noam},
  title   = {Isoperimetric lower bounds for critical exponents for long-range percolation},
  journal = {Annales de l'Institut Henri Poincar{\'e}, Probabilit{\'e}s et Statistiques},
  volume  = {60},
  number  = {1},
  year    = {2024},
  pages   = {721--730},
  doi     = {10.1214/22-AIHP1342}
}

@article{DeijfenHofstadHooghiemstra2013,
  author  = {Deijfen, Maria and van der Hofstad, Remco and Hooghiemstra, Gerard},
  title   = {Scale-free percolation},
  journal = {Annales de l'Institut Henri Poincar{\'e}, Probabilit{\'e}s et Statistiques},
  volume  = {49},
  number  = {3},
  year    = {2013},
  pages   = {817--838},
  doi     = {10.1214/12-AIHP480}
}

@article{Jorritsma2023a,
  author  = {Jorritsma, Joost and Komj{\'a}thy, J{\'u}lia and Mitsche, Dieter},
  title   = {Cluster-size decay in supercritical kernel-based spatial random graphs},
  journal = {The Annals of Probability},
  volume  = {53},
  number  = {4},
  year    = {2025},
  pages   = {1537--1597},
  doi     = {10.1214/24-AOP1742}
}

@article{Bringmann2019,
  author  = {Bringmann, Karl and Keusch, Ralph and Lengler, Johannes},
  title   = {Geometric inhomogeneous random graphs},
  journal = {Theoretical Computer Science},
  volume  = {760},
  year    = {2019},
  pages   = {35--54},
  doi     = {10.1016/j.tcs.2018.08.014}
}

@article{NewmanZiff2001,
  author  = {Newman, Mark E. J. and Ziff, Robert M.},
  title   = {Fast {M}onte {C}arlo algorithm for site or bond percolation},
  journal = {Physical Review E},
  volume  = {64},
  year    = {2001},
  pages   = {016706},
  doi     = {10.1103/PhysRevE.64.016706}
}

\end{document}